\documentclass[]{amsart}

\pdfoutput=1

\usepackage{amssymb}
\usepackage{amsfonts}
\usepackage{amsmath}
\usepackage{amsthm}
\usepackage{mathtools}
\usepackage{graphicx}
\usepackage{mathrsfs}
\usepackage{dsfont}
\usepackage{amscd}
\usepackage{multirow}

\usepackage[all]{xy}
\normalfont
\usepackage[T1]{fontenc}
\usepackage{calligra}
\usepackage{verbatim}
\usepackage{fourier}
\usepackage[usenames,dvipsnames]{xcolor}
\usepackage[colorlinks=true,linkcolor=Blue,citecolor=Violet]{hyperref}
\usepackage{enumerate}
\usepackage{slashed}
\usepackage{nicematrix}
\usepackage{microtype}
\usepackage{tikz}\usetikzlibrary{snakes,math}
\usepackage{caption}
\usepackage{subcaption}
\usepackage{faktor}

\newcommand{\N}{{\mathds{N}}}
\newcommand{\Z}{{\mathds{Z}}}

\newcommand{\R}{{\mathds{R}}}
\newcommand{\C}{{\mathds{C}}}
\newcommand{\T}{{\mathds{T}}}
\newcommand{\U}{{\mathcal{Z}}}
\newcommand{\D}{{\mathfrak{D}}}
\newcommand{\A}{{\mathfrak{A}}}
\newcommand{\B}{{\mathfrak{B}}}
\newcommand{\M}{{\mathfrak{M}}}

\newcommand{\bigslant}[2]{\faktor{#1}{#2}}
\newcommand{\Nbar}{\overline{\N}}

\newcommand{\Pbar}{\overline{\N}_\ast}
\newcommand{\Lip}[1][L]{{\mathsf{#1}}}
\newcommand{\TLip}{{\mathsf{T}}}

\newcommand{\Hilbert}[1][H]{{\mathscr{#1}}}

\newcommand{\gradient}{\operatorname{grad}}

\newcommand{\dpropinquity}[1]{{\mathsf{\Lambda}^\ast_{#1}}}

\newcommand{\spectralpropinquity}[1]{{\mathsf{\Lambda}^{\mathsf{spec}}_{#1}}}

\newcommand{\Kantorovich}[1]{{\mathsf{mk}_{#1}}}

\newcommand{\Haus}[1]{{\mathsf{Haus}_{{#1}}\,}}

\newcommand{\StateSpace}{{\mathscr{S}}}

\newcommand{\MongeKant}{{Mon\-ge-Kan\-to\-ro\-vich metric}}

\newcommand{\mcc}[3]{{\mathrm{metCor}\left({#1},{#2},{#3}\right)}}

\newcommand{\qcms}{quantum compact metric space}

\newcommand{\unit}{1}

\newcommand{\sa}[1]{{\mathfrak{sa}\left({#1}\right)}}

\newcommand{\inner}[3]{{\left<{#1},{#2}\right>_{#3}}}

\newcommand{\dom}[1]{{\operatorname*{dom}\left({#1}\right)}}

\newcommand{\codom}[1]{{\operatorname*{codom}\left({#1}\right)}}

\newcommand{\norm}[2]{\left\|{#1}\right\|_{#2}}

\newcommand{\targetsettunnel}[3]{{\mathfrak{t}_{#1}\left({#2}\middle\vert{#3}\right)}}

\newcommand{\CDN}{{\mathsf{DN}}}
\newcommand{\TDN}{{\mathsf{TN}}}

\newcommand{\worknote}[1]{}
\newcommand{\opnorm}[3]{{\left|\mkern-1.5mu\left|\mkern-1.5mu\left| {#1} \right|\mkern-1.5mu\right|\mkern-1.5mu\right|_{#3}^{#2}}}

\newcommand{\tunnelmagnitude}[1]{{\mu\left({#1}\right)}}
\newcommand{\tunnelmodmagnitude}[2]{{\mu\left({#1}\middle\vert{#2}\right)}}

\newcommand{\tunnelextent}[1]{{\chi\left({#1}\right)}}

\newcommand{\tunnelsep}[2]{\mathrm{sep}\left({#1},{#2}\right)}

\newcommand{\alg}[1]{{\mathfrak{#1}}}

\newcommand{\module}[1]{{\mathscr{#1}}}

\newcommand{\ModState}[1]{\widehat{\StateSpace}}

\newcommand{\closure}[1]{\mathrm{cl}\left({#1}\right)}

\renewcommand{\geq}{\geqslant}
\renewcommand{\leq}{\leqslant}

\newcommand{\Dirac}[1][D]{{\slashed{#1}}}

\newcommand{\car}[1][c]{{\mathfrak{#1}}}
\newcommand{\bpartial}{{\overleftarrow{\partial}}}
\newcommand{\fpartial}{{\overrightarrow{\partial}}}
\newcommand{\Spinor}{{\mathscr{S}}}

\theoremstyle{plain}
\newtheorem{theorem}{Theorem}[section]
\newtheorem*{theorem*}{Theorem}

\newtheorem{corollary}[theorem]{Corollary}

\newtheorem{lemma}[theorem]{Lemma}
\newtheorem{proposition}[theorem]{Proposition}

\newtheorem{theorem-definition}[theorem]{Theorem-Definition}

\theoremstyle{definition}
\newtheorem{definition}[theorem]{Definition}
\newtheorem*{definition*}{Definition}

\newtheorem{notation}[theorem]{Notation}

\newtheorem{convention}[theorem]{Convention}

\theoremstyle{remark}

\newtheorem{remark}[theorem]{Remark}

\numberwithin{equation}{section}

\let\oldtocsection=\tocsection
\let\oldtocsubsection=\tocsubsection
\let\oldtocsubsubsection=\tocsubsubsection

\renewcommand{\tocsection}[2]{\bfseries \hspace{0em}\oldtocsection{#1}{#2}}
\renewcommand{\tocsubsection}[2]{\itshape \hspace{1em}\oldtocsubsection{#1}{#2}}
\renewcommand{\tocsubsubsection}[2]{\hspace{2em}\oldtocsubsubsection{#1}{#2}}

\begin{document}

\title[]{How to approximate the flat spectral triple of a quantum torus by fuzzy tori : a twisted tale}

\author{Fr\'{e}d\'{e}ric Latr\'{e}moli\`{e}re}
\email{frederic@math.du.edu}
\urladdr{http://fredericlatremoliere.com/}
\address{Department of Mathematics \\ University of Denver \\ Denver CO 80208}

\date{\today}

\begin{abstract}
We prove that the classical and the quantum flat torus can be rigorously approximated at a differential level by finite-dimensional fuzzy tori within the framework of the spectral propinquity, a distance up to unitary equivalence which we extend in this paper to the space of twisted, metric spectral triples. Standard attempts to establish this convergence are traditionally obstructed by the intrinsic non-locality of discrete calculus and the subsequent failure of the Leibniz rule. While contemporary alternatives such as spectral truncations circumvent this issue by abandoning $C^*$-algebras in favor of operator systems, our approximations instead remain within the $C^*$-algebraic category, generalizing the commutator formula used to define the Lipschitz seminorm induced by spectral triples. To this end, we introduce a relaxed notion of a twisted spectral triple where the twist is a linear map acting as a discretized Riesz transform that encapsulates the non-locality of the discrete world. Self-adjointness is obtained by adjoining the backward difference to the forward difference and representing this doubled calculus with the creation and annihilation operators of the CAR algebra. The maximum of the norm of the twisted commutator and norm of the displacement of the twist from the identity becomes our Lipschitz seminorm, from which the quantum metric structure of {\qcms s} on fuzzy tori is constructed. Our proposal to connect twisted spectral triples with {\qcms s} recovers the usual connection between metric spectral triples and {\qcms s}. By extending the spectral propinquity to this generalized setting of twisted spectral triples with possibly unbounded twists, we prove that fuzzy tori equipped with their natural two-sided discrete calculus converge to an amplification of the standard flat Dirac triple, while the underlying twists converge to the identity.
\end{abstract}
\maketitle
\tableofcontents


\section{Introduction}

The construction of matrix models which approximate the differential structure of classical flat tori, and their noncommutative counterparts, the quantum tori, have proven to be an elusive challenge. Discretizing the natural calculus of the torus onto finite spaces, such as closed subgroups of the torus, or their noncommutative analogues, the fuzzy tori, invariably breaks the Leibniz rule. This technical obstruction reflects the non-locality of the discretized calculus. It also immediately precludes that the discretized calculus is implemented by a spectral triple. Our previous work in \cite{Latremoliere13c} proves that fuzzy tori can approximate quantum tori from a purely metric perspective using the propinquity \cite{Latremoliere13,Latremoliere13b,Latremoliere14,Latremoliere15}, while our work in \cite{Latremoliere21a} proves that certain matrix models in mathematical physics approximate quantum tori with an unusual spectral triple in the sense of the spectral propinquity. However, we are left with the vexing issue of approximating the standard, flat spectral triple of the quantum torus with matrix models. This is the question this paper proposes to answer. As we shall see, solving this problem relies on a natural extension of the spectral propinquity \cite{Latremoliere18g} to twisted spectral triples. Our proposed extension of the spectral propinquity even allows for unbounded twists with a natural closed graph condition.

Historically, efforts to bypass the failure of the Leibniz rule and construct matrix models of tori have diverged into three distinct tracks, each making its own compromise. The first approach, heavily motivated by the physics literature (e.g. \cite{Schreivogl13, Kimura01}), retains both the fuzzy tori, viewed as C*-algebras, and the standard spectral triple framework by replacing the natural discrete derivatives with commutators. This restores the ordinary Leibniz identity and yields convergent matrix models for the spectral propinquity, as we demonstrated in \cite{Latremoliere21a}, though the resulting limit is a non-standard spectral triple over the torus.

More elaborate constructions of finite real spectral triples
\cite{Barrett15, BarrettGaunt25} incorporate additional terms corresponding to a spin connection. The fuzzy-torus constructions of Barrett and Gaunt concern dimension two and are designed to reproduce prescribed flat commutative Dirac operators, expressed in rotating frames and acting on amplified spinor spaces with specified spin structures. Their commutators encode coordinate-weighted inner derivations rather than the finite-difference gradient supplied directly by the finite translation action. The connection terms are chosen to recover the desired continuum operator. Although their eigenvalue formulas exhibit the corresponding fixed-mode spectral limits, convergence of the quantum metric spaces or spectral triples is not established.

These constructions also have structural limitations. In the basic square model, keeping a noncommutative limiting deformation parameter leaves the Dirac operators uniformly bounded. Such a limit cannot have compact resolvent on an infinite-dimensional Hilbert space. The same obstruction excludes irrational quantum-torus limits for their fixed integral-metric constructions. Moreover, metricity---namely, that the commutator seminorm induces a {\qcms} structure on the fuzzy torus---requires requires the defining cocycle to yield a full
matrix algebra, rather than a direct sum of matrix algebras. In dimension two, for the usual square finite groups, this means that the commutation parameter must have maximal order. If we were to pursue an extension of their construction to higher even dimensions, the corresponding nondegeneracy requirement likewise excludes every cocycle whose twisted group algebra has a nontrivial center. Indeed, their Dirac operators are assembled from left and right multiplication by elements of the original algebra, so their commutators vanish on its center. Consequently, every nontrivial center obstructs metricity; on the commutative C*-algebra of a nontrivial finite product of cyclic groups, the commutator seminorm is identically zero. 

These restrictions arise from the formulas themselves: changing the proposed limiting geometry does not remove them without modifying the construction. In particular, for the square flat geometry considered here, their normalization excludes noncommutative torus limits. Furthermore, an extension to odd-dimensional tori using the standard finite approximations $(\Z/N\Z)^d$ would necessarily encounter nontrivial centers, which obstruct metricity for the same commutator-based construction. Thus this approach does not provide the uniform treatment of arbitrary-dimensional quantum tori obtained here, and its gradient is tailored to reproduce a prescribed commutative continuum operator rather than taken directly from the canonical finite-difference calculus.

Our approach instead starts from the finite-difference calculus associated with the canonical finite translation actions. It applies in arbitrary dimension, allows degenerate cocycles and nontrivial centers, and establishes convergence to the flat spectral geometry of arbitrary quantum tori, including commutative ones. The twists accommodate the discrete product identities while retaining the C*-algebraic structure; the limiting calculus is obtained directly from the convergence of the discrete derivatives.

A second approach is to preserve the natural discrete calculus on the fuzzy tori but replace spectral triples with noncommutative differential bimodules \cite{Madore,Sauvageot89,Madore90,Grosse01,Sauvageot03,Rieffel14,Majid20}. However, developing the metric structure independently of spectral triples does not by itself retain the well-developed analytic machinery of Connes' framework, including the spectral properties of the Dirac operator and the connections to K-homology and other aspects of noncommutative geometry. While differential bimodule approaches may themselves involve Hilbert spaces, our aim is to retain the specific connection between the differential calculus, the metric, and the Dirac operator provided by spectral triples. This connection is also natural from the perspective of quantum theory, where observables are represented by operators on Hilbert spaces.

The third, recent, alternative is the method of spectral truncations, which attempts to keep a spectral-triple-like framework and the exact commutator formula by abandoning the fuzzy tori and their C*-algebraic structure entirely, in favor of operator systems \cite{Suij21, Suij23a, Suij23}. Truncations offer a compelling physical interpretation, representing the restriction of physical data under partial spectral information. However, because they rely on spectral projections of the classical Dirac operator, they do not natively approximate the torus via matrix models. Furthermore, transitioning to operator systems exacts a toll: we now have to understand anew all the tools of noncommutative geometry within this larger category, e.g. \cite{Suij24,Suij24b}. From the standpoint of noncommutative metric geometry, truncations also reverse the logic we hope to eventually apply when working with the spectral propinquity: truncations inherently assume the continuum limit is known a priori so that it may be truncated, which directly contradicts the core ambition of discovering unknown limits from finite metric models. It should be added that metric approximations in noncommutative geometry based upon operator system were of course pioneered by Kerr \cite{Kerr02, Kerr09}, and that truncations based upon projections on spectral subspaces was a technique used in this area from its early days in \cite{Rieffel00}, and our own \cite{Latremoliere05}. 

These alternative approaches show that a choice must be made: one can either preserve the exact commutator formula at the cost of the C*-algebraic structure, or preserve the C*-algebraic structure at the cost of modifying the commutator. We advocate here for the latter. We propose retaining the fuzzy tori as C*-algebras, while accommodating the natural non-locality of the discrete calculus through a generalized, \emph{twisted} spectral triple. This enables us to bring the machinery developped for the convergence of spectral triples \cite{Latremoliere18g,Latremoliere18b,Latremoliere18c,Latremoliere18d,Latremoliere21a,Latremoliere20a} to bear on this problem.

We thus propose to start with a natural discrete calculus on fuzzy tori, and prove a convergence result for this calculus. There are, in fact, several such natural candidates where derivatives are replaced with difference operators. However, difference operators, when defined on the GNS Hilbert space for the natural faithful trace on our fuzzy tori, are not self-adjoint. A standard technique is thus used which leads to a form of doubling the spinor space. The main consequence is that the limit of our discrete calculus is a trivial amplification of the sought after flat spectral triple on quantum tori.

Our rationale for this approach is threefold. First, C*-algebras remain the natural, well-developed setting for quantum metric geometry, physical modeling, and the study of auxiliary structures such as group actions. Notably, without the multiplicative structure of a C*-algebra, it becomes conceptually ambiguous to define what a "derivation" or "gradient" fundamentally is among arbitrary linear maps. Second, this departure is well-justified by geometric precedent; twisted spectral triples naturally emerge, from quantum groups to twisted cyclic cohomology, precisely when standard commutators fail to capture the underlying geometry, e.g. \cite{Connes08, Kaad12,Marcolli14,Kaad25}. Because the failure of the Leibniz property is intrinsically tied to the rigid algebraic structure of a standard commutator, introducing an algebraic twist is a natural and mathematically elegant remedy. Third, and crucially for our program, twisted spectral triples integrate seamlessly with the spectral propinquity. The definition of the spectral propinquity does not strictly rely on the exact commutator formula relating a Dirac operator to a gradient \cite{Latremoliere18g}; rather, we construct a map from spectral triples to actions on quantum vector bundles, then apply the more general theory of the covariant metrical propinquity \cite{Latremoliere18d}. By extending this map to accommodate twists, we immediately obtain a distance up to unitary equivalence via the covariant propinquity, a construction that strictly requires the underlying C*-algebraic framework even as it relaxes the commutator formula. Such a construction is of intrinsic interest to possibly handle the convergence of other twisted metric spectral triples. Notably, we will be able to develop this framework to handle some unbounded twists, which is surprising since these typically lead to a complete lack of even a relaxed form of the Leibniz inequality for the associated Lipschitz seminorm analogue, a feature which was of great importance in our program.

The twisted spectral triples developed in this paper are more general than those commonly  encountered in the literature. Our twists are \emph{linear} maps rather than algebra homomorphisms and take values in the full bounded operator algebra over the Hilbert space of the triple. While in this work, our twists for the fuzzy tori are bounded, we will extend the spectral propinquity so that it may also accommodate certain unbounded twists, with an eye on common constructions in the literature. We see our twists more as a perturbation of the identity, and since our examples here are finite dimensional, we leave the question of how to handle such twisted spectral triples to obtain, for instance, KK cycles, to another publication.  Ultimately, because the approximation of the classical tori by fuzzy tori is a foundational test case for noncommutative geometry, any robust metric framework must accommodate it. The present work fills this critical gap, demonstrating that generalized twists offer the precise flexibility required to bridge the discrete and the continuous.

Future applications of the present work include new examples of finite dimensional approximations, such as approximations of various spectral triples on noncommutative solenoids \cite{Latremoliere11c,Latremoliere16,FarsiPacker22,Latremoliere23a,Latremoliere24a}, and other examples based upon quantum tori \cite{Latremoliere15c}; moreover, our spectral propinquity can now be applied to twisted spectral triples, new or already in the literature. We also should investigate other natural metric on tori, such as the natural geometry from the usual embeddings of the torus in Euclidean spaces.

\medskip

Some other fuzzy spaces, such as fuzzy spheres, are also extensively studied in the mathematical physics literature. Foundational contributions include \cite{Hoppe82,Madore91, Madore92,Madore} and subsequent matrix-model developments such as \cite{Kimura01}. Dirac operators and spectral triples on fuzzy spheres have been studied from several perspectives, including \cite{GrossePresnajder95, CarowWatamura97, CarowWatamura98, BalachandranImmirzi03, Barrett15}. Their metric properties and their relationship with
spectral cut-offs and truncations have also been investigated in \cite{DAndreaLizziVarilly13, dandrea13,Suij21}. A complementary geometric approach constructs Dirac operators from quantum metrics and bimodule connections, with applications
to fuzzy spheres \cite{LiraTorresMajid21, LiraTorresMajid22, Majid25}; earlier work on their differential calculi includes \cite{Majid11}.

The idea of using metrics to formalize the convergence of fuzzy spheres to $2$-spheres \cite{Rieffel01} was part of the original motivation in Rieffel's introduction of the quantum Gromov-Hausdorff distance \cite{Rieffel00}, a precursor to the propinquity. Additional work on this topic by Rieffel includes \cite{Rieffel05,Rieffel08,Rieffel09,Rieffel10, Rieffel10b,Rieffel15}. More recent developments in this direction established convergence of the fuzzy spheres to the sphere for the dual propinquity \cite{Rieffel21}. Rieffel also discusses a possible extension to convergence for the spectral propinquity, with the important qualification that the proposed limiting spectral triple is the one constructed from the group action, rather than the geometric spinor triple on the homogeneous space. Notably, unlike the present case of the quantum torus, the calculus on fuzzy spheres used by Rieffel are the infinitesimal representations of the Lie algebra of $SU(2)$ associated with irreducible representations of this group on matrix algebra, the core idea being that the symmetries of the $2$-sphere, seen as a homogeneous space for $SU(2)$, still make sense on the quantized finite approximations, the fuzzy sphere, which matters since physics very much care about such symmetries. As a result, the spectral triple obtained from the ergodic action of $SU(2)$ on fuzzy spheres by conjugation with an irreducible representation, using \cite{Rieffel98a, Gabriel16}, is indeed a genuine spectral triples with no requirement for a twist : it is still ``local'', i.e. it is an infinitesimal calculus satisfying the usual Leibniz identity; in fact these are indeed defined explicitely as derivations in the calculus sense using the Lie group action of $SU(2)$ \cite{Bratteli86}). 

The fundamental difference between fuzzy spheres and our present work on fuzzy tori is interesting: while $SU(2)$ admits ergodic unitary representations on matrix algebras of any  (finite) dimension (that is, actions with trivial fixed point subalgebra), tori do not. Indeed, any action $\alpha$ of a torus on $M_n(\C)$ decomposes the algebra in a direct sum of spectral subspaces (invariant subspaces of the form $\A_\xi \coloneqq \{ a \in \A : \forall z \in \T^d \quad \alpha^z(a) = \xi(z) a \}$ for a fixed character $\xi \in \Z^d$ of $\T^d$). Let $\xi$ be a character of $\T^d$ such that $\A_\xi \neq\{0\}$.  If $\alpha$ was ergodic, and if $a \in \A_\xi$, with $a\neq 0$, then $\alpha^z(a) = \xi(z) a$ for all $z\in \T^d$, so $\alpha^z(a^\ast a) = a^\ast a$. So $a^\ast a \in \C I_n$ by ergodicity. Since $a\neq 0$, we conclude $a^\ast a $ is invertible. So $a$ is invertible as well. Now, actions on $\M_n$ are inner, so 
\begin{equation*}
	\xi(z)^n \det(a) = \det(\xi(z) a) = \det(\alpha^z(a)) = \det(a)\text.
\end{equation*}
Since $\det(a) \neq 0$, we conclude that $\xi(z)^n = 1$ for all $z\in \T^d$. This implies $\xi = 1$ since $\T^d$ is connected. Hence, the only nontrivial spectral subspace is the fixed point subalgebra $\C I_n$ of $\alpha$, which must thus equal $M_n$; we conclude that $n = 1$. In fact, if we replace $M_n$ by a more general finite dimensional C*-algebra, we note that connectedness of $\T^d$ forces the center of $\A$ to lie within the fixed point algebra for the action, and thus if the action is ergodic, the center is trivial, leading us back to the full matrix algebra case.

Now, the obstruction to the existence of an ergodic action of $\T^d$ on a finite dimensional C*-algebra is the combination of the connectedness and the commutative nature of $\T^d$ (the above can be generalized to the action of any connected compact Abelian group).
 But we note that characters $\xi$ of products of cyclic groups such that $\xi^n = 1$ for some $n \in \N\setminus\{0\}$ do exist : this is of course the very situation we exploit in this paper to give fuzzy tori a calculus. However, in doing so, we lose the local nature of the calculus on the continuum, seen as the loss of the Leibniz property. For the fuzzy sphere, in  contrast, the calculus still arises as an infinitesimal calculus, not a discrete one, even though $SU(2)$ is connected. It would also be interesting to investigate finite-difference calculi on matrix approximations of the sphere, and whether the methods developed here can establish their convergence in spectral propinquity. In short, the key difference between fuzzy tori and fuzzy spheres is the nature of the underlying symmetries, which are fundamentally discrete for fuzzy tori while still continuous for fuzzy spheres. More general fuzzy spaces, where the calculus might not arise from symmetries alone, will likely benefit from the discrete perspective adopted here.

\medskip

In deference to standard conventions in the literature, we reluctantly continue to refer to the following generalized structure as a "triple."

\begin{definition}\label{twisted-def}
A \emph{twisted spectral triple} $(\A,\Hilbert,\Dirac,\rho)$ consists of a C*-algebra $\A$ acting on a Hilbert space $\Hilbert$, a self-adjoint operator $\Dirac$ defined on a dense subspace $\dom{\Dirac}$ of $\Hilbert$ with compact resolvent, and a \emph{linear} map
\begin{equation*}
\rho : \A_{\Dirac} \rightarrow \B(\Hilbert)
\end{equation*}
defined on a dense subspace $\A_{\Dirac}$ of $\A$, such that if $a \in \A_{\Dirac}$ then $a \dom{\Dirac} \subseteq \dom{\Dirac}$, and the twisted commutator
\begin{equation*}
[\Dirac,a]_\rho \coloneqq \Dirac a - \rho(a) \Dirac
\end{equation*}
is bounded.

A \emph{spectral triple} $(\A,\Hilbert,\Dirac)$ is a twisted spectral triple $(\A,\Hilbert,\Dirac,\mathrm{id}_\A)$ with twist the identity of $\A$; in that case we would insist that $\A_{\Dirac}$ be a dense *-subalgebra.
\end{definition}

\begin{remark}
We emphasize that $\rho$ is only assumed to be a linear map, not an algebra homomorphism, that it is not necessarily valued in $\A$, and $\A_{\Dirac}$ is not necessarily a subalgebra of $\A$. We will impose additional requirements when specifically qualifying these structures as metric spectral triples.
\end{remark}

Our presentation begins with an exposition of the discretized calculus on fuzzy tori, and presents our twisted spectral triple. We then prove that indeed, fuzzy tori converge to quantum tori for the propinquity, as {\qcms s}. To handle the differential structure, we extend the spectral propinquity to twisted metric spectral triples, and prove that it is indeed zero exactly between unitarily equivalent twisted spectral triples. We apply this construction to conclude by proving the convergence of our twisted spectral triples on fuzzy tori to an amplification of the usual spectral triple on quantum tori.

\section{The twisted Spectral Triple for Fuzzy Tori}

\subsection{Quantum and Fuzzy Tori}

We begin our exposition with a description of the C*-algebras of the fuzzy and quantum tori, choosing a presentation suited for our purpose. We rely upon our presentation in \cite{Latremoliere21a} and thus, we refer to that paper for proofs of our statements.

We fix $d \in \N\setminus\{0,1\}$ in this entire article. Let $\Pbar = (\N\setminus\{0\})\cup\{\infty\}$ be the one point compactification of the set $\N_\ast \coloneqq \N\setminus\{0\}$ of strictly positive natural numbers.

For all $k\coloneqq(k_1,\ldots,k_d) \in \Pbar^d$, we define the groups:
\begin{equation*}
	k\Z^d \coloneqq \prod_{j=1}^d k_j \Z \text{ and }\Z_k^d \coloneqq  \bigslant{\Z^d}{k\Z^d} = \prod_{j=1}^d \bigslant{\Z}{k_j \Z} \text,
\end{equation*}
with the convention that $\infty \Z = 0 \Z = \{ 0 \}$. In particular, $\Z^d_{\infty,\ldots,\infty} = \Z^d$. Every quotient of $\Z^d$ is isomorphic as a discrete group to $\Z^d_k$ for some $k \in \Pbar^d$. 

Let $\mathcal{C}_2(\Z^d)$ be the space of all normalized $2$-cocycles of $\Z^d$, endowed with the topology of pointwise convergence over $\Z^d$ (note that by Tychonoff theorem, $\mathcal{C}_2(\Z^d)$ is compact). We now define our parameter space as the following subspace of $\Pbar^d\times \mathcal{C}_2(\Z^d)$:
\begin{equation*}
	\Xi \coloneqq \left\{ (k,\sigma) \in \Pbar^d\times \mathcal{C}_2(\Z^d) : \sigma \text{ induces a $2$-cocycle on }\Z^d_k \right\} \text,
\end{equation*}
which is again compact (as a closed subspace of the compact $\Pbar^d\times\mathcal{C}_2(\Z^d)$).

For any $(k,\sigma) \in \Xi$, for any $m,n \in \Z^d_k$, and for any $\xi \in \ell^2(\Z^d_k)$, we define:
\begin{equation*}
	W_{k,\sigma}^m \xi (n) = \sigma(m,n-m)\xi(n-m)
\end{equation*}
and observe that $W_{k,\sigma}^z$ thus defined is a unitary operator on $\ell^2(\Z^d_k)$; moreover $z \in \Z^d_k \mapsto W_{k,\sigma}^z$ is a $\sigma$-projective unitary representation of $\Z^d_k$ \cite[Lemma 2.6]{Latremoliere21a}, i.e. for all $m,n \in \Z^d_k$,
\begin{equation*}
	W_{k,\sigma}^m W_{k,\sigma}^n = \sigma(m,n) W_{k,\sigma}^{m+n} \text.
\end{equation*}
The C*-algebra generated by the unitaries $\{ W_{k,\sigma}^m : m \in \Z^d_k \}$ is the twisted group C*-algebra $C^\ast(\Z^d_k,\sigma)$ of $\Z^d_k$ by the $2$-cocycle $\sigma$; we shall however use the following simplified notation throughout our paper:
\begin{equation*}
	\A_{k,\sigma} \coloneqq C^\ast(\Z^d_k,\sigma) = C^\ast\left(\left\{ W_{k,\sigma}^m : m \in \Z^d_k \right\}\right) \text.
\end{equation*}

\begin{remark}
	Since the groups $\Z^d_k$ are Abelian for any $k\in\Pbar$, they are amenable, and thus we do not distinguish between the reduced group C*-algebras (as defined here) and the full group C*-algebras, as they are *-isomorphic.
\end{remark}

\begin{notation}\label{integrated-rep-notation}
Let $(k,\sigma) \in \Xi$, and let $q_k : \Z^d \to \Z^d_k$ be the canonical surjection. Let $S\subseteq\Z^d$ be finite and not empty. If $f : S\to \C$, then we define:
\begin{equation*}
	f_{k,\sigma} \coloneqq \sum_{m \in S} f(m) W_{k,\sigma}^{q_k(m)}\text.
\end{equation*}
Note that, since $S$ is finite, there exists a neighborhood $N$ of $(\infty,\ldots,\infty)$ in $\Pbar^d$ such that if $(k,\sigma) \in \left\{ (h,\varsigma) \in \Xi : h \in N \right\}$, then the canonical surjection $q_k : \Z^d \twoheadrightarrow \Z^d_k$ is injective when restricted to $S$. In this case, the map $f \in C(S) \mapsto f_{k,\sigma} \in \A_{k,\sigma}$ is an injection. The subspace $\left\{ f_{k,\sigma} : f \colon S \to \C \right\}$ is denoted by $\A_{k,\sigma}(S)$.
\end{notation}

The Pontryagin dual group of $\Z^d_k$, for $k \coloneqq (k_1,\ldots,k_d) \in \Pbar^d$, is the closed subgroup $\U_k^d$ of $\T^d$ defined by:
\begin{equation*}
	\U_k^d \coloneqq \left\{ z \in \T^d : z^k = (1,\ldots,1) \right\}
\end{equation*}
where $\T = \{ z\in\C : |z| = 1\}$, and for all $z \in \T$, we write $z^\infty = 1$, and for all $z\coloneqq (z_1,\ldots,z_d) \in \T^d$, we write
\begin{equation*}
	z^k \coloneqq (z_1^{k_1}, \ldots, z_d^{k_d})\text.
\end{equation*}
There is a canonical, strongly continuous action $\alpha_{k,\sigma}$ of $\U_k^d$ on $\A_{k,\sigma}$ by *-automorphisms, called the dual action, which moreover is spatially implemented in our chosen representation of $\A_{k,\sigma}$. For each $z \in \U^d_k$, we define a unitary $u_k^z$ on $\ell^2(\Z^d_k)$ by setting, for all $\xi \in \ell^2(\Z^d_k)$:
\begin{equation*}
	u_k^z \xi : m \in \Z^d_k \mapsto z^m \xi(m) \text.
\end{equation*}
It is a standard exercise to check that if we denote by $\alpha_{k,\sigma}$ the action induced on the C*-algebra $\B(\ell^2(\Z^d_k))$ of bounded linear operators on $\ell^2(\Z^d_k)$ by setting, for all $z\in\U^d_k$,
\begin{equation*}
	\alpha_{k,\sigma}^z(T) \coloneqq u_k^z T u_k^{-z}
\end{equation*}
then $\alpha_{k,\sigma}(\A_{k,\sigma}) = \A_{k,\sigma}$; in fact a direct computation shows \cite[Lemma 2.24, Corollary 2.26]{Latremoliere21a} that for all $m \in \Z^d_k$:
\begin{equation*}
	\alpha_k^z(W_{k,\sigma}^m) = z^m W_{k,\sigma}^m \text.
\end{equation*}
This action is a strongly continuous action by *-automorphisms on $\A_{k,\sigma}$. Notation asides, these results date back to \cite{Zeller-Meier68}.

\begin{convention}
	As an abuse of notation, we will denote $(\infty,\ldots,\infty) \in \Pbar^d$ simply by $\infty$ when no confusion may arise.
\end{convention}

The C*-algebras $\A_{\infty,\sigma}$, for $\sigma \in \mathscr{C}_2(\Z^d)$, is known as a \emph{quantum torus}; in particular if $\sigma = 1$ then $\A_{\infty,1} = C(\T^d)$ is the C*-algebra of $\C$-valued continuous functions over the $d$-torus $\T^d$. Quantum tori are the prototypes of a noncommutative manifolds \cite{Connes80}. When $(k,\sigma) \in \Xi$ with $k \in \N^d$, the C*-algebra $\A_{k,\sigma}$ is known as a \emph{fuzzy tori} \cite{Susskind97,Kimura01,Bigatti01, Schreivogl13}; note that references to fuzzy tori are often from the mathematical physics world. It is the noncommutative generalization of the commutative case $\A_{k,1} = C(\U^d_k)$; in particular these C*-algebras are finite dimensional and many of them are full matrix algebras. Our purpose is to show that we can discretize the standard calculus on $C(\T^d)$, and more generally $\A_{\infty,\sigma}$ to fuzzy tori in such a way as to obtain approximations in the sense of the spectral propinquity appropriately extended to our current context. In general, approximations of quantum tori by fuzzy tori are already a subject of interest \cite{Latremoliere13c,Latremoliere21a}, but either with different metrics, or only as metric approximations. As a side note, we do not really know of a name for $\A_{k,\sigma}$ when $k\neq\infty$ and $k\notin \N^d$; we shall simply refer to these as fuzzy tori here.

\subsection{The Quantized Calculus on Fuzzy and Quantum Tori}

If $z \in \T$ an $j \in \{1,\ldots,d\}$, we introduce the notation
\begin{equation*}
	(j \gets z) \coloneqq \big( 1, \ldots, 1 , \underbracket{z}_{\text{index }j }, 1, \ldots , 1 \big) \text.
\end{equation*}	

We fix $(k,\sigma) \in \Xi$, so $k \coloneqq (k_1,\ldots,k_d) \in \Nbar^d$ and $\sigma$ is a normalized $2$-cocycle of $\Z^d_k$ (identified with its unique lift to $\Z^d$ whenever convenient). We use the dual action $\alpha_{k,\sigma}$ to define a quantized calculus on $\A_{k,\sigma}$. 

Let $j \in \{1,\ldots,d\}$. Let us first work the case where $k_j = \infty$. 
 By \cite{Bratteli79}, there exists a dense *-subalgebra $\dom{\partial_{k,\sigma}^j}$ of $\A_{k,\sigma}$ and a closed linear endomorphism $\partial_{k,\sigma}^j  \colon \dom{\partial_{k,\sigma}^j} \to \A_{k,\sigma}$ such that, for all $a \in C_c(\Z^d_k) \subseteq \A_{k,\sigma}$, the following holds:
\begin{equation*}
	\partial_{k,\sigma}^j(a) = \lim_{t\rightarrow 0} \frac{\alpha_{k,\sigma}^{(j\gets \exp(i t))}(a)-a}{t}\text.
\end{equation*}
The linear endomorphism $\partial_{k,\sigma}^j$, in this case, is a *-derivation of $\A_{k,\sigma}$. For all $a,b \in \dom{\partial_{k,\sigma}^j}$, we indeed have the Leibniz rule:
\begin{equation*}
	\partial_{k,\sigma}^j (a b) = a \partial_{k,\sigma}^j (b) + \partial_{k,\sigma}^j (a) b \text.
\end{equation*}

We now work on the case when $k_j < \infty$. We define:
\begin{equation*}
	\lambda_{k,j} \coloneqq \exp\left( \frac{2 i \pi}{k_j} \right) \text.
\end{equation*}
We then denote the automorphism $\alpha_{k,\sigma}^{(j\gets \lambda_{k,j})}$ by $\alpha_{k,\sigma}^j$ for simplicity. In the discrete case, there really is a continuum of natural choices of discrete derivatives, namely the convex combinations of the forward and backward difference operators defined on $\A_{k,\sigma}$ by:
\begin{equation*}
	\fpartial_{k,\sigma}^j : a\in\A_{k,\sigma} \mapsto \frac{k_j}{2\pi} \left( \alpha_{k,\sigma}^j(a) - a \right) \text.
\end{equation*}
and
\begin{equation*}
	\bpartial_{k,\sigma}^j : a\in\A_{k,\sigma} \mapsto \frac{k_j}{2\pi} \left( a - (\alpha_{k,\sigma}^j)^{-1}(a)  \right) \text.
\end{equation*}
Thus $\bpartial_{k,\sigma}^j=(\alpha_{k,\sigma}^j)^{-1}\circ\fpartial_{k,\sigma}^j$.

Now, the linear endomorphisms $\fpartial_{k,\sigma}^j$ and $\bpartial_{k,\sigma}^j$ of $\A_{k,\sigma}$ are \emph{not} derivations of $\A_{k,\sigma}$, but rather satisfy a ``twisted'' form of the Leibniz identity: for all $a,b \in \A_{k,\sigma}$,
\begin{equation*}
	\fpartial_{k,\sigma}^j (ab) = a\fpartial_{k,\sigma}^j (b) + \fpartial_{k,\sigma}^j (a) \alpha_{k,\sigma}^j (b) \text,
\end{equation*}
and
\begin{equation*}
	\bpartial_{k,\sigma}^j(ab)=a\bpartial_{k,\sigma}^j(b)+\bpartial_{k,\sigma}^j(a)(\alpha_{k,\sigma}^j)^{-1}(b).
\end{equation*}
This is an obvious obstruction to implementing this calculus with a spectral triple. We note that here, $\dom{\fpartial_{k,\sigma}^j} = \dom{\bpartial_{k,\sigma}^j} = \A_{k,\sigma}$.

We now use our calculi to define gradients. We first introduce the CAR algebra $\operatorname{CAR}_d$ associated with $\C^d$, generated by $\car_1,\ldots,\car_d$ subject to
\begin{equation*}
	\car_j\car_l+\car_l\car_j=0,
	\qquad
	\car_j\car_l^*+\car_l^*\car_j=\delta_{j,l}1\text{, where $\delta_{j,l}$ is the Kronecker symbol.}
\end{equation*}
We use its canonical faithful representation on the fermionic Fock space
\begin{equation*}
	\Spinor_d \coloneqq \bigwedge \C^d	= \bigoplus_{k=0}^d \bigwedge^k \C^d
	= \C \oplus \bigoplus_{k=1}^d \bigwedge^k \C^d \text,
\end{equation*}
where $\bigwedge \C^d$ is the exterior algebra. If $e_1,\ldots,e_d$ is the canonical basis of $\C^d$, then we endow $\Spinor_d$ with the unique inner product for which
\begin{equation*}
	1
	\quad\text{and}\quad
	e_{j_1}\wedge\cdots\wedge e_{j_k},
	\qquad
	1\leq j_1<\cdots<j_k\leq d,
\end{equation*}
form an orthonormal basis. This representation sends each $\car_j$ to the adjoint of the creation operator
\begin{equation*}
	\xi\in\bigwedge\C^d\longmapsto e_j\wedge\xi \text.
\end{equation*}

For later comparison with the usual Clifford presentation, set
\begin{equation*}
	\widehat\gamma_j\coloneqq\car_j+\car_j^*,
	\qquad
	\gamma_j\coloneqq i(\car_j-\car_j^*).
\end{equation*}
The $2d$ operators $\widehat\gamma_1,\ldots,\widehat\gamma_d,\gamma_1,\ldots,\gamma_d$ are anticommuting self-adjoint unitaries. The choice of the CAR generators will make our presentation more elegant, and the doubling of the spinors more natural.

For use with the continuous-field argument below, we also identify $\operatorname{CAR}_d$ with the twisted group C*-algebra of $\Z_2^{2d}$ for the cocycle
\begin{equation}\label{car-field-cocycle-eq}
	\mathfrak c(p,q)\coloneqq(-1)^{\sum_{1\leq r<s\leq2d}q_rp_s}
	\qquad(p,q\in\Z_2^{2d}).
\end{equation}
Under this identification, $\car_j$ and $\car_j^*$ are fixed elements of this finite-dimensional algebra.

It will be helpful, to keep our formulas ``symmetric'', to introduce the following notation when $k_j = \infty$:
\begin{equation*}
	\fpartial_{k,\sigma}^j = \bpartial_{k,\sigma}^j \coloneqq \partial_{k,\sigma}^j \text.
\end{equation*}
We note that $\dom{\bpartial_{k,\sigma}^j} = \dom{\fpartial_{k,\sigma}^j}$ for any $(k,\sigma) \in\Xi$. Let
\begin{equation*}
	\dom{\gradient_{k,\sigma}}
	\coloneqq
	\bigcap_{j=1}^d \dom{\fpartial_{k,\sigma}^j} \text.
\end{equation*}

Rather than selecting a scalar convex combination of the forward and backward difference operators, we retain both operators and use the CAR generators as independent coefficients. This treats the two orientations symmetrically, preserves the information carried by each, and is adapted to the construction of a self-adjoint Dirac operator, since taking adjoints exchanges creation and annihilation.

Accordingly, the quantized two-sided gradient of $a\in\dom{\gradient_{k,\sigma}}$ is:
\begin{equation}\label{car-gradient-eq}
	\gradient_{k,\sigma}a \coloneqq
	\sum_{j=1}^d\left( \fpartial_{k,\sigma}^j(a)\otimes\car_j - \bpartial_{k,\sigma}^j(a)\otimes\car_j^* \right)\text.
\end{equation}

Strictly speaking, this element acts as the Clifford multiplication by the quantized two-sided gradient on $\A_{k,\sigma}$, but we will not make this distinction here. We record that $\gradient_{k,\sigma}$ is closed for the norm topology. At $k=\infty$, we recover:
\begin{equation*}
	\gradient_{\infty,\sigma}a=-i\sum_{j=1}^d \partial_{\infty,\sigma}^j(a)\otimes\gamma_j\text,
\end{equation*}
which is the usual flat Clifford multiplication, represented on the faithful Fock module.

The CAR relations give, for every $j\in\{1,\ldots,d\}$,
\begin{align}
	\fpartial_{k,\sigma}^j(a)\otimes1
	&=(1\otimes\car_j^*)\gradient_{k,\sigma}a
	+\gradient_{k,\sigma}a(1\otimes\car_j^*),
	\label{car-forward-extract-eq}\\
	\bpartial_{k,\sigma}^j(a)\otimes1
	&=-\left((1\otimes\car_j)\gradient_{k,\sigma}a
	+\gradient_{k,\sigma}a(1\otimes\car_j)\right).
	\label{car-backward-extract-eq}
\end{align}
Consequently,
\begin{equation}\label{car-component-bound-eq}
	\max\left\{
	\norm{\fpartial_{k,\sigma}^j(a)}{\A_{k,\sigma}},
	\norm{\bpartial_{k,\sigma}^j(a)}{\A_{k,\sigma}}
	:1\leq j\leq d
	\right\}
	\leq2\norm{\gradient_{k,\sigma}a}{\A_{k,\sigma}\otimes\operatorname{CAR}_d}.
\end{equation}

We thus have defined a form of quantized calculus, albeit one which, on its discrete component, satisfies only a twisted version of the Leibniz identity. On the other hand, when $k = (\infty,\ldots,\infty)$, our construction exactly recovers the usual quantized calculus of the quantum torus \cite{Connes80,Connes}, \emph{which is implemented by a spectral triple, in fact using the usual Dirac operator for the flat metric on the torus $\T^d$}. Our claim that the naturally discretized version is indeed the appropriate substitute for fuzzy tori is arguably the point of this paper, as we shall see it indeed approximates the usual quantized calculus on quantum tori in a precise and quite general sense.

We use our quantized calculus to define a {\qcms} structure, defined below, over $\A_{k,\sigma}$ \cite{Connes89, Rieffel98a,Latremoliere13,Latremoliere13b,Latremoliere15}. The proposed analogue of a Lipschitz seminorm on $\A_{k,\sigma}$, called an L-seminorm, is given here by defining $\Lip_{k,\sigma}$ on $\A_{k,\sigma}$ as follows.

\begin{notation}
	If $E$ is a normed vector space, its norm is denoted by $\norm{\cdot}{E}$ unless stated otherwise. Moreover, the norm of a bounded operator $T$ on $E$ is denoted by $\opnorm{T}{}{E}$.
\end{notation}

\begin{definition}
	We define the seminorm $\Lip_{k,\sigma}$ (allowing for the value $\infty)$ as the Minkowski gauge functional for the closure of the set
\begin{equation*}
	\left\{ a \in \dom{\gradient_{k,\sigma}} : \norm{\gradient_{k,\sigma} a}{\A_{k,\sigma}\otimes\operatorname{CAR}_d} \leq 1 \right\} \text.
\end{equation*}
\end{definition}

In particular, if $a \in \dom{\gradient_{k,\sigma}}$ then
\begin{equation*}
	\Lip_{k,\sigma}(a) =  \norm{\gradient_{k,\sigma} a}{\A_{k,\sigma}\otimes\operatorname{CAR}_d} \text.
\end{equation*}
\begin{remark}
	Since actually both $\operatorname{CAR}_d$ and $\A_{k,\sigma}$ are nuclear, there is a unique C*-norm on $\A_{k,\sigma} \otimes \operatorname{CAR}_d$ (we only need one tensorial factor to be nuclear of course, and $\operatorname{CAR}_d$ being finite dimensional, this would be the natural choice, but in any case, both are).
\end{remark}

We wish to prove that $(\A,\Lip_{k,\sigma})$ is indeed a {\qcms}. Our definition is written in  prevision of our section about the spectral propinquity, by introducing the option for a {\qcms} to lack the Leibniz property, in which case we refer to it as \emph{relaxed}. This is not important until our work later in this paper though.
\begin{definition}
	A relaxed {\qcms} $(\A,\Lip)$ is given by a unital C*-algebra $\A$ and a seminorm $\Lip$ defined on a dense subspace $\dom{\Lip}$ of the self-adjoint part $\sa{\A}$ of $\A$ such that:
	\begin{enumerate}
		\item $\{ a \in \dom{\Lip} : \Lip(a) = 0 \} = \R \unit_\A$,
		\item the \emph{\MongeKant} $\Kantorovich{\Lip}$ on the state space $\StateSpace(\A)$ of $\A$, defined for any two states $\varphi,\psi \in \StateSpace(\A)$ by:
		\begin{equation*}
			\Kantorovich{\Lip}(\varphi,\psi) = \sup\left\{ \left| \varphi(a) - \psi(a) \right| : a\in \dom{\Lip}, \Lip(a) \leq 1 \right\}
		\end{equation*}
		is a distance function whose induced topology is the weak* topology on $\StateSpace(\A)$,
		\item $\{a \in \dom{\Lip} : \Lip(a) \leq 1 \}$ is closed in $\A$.
	\end{enumerate}
	A relaxed {\qcms} is a $(C,D)$-{\qcms}, for $C\geq 1$ and $D\geq 0$, when moreover, $\Lip$ is a \emph{$(C,D)$-Leibniz seminorm}, meaning that $\dom{\Lip}$ is a Jordan-Lie subalgebra, and for all $a,b \in \dom{\Lip}$:
	\begin{equation*}
		\max\left\{ \Lip\left(\frac{ab+ba}{2}\right), \Lip\left(\frac{ab-ba}{2i}\right) \right\} \leq C \left( \Lip(a) \norm{b}{\A} + \norm{a}{\A} \Lip(b) \right) + D \Lip(a) \Lip(b) \text.
	\end{equation*} 
\end{definition}

\begin{convention}
	If $L$ is a seminorm on a subspace $F$ of a vector space $E$, we define $L(e) = \infty$ for any $e \in E\setminus F$. With this convention, $\dom{L} = \{ e \in E : L(e) < \infty \} = F$. For convenience, in this context, we write $0\infty = 0$; this way $L$ behaves as a seminorm but valued in $[0,\infty]$. Under this assumption, when $(\A,\Lip)$ is a relaxed {\qcms}, the seminorm $\Lip$ is a lower semicontinuous function on $\A$. 
\end{convention}

We pause for a terminology note. Just as twisted spectral triples are quadruples, and are generalizations of spectral triples, relaxed {\qcms s} are generalizations of {\qcms s}, and if we state that an ordered pair $(\A,\Lip)$ is a {\qcms}, we do mean that it is a $(C,D)$-Leibniz {\qcms} for some $C\geq 1$ and $D\geq 0$. In other words, we must qualify a {\qcms} as  ``relaxed'' whenever we mean to remove the Leibniz condition. We have typically preferred in our work to avoid this use of an adjective to widen, rather than restrict, the meaning of a term, but this is the convention we shall adopt here in accordance with the wider literature. A relaxed {\qcms} is pretty much a compact quantum metric space in the sense of \cite{Rieffel98a,Rieffel99,Rieffel00}, but with the underlying object being a C*-algebra rather than an order unit space, and requiring the Lip-norm to be closed.

Returning to our main topic for this section, we now prove the following.
\begin{proposition}
	For all $(k,\sigma) \in \Xi$, the {\MongeKant} induced on $\StateSpace(\A_{k,\sigma})$ by the seminorm:
	\begin{equation*}
		a \in \dom{\gradient_{k,\sigma}} \mapsto \norm{\gradient_{k,\sigma} a}{\A_{k,\sigma}\otimes\operatorname{CAR}_d}
	\end{equation*}
	metrizes the weak* topology.
\end{proposition}

\begin{proof}
First, note that $W_{k,\sigma}^m \in \dom{\gradient_{k,\sigma}}$ for all $m \in \Z^d_k$. So $\dom{\gradient_{k,\sigma}}$, as a subspace containing a total subset of $\A_{k,\sigma}$, is dense in $\A_{k,\sigma}$.

Let $a \in \dom{\gradient_{k,\sigma}}$. By the component-extraction identities \eqref{car-forward-extract-eq} and \eqref{car-backward-extract-eq}, we have:
\begin{align*}
	\max\left\{
	\norm{\fpartial_{k,\sigma}^j(a)}{\A_{k,\sigma}},
	\norm{\bpartial_{k,\sigma}^j(a)}{\A_{k,\sigma}} : j \in \{1,\ldots,d\} 
	\right\}
	&\leq2\norm{\gradient_{k,\sigma} a}{\A_{k,\sigma}\otimes\operatorname{CAR}_d}  \text.
\end{align*}
Therefore,
\begin{equation}\label{cqms-thm-eq-1}
	\max\left\{ \norm{\fpartial_{k,\sigma}^j a}{\A_{k,\sigma}} : j \in \{1,\ldots,d\} \right\} \leq 2\norm{\gradient_{k,\sigma} a}{\A_{k,\sigma}\otimes\operatorname{CAR}_d}
\end{equation}
which proves that
\begin{equation}\label{grad-L-norm}
	a\in\dom{\gradient_{k,\sigma}}\mapsto \norm{\gradient_{k,\sigma} a}{\A_{k,\sigma}\otimes\operatorname{CAR}_d}
\end{equation}
is a Lip-norm by \cite[Theorem 3.1]{Rieffel98a} and then\cite[Lemma 1.10]{Rieffel98a} as in \cite[Section 4]{Rieffel98a}.
\end{proof}

\begin{corollary}\label{Kantorovich-cor}
	For all $(k,\sigma) \in \Xi$, the following inclusion holds:
	\begin{equation*}
		\dom{\gradient_{k,\sigma}} \subseteq \dom{\Lip_{k,\sigma}}
	\end{equation*}
	and in particular, $\dom{\Lip_{k,\sigma}}$ is a dense subspace of $\A_{k,\sigma}$.
	
	Moreover, the {\MongeKant} $\Kantorovich{\Lip_{k,\sigma}}$ metrizes the weak* topology on the state space $\StateSpace(\A_{k,\sigma})$ of $\A_{k,\sigma}$. In particular, 
	\begin{equation*}
		\left\{ a \in \dom{\Lip_{k,\sigma}} : \Lip_{k,\sigma}(a) = 0 \right\} = \C \unit_{\A_{k,\sigma}} \text. 
	\end{equation*}
\end{corollary}

\begin{proof}
	A direct computation shows that $\dom{\gradient_{k,\sigma}} \subseteq \dom{\Lip_{k,\sigma}}$.

	By the construction of $\Lip_{k,\sigma}$, by definition of $\Kantorovich{\Lip_{k,\sigma}}$, and since states are continuous, it is immediate that $\Kantorovich{\Lip_{k,\sigma}}$ agrees with the {\MongeKant} for the seminorm in Expression \eqref{grad-L-norm}. Since $\left(\StateSpace(\A_{k,\sigma}),\Kantorovich{\Lip_{k,\sigma}}\right)$ is a compact metric space, it has bounded diameter, which in turns implies that if $\Lip_{k,\sigma}(a) = 0$ and $\psi \in \StateSpace(\A_{k,\sigma})$,  then $\varphi(a) = \psi(a)$ for all $\varphi \in \StateSpace(\A_{k,\sigma})$, so $a \in \C\unit_{\A_{k,\sigma}}$. Of course, $\Lip_{k,\sigma}(\unit_{\A_{k,\sigma}}) = 0$. This concludes our proof.
\end{proof}

Now, $\Lip_{k,\sigma}$ is a $(4d,0)$-Leibniz seminorm in general; if $k=\infty$ then it is even $(1,0)$-Leibniz.
\begin{lemma}\label{L-Leibniz-lemma}
	For all $a,b \in \dom{\Lip_{k,\sigma}}$, 
	\begin{equation*}
		\Lip_{k,\sigma}(ab) \leq \norm{a}{\A_{k,\sigma}} \Lip_{k,\sigma}(b) + 4d\norm{b}{\A_{k,\sigma}} \Lip_{k,\sigma}(a) \text.
	\end{equation*}
\end{lemma}

\begin{proof}
For all $a,b \in \dom{\gradient_{k,\sigma}}$:
\begin{equation}\label{L-Leibniz-eq}
\begin{split}
	\norm{\gradient_{k,\sigma} (ab)}{\A_{k,\sigma}\otimes\operatorname{CAR}_d}
	&\leq \norm{a}{\A_{k,\sigma}}\norm{\gradient_{k,\sigma}b}{\A_{k,\sigma}\otimes\operatorname{CAR}_d} \\
	&\quad + \sum_{j=1}^d \norm{\fpartial_{k,\sigma}^j(a)\alpha_{k,\sigma}^j(b)}{\A_{k,\sigma}}
	+ \sum_{j=1}^d \norm{\bpartial_{k,\sigma}^j(a)(\alpha_{k,\sigma}^j)^{-1}(b)}{\A_{k,\sigma}} \\
	&\leq \norm{a}{\A_{k,\sigma}}\norm{\gradient_{k,\sigma}b}{\A_{k,\sigma}\otimes\operatorname{CAR}_d}
	+ 4d\norm{b}{\A_{k,\sigma}}\norm{\gradient_{k,\sigma}a}{\A_{k,\sigma}\otimes\operatorname{CAR}_d} \text,
	\end{split}
\end{equation}
where the last inequality follows from the component estimates in Expression \eqref{car-component-bound-eq}.
This inequality extends to $\dom{\Lip_{k,\sigma}}$, as follows. Let $a,b \in \dom{\Lip_{k,\sigma}}$ with $\max\{ \Lip_{k,\sigma}(a), \Lip_{k,\sigma}(b)\} \leq 1$. By definition, there exists $(a_n)_{n\in\N}$ and $(b_n)_{n\in\N}$ in $\dom{\gradient_{k,\sigma}}$ converging, respectively, to $a$ and $b$ in norm, and such that, for all $n \in \N$,
\begin{equation*}
\norm{\gradient_{k,\sigma} a_n}{\A_{k,\sigma}\otimes\operatorname{CAR}_d} \leq 1 \text{ and }\norm{\gradient_{k,\sigma} b_n}{\A_{k,\sigma}\otimes\operatorname{CAR}_d} \leq 1 \text.
\end{equation*}
Therefore, for all $n \in \N$, we conclude by Expression \eqref{L-Leibniz-eq} that
\begin{equation*}
\norm{\gradient_{k,\sigma} (a_n b_n)}{\A_{k,\sigma}\otimes\operatorname{CAR}_d} \leq \norm{a_n}{\A_{k,\sigma}} + 4d \norm{b_n}{\A_{k,\sigma}} \xrightarrow{n\to\infty} \norm{a}{\A_{k,\sigma}} + 4d \norm{b}{\A_{k,\sigma}} \text.
\end{equation*}
Therefore, $\Lip_{k,\sigma}(ab) \leq \norm{a}{\A_{k,\sigma}} + 4d\norm{b}{\A_{k,\sigma}}$; in particular, $ab \in \dom{\Lip_{k,\sigma}}$. It follows by homogeneity that $\dom{\Lip_{k,\sigma}}$ is a subalgebra. It is easy to check that this domain is closed under adjunction and $\Lip_{k,\sigma}$ is a *-seminorm as well.
 
By homogeneity again, if $a,b \in \dom{\Lip_{k,\sigma}}$ with $\Lip_{k,\sigma}(a) > 0$ and $\Lip_{k,\sigma}(b) > 0$, then
\begin{align*}
	\Lip_{k,\sigma}(ab)
	&=\Lip_{k,\sigma}(a) \Lip_{k,\sigma}(b) \Lip_{k,\sigma}\left(\frac{a}{\Lip_{k,\sigma}(a)}\frac{b}{\Lip_{k,\sigma}(b)}\right) \\
	&\leq \Lip_{k,\sigma}(a) \Lip_{k,\sigma}(b)  \left( \norm{\frac{a}{\Lip_{k,\sigma}(a)}}{\A_{k,\sigma}} + 4d \norm{\frac{b}{\Lip_{k,\sigma}(b)}}{\A_{k,\sigma}} \right) \\
	& = \Lip_{k,\sigma}(b) \norm{a}{\A_{k,\sigma}} + 4d \Lip_{k,\sigma}(a) \norm{b}{\A_{k,\sigma}} \text,
\end{align*}
as desired. If $\Lip_{k,\sigma}(a) = 0$ then $a \in \C$ and the result is immediate, similarly if $\Lip_{k,\sigma}(b) = 0$. This concludes our proof.
\end{proof}

Applying the preceding estimate to both $ab$ and $ba$, and using $4d\geq1$, gives the Jordan and Lie estimates with constant $4d$. We can now summarize the crux of our construction thus far:
\begin{theorem}\label{qcms-thm}
	For all $(k,\sigma) \in\Xi$, the ordered pair $(\A_{k,\sigma},\Lip_{k,\sigma})$ is a $(4d,0)$-Leibniz {\qcms}.
\end{theorem}

\begin{proof}
	The seminorm $\Lip_{k,\sigma}$ is defined on a dense  *-subalgebra of $\A_{k,\sigma}$ whose {\MongeKant} $\Kantorovich{\Lip_{k,\sigma}}$ metrizes the weak* topology by Corollary \ref{Kantorovich-cor}, with $\ker \Lip_{k,\sigma} = \C$, and $\Lip_{k,\sigma}$ is lower semicontinuous by construction since its closed unit ball is closed in $\A_{k,\sigma}$. By Lemma \ref{L-Leibniz-lemma}, it has the $(4d,0)$-Leibniz property.
\end{proof}

\begin{remark}
Note that $\norm{\cdot}{\A_{k,\sigma}} + \norm{\gradient_{k,\sigma}\cdot}{\A_{k,\sigma}\otimes \operatorname{CAR}_d}$ is a noncommutative analogue of the $C^1$ norm, whereas $\Lip_{k,\sigma}$, as we shall see, is the analogue of the Lipschitz seminorm, which has a larger domain as soon as $k\notin \N^d$. Among many advantages in constructing a twisted spectral triple for our present calculus, we will indeed obtain a closed form of the seminorm $\Lip_{k,\sigma}$.
\end{remark}

\begin{remark}
	At this point, we only know that $\Lip_{k,\sigma}(a) \leq\norm{\gradient_{k,\sigma} a}{\A_{k,\sigma}\otimes\operatorname{CAR}_d}$ for all $a\in\dom{\gradient_{k,\sigma}}$. We will improve this a bit later on. 
\end{remark}

This construction turns the C*-algebra $\A_{k,\sigma}$ into a fuzzy torus or quantum torus with its natural discretized calculus, derived from the usual torus \cite{Connes80}. In particular, if $k \in \N^d$ and $\sigma_k = 1$, we recover the two-sided difference calculus generated by the forward difference and its adjoint on the subset $\U_k^d$ of the Lie group $\T^d$. We note however that we involve a two sided discrete calculus which doubles the dimension of the spinors.

\subsection{The twisted Spectral Triple}

A spectral triple could be thought of as a spatial representation of a differential calculus, though it has its origins in the more involved concept of a Dirac operator \cite{Connes80}. In any case, the lack of Leibniz property precludes our discretized gradient to be expressed as a commutator, i.e. to be represented by a spectral triple. We now investigate this matter and propose our twisted  spectral triple for the present work. The goal here is to obtain a spatial implementation of the calculus presented above, in such a way that we recover the L-seminorm $\Lip_{k,\sigma}$ using a twisted commutator.

Since our quantized calculi are induced by the action of the groups $\U^d_k$, and since this actions are in fact spatially implemented in the previous section, we naturally begin by trying to  represent our quantized calculus on the Hilbert space $\ell^2(\Z^d_k)$.

Let $j \in \{1,\ldots,d\}$. If $k_j = \infty$, then we set 
\begin{equation*}
	\dom{\nabla_k^j} \coloneqq \left\{ \xi \in \ell^2(\Z^d_k) :  (m_j \xi(m))_{m=(m_1,\ldots,m_d) \in \Z^d_k} \in \ell^2(\Z^d_k)  \right\} \text.
\end{equation*}
and for all $\xi \in \dom{\nabla_k^j}$, we set:
\begin{equation*}
	\nabla_k^j(\xi) \colon m \in \Z^d_k \mapsto i m_j \xi(m) \text.
\end{equation*}
Thus defined, $\nabla_k^j$ is a closed, unbounded operator from the dense subspace $\dom{\nabla_k^j}$ of $\ell^2(\Z^d_k)$ to $\ell^2(\Z^d_k)$. Notably, $\nabla_k^j$ is the generator of the circle action $z \in \T \mapsto u_k^{(j\gets z)}$ (so skew-adjoint), thus relating it to our previous section. More specifically, a simple computation shows that if $a\in \dom{\partial_{k,\sigma}^j}$, then
\begin{equation*}
	[\nabla_k^j , a] = \partial_{k,\sigma}^j(a) \text.
\end{equation*}
This reflects the fact that $\partial_{k,\sigma}^j$ is a *-derivation when $k_j = \infty$. For later convenience, we set $u_k^j = 1$.

If, instead, $k_j < \infty$, we denote $u_k^{(j\gets \lambda_{k,j})}$ by $u_k^j$, and then define the operator:
\begin{equation*}
	\nabla_k^j : \xi \in \ell^2(\Z_k^d) \mapsto \frac{k_j}{2\pi} \left( u_k^j - 1 \right)\xi \text.
\end{equation*}
This choice is the obvious one. The issue which arises is that we do not quite get the desired spatial implementation of our discretized calculus, of course because of the relaxed form of the Leibniz property. For all $a \in \A_{k,\sigma}$, we compute instead:
\begin{align*}
	[\nabla_k^j,a] 
	&= \frac{k_j}{2\pi} [u_k^j, a]  \\
	&= \frac{k_j}{2\pi} (u_k^j a  - a u_k^j ) \\
	&= \frac{k_j}{2\pi} (u_k^j a (u_k^j)^\ast u_k^j - a u_k^j ) \\
	&= \frac{k_j}{2\pi} (\alpha_{k,\sigma}^j (a) - a) u_k^j  \\
	&= \fpartial_{k,\sigma}^j(a) u_k^j  \text.
\end{align*}
Since $(\nabla_k^j)^*=\frac{k_j}{2\pi}((u_k^j)^*-1)$ in a finite direction, and $(\nabla_k^j)^*=-\nabla_k^j$ in a continuous direction, the same computation gives, with our convention $u_k^j=1$ when $k_j=\infty$,
\begin{equation}\label{backward-spatial-eq}
	[(\nabla_k^j)^*,a]
	=-\bpartial_{k,\sigma}^j(a)(u_k^j)^*.
\end{equation}
Now, individually for each $j \in \{1,\ldots,d\}$, the appearance of the unitary $u_k^j$ in this commutator is not necessarily worrisome. Where things become problematic is when we try to assemble our gradient from these components, because these parasitic unitaries are different for each component.

Indeed, let us now define the natural discretized Dirac operator for our quantized calculus. We will work on the Hilbert space:
 \begin{equation*}
	\Hilbert_k\coloneqq\ell^2(\Z_k^d)\otimes\Spinor_d
\end{equation*}
and define
\begin{equation}\label{doubled-dirac-eq}
	\Dirac_k^0 \coloneqq \sum_{j=1}^d\left(\nabla_k^j\otimes\car_j+(\nabla_k^j)^*\otimes\car_j^*\right) \text.
\end{equation}
Since the operators $\nabla_k^j$ are commuting normal operators, the CAR relations give
\begin{equation}\label{dirac-square-eq}
	(\Dirac_k^0)^2=\sum_{j=1}^d(\nabla_k^j)^*\nabla_k^j\otimes1.
\end{equation}
In particular, if $e_m$ denotes the canonical Fourier vector
\begin{equation*}
	e_m(n)\coloneqq\delta_{m,n} =
	\begin{cases}
		1 & \text{if $n=m$,}\\
		0 & \text{otherwise.}
	\end{cases}
\end{equation*}
 then
\begin{equation}\label{dirac-symbol-eq}
	\norm{\Dirac_k^0(e_m\otimes\zeta)}{}^2 =
	\left(
	\sum_{k_j=\infty}m_j^2
	+
	\sum_{k_j<\infty}\left(\frac{k_j}{\pi}\sin\left(\frac{\pi m_j}{k_j}\right)\right)^2
	\right)\norm{\zeta}{}^2.
\end{equation}
Thus $\Dirac_k^0$ is self-adjoint with compact resolvent. Moreover, its nonzero spectrum is uniformly separated from $0$, since $\frac{k_j}{\pi}\sin(\frac{\pi}{k_j})\geq\frac2\pi$ when $k_j\geq2$.

We thus can choose an additional translation term $\varepsilon_k 1$to ensure that 
\begin{equation*}
	\Dirac_k \coloneqq \Dirac_k^0 + \varepsilon_k 1
\end{equation*}
is invertible when $k\neq\infty$. We make the following choice throughout this paper.
\begin{convention}
	We set $\varepsilon_\infty=0$. For each $k\neq\infty$, we choose $0<\varepsilon_k<\frac1\pi$ with $\lim_{k\to(\infty,\ldots,\infty)} \varepsilon_k = 0$, and we set $\Dirac_k=\Dirac_k^0+\varepsilon_k1$.
\end{convention}

These choices are natural, and the formal analogy between $\Dirac_k$ and $\gradient_{k,\sigma}$ is apparent. Furthermore,
\begin{equation*}
	\Dirac_{\infty,\ldots,\infty}=\sum_{j=1}^d m_j\otimes\gamma_j
\end{equation*}
is exactly the Dirac operator which we are trying to approximate, represented on the fermionic Fock module $\Spinor_d$; thus, it is a finite-dimensional Clifford amplification of the usual flat Dirac operator. Moreover, on the span of any finitely many fixed Fourier vectors, $\Dirac_k$ looks increasingly like $\Dirac_{\infty,\ldots,\infty}$ as $k$ tends to infinity.

\begin{remark}
	The use of the faithful Fock representation changes only the finite spinor multiplicity. If $d$ is even, it is the direct sum of $2^{d/2}$ copies of the irreducible complex spin representation. If $d$ is odd, it contains both irreducible types, each with multiplicity $2^{(d-1)/2}$. Thus the limiting differential and metric are the standard flat ones, represented on the Fock module required by the self-adjoint doubling.
\end{remark}

However, as announced, for all $a\in \dom{\gradient_{k,\sigma}}$, we obtain:
\begin{equation}\label{actual-eq}
	[\Dirac_k,a] = \sum_{j=1}^d\left(
	\fpartial_{k,\sigma}^j(a)u_k^j\otimes\car_j
	-
	\bpartial_{k,\sigma}^j(a)(u_k^j)^*\otimes\car_j^*
	\right)\text.
\end{equation}
This is an important departure from the desired formula of Clifford multiplication by the discrete gradient in Expression \eqref{car-gradient-eq}, because using this commutator to define a Lipschitz seminorm would not give us the correct metric leading to an approximation of the continuous distance. This seems to be the end of this story in the literature. Note that it fails even when $\sigma = 1$, i.e. in the commutative case, as soon as $k$ has at least one finite component. The doubling fixes self-adjointness, but it does not by itself fix this separate failure of the usual Leibniz property.

This can, and it seems, this must be fixed by introducing a linear twist in our commutator (various other approaches seem to indeed fail; we have not found a Dirac operator which would differ from our natural choice yet fix this issue either, even introducing additional dimensions). We thus want to introduce a linear map $\rho_{k,\sigma} : \A_{k,\sigma} \rightarrow \B(\Hilbert_k)$ such that, if we set for all $a\in \dom{\gradient_{k,\sigma}}$,
\begin{equation*}
	[\Dirac_k,a]_{\rho_{k,\sigma}} \coloneqq \Dirac_k a - \rho_{k,\sigma}(a) \Dirac_k
\end{equation*}
then
\begin{equation}\label{wish-eq}
	[\Dirac_k,a]_{\rho_{k,\sigma}} = \gradient_{k,\sigma}a \text.
\end{equation}
Comparing Expression \eqref{actual-eq} with Expression \eqref{wish-eq}, we see that we seek to solve for $\rho_{k,\sigma}(a)$ in the equation:
\begin{align*}
	\underbracket[1pt]{( \rho_{k,\sigma}(a) - a ) \Dirac_k}_{ = [\Dirac_k,a] - [\Dirac_k,a]_{\rho_{k,\sigma}}} 
	&=\sum_{k_j<\infty}\left(
	\fpartial_{k,\sigma}^j(a)(u_k^j-1)\otimes\car_j
	-
	\bpartial_{k,\sigma}^j(a)((u_k^j)^*-1)\otimes\car_j^*
	\right)\text.
\end{align*}
For $k\neq\infty$, this is easily solved since $\Dirac_k$ is invertible, and we thus obtain that for all $a \in \dom{\gradient_{k,\sigma}}$, if we set:
\begin{equation}\label{car-twist-eq}
	\rho_{k,\sigma}(a) \coloneqq a+E_{k,\sigma}(a)\Dirac_k^{-1},
\end{equation}
where
\begin{equation}\label{car-error-eq}
	E_{k,\sigma}(a)\coloneqq\sum_{k_j<\infty}\left(
	\fpartial_{k,\sigma}^j(a)(u_k^j-1)\otimes\car_j - \bpartial_{k,\sigma}^j(a)((u_k^j)^*-1)\otimes\car_j^* \right) \text,
\end{equation}
then the desired Expression (\ref{wish-eq}) holds. For $k=\infty$, we set $\rho_{\infty,\sigma}=\mathrm{id}_{\A_{\infty,\sigma}}$, and Expression \eqref{wish-eq} holds without using an inverse.
\begin{remark}
	Thus defined, our twist $\rho_{k,\sigma}$ is actually a bounded map. We will develop our framework for the spectral propinquity for twisted spectral triples, however, to include the possibility of certain types of unbounded twists, for future applications.
\end{remark}

The factorization in Expression \eqref{car-error-eq} is also the key to controlling the twist.
\begin{lemma}\label{car-riesz-lemma}
	If $k_j<\infty$, then
	\begin{equation*}
		\max\left\{
		\opnorm{((u_k^j-1)\otimes\car_j)\Dirac_k^{-1}}{}{\Hilbert_k},
		\opnorm{(((u_k^j)^*-1)\otimes\car_j^*)\Dirac_k^{-1}}{}{\Hilbert_k}
		\right\}
		\leq\frac{4\pi}{k_j}.
	\end{equation*}
\end{lemma}

\begin{proof}
	For every $m\in\Z_k^d$, set
	\begin{equation*}
		r_k(m)
		\coloneqq
		\left(
			\sum_{k_l=\infty}m_l^2
			+
			\sum_{k_l<\infty}
			\left(
				\frac{k_l}{\pi}
				\sin\left(\frac{\pi m_l}{k_l}\right)
			\right)^2
		\right)^{\frac{1}{2}}\text.
	\end{equation*}
	By definition, $\Dirac_k^0$ leaves 	$\C e_m\otimes\Spinor_d$ invariant, and Expression \eqref{dirac-symbol-eq} gives
	\begin{equation*}
		(\Dirac_k^0)^2(e_m\otimes\zeta) = r_k(m)^2e_m\otimes\zeta
	\end{equation*}
	for every $\zeta\in\Spinor_d$. Since $\Dirac_k=\Dirac_k^0+\varepsilon_k1$ is invertible and $\C e_m\otimes\Spinor_d$ is finite-dimensional, we have
	\begin{equation*}
		\Dirac_k\left(\C e_m\otimes\Spinor_d\right) = \C e_m\otimes\Spinor_d =
		\Dirac_k^{-1}\left(\C e_m\otimes\Spinor_d\right).
	\end{equation*}

	For every $l\in\{1,\ldots,d\}$, the $l$th summand in the definition of $r_k(m)^2$ vanishes if, and only if, $m_l=0$. Consequently,
	\begin{equation*}
		r_k(m)=0 \quad\Longleftrightarrow\quad m=0\text.
	\end{equation*}
	
	Now assume that $m\neq0$, and choose $l$ such that $m_l\neq0$. If $k_l=\infty$, then
	\begin{equation*}
		r_k(m)\geq |m_l|\geq1.
	\end{equation*}
	If $k_l<\infty$, then
	\begin{equation*}
		r_k(m) \geq
		\frac{k_l}{\pi} \left|\sin\left(\frac{\pi m_l}{k_l}\right) \right|
		\geq \frac{k_l}{\pi} \sin\left(\frac{\pi}{k_l}\right) \geq \frac{2}{\pi}
	\end{equation*}
	where the last inequality follows from Jordan's inequality \cite[Sec. 2.3, eq. 1]{Mitrinovic70}:
	\begin{equation*}
		\forall x \in \left[0,\frac{\pi}{2}\right] \qquad\frac{2}{\pi} x \leq \sin(x) \Bigg( \leq x \Bigg)
	\end{equation*}
	which can be checked directly by calculus, or since $\sin$ is concave on this interval and thus above its secant from $(0,0)$ to $(\pi/2,1)$.
	
	Thus
	\begin{equation*}
		r_k(m)\geq\frac{2}{\pi}
	\end{equation*}
	whenever $m\neq0$.

	We continue working with $m\neq 0$. For every $\eta\in\C e_m\otimes\Spinor_d$, the self-adjointness of $\Dirac_k^0$ gives
\begin{equation*}
	\norm{\Dirac_k^0\eta}{}^2 = \inner{\eta}{(\Dirac_k^0)^2\eta}{}
	= r_k(m)^2\norm{\eta}{}^2\text.
\end{equation*}
Consequently,
\begin{align*}
	\norm{\Dirac_k\eta}{}
	&\geq \norm{\Dirac_k^0\eta}{} - \varepsilon_k\norm{\eta}{}\\
	&= (r_k(m)-\varepsilon_k)\norm{\eta}{}\\
	&\geq \frac{r_k(m)}2\norm{\eta}{}\text{ since }\varepsilon_k \in \left(0,\frac{1}{\pi}\right) \subseteq \left(0,\frac{r_k(m)}{2}\right)\text.
\end{align*}
Applying this inequality to
$\eta=\Dirac_k^{-1}(e_m\otimes\zeta)$ yields
\begin{equation}\label{car-riesz-proof-inverse-eq}
	\norm{\Dirac_k^{-1}(e_m\otimes\zeta)}{} \leq \frac{2}{r_k(m)}\norm{\zeta}{}\text.
\end{equation}

	Fix $j$ such that $k_j<\infty$. For every $m\in\Z_k^d$, we
	have
	\begin{equation*}
		u_k^j e_m = \exp\left(\frac{2\pi i m_j}{k_j}\right)e_m
	\end{equation*}
	and therefore
	\begin{align*}
		\left| \exp\left(\frac{2\pi i m_j}{k_j}\right)-1 \right|
		&=	2\left|\sin\left(\frac{\pi m_j}{k_j}\right)\right|\\
		&\leq\frac{2\pi}{k_j}r_k(m).
	\end{align*}
	If $m\neq0$, then, for every $\zeta\in\Spinor_d$, Expression \eqref{car-riesz-proof-inverse-eq} and $\opnorm{\car_j}{}{\Spinor_d}=1$ give
	\begin{align*}
		&
		\norm{
			\left(
				((u_k^j-1)\otimes\car_j)\Dirac_k^{-1}
			\right)(e_m\otimes\zeta)
		}{}\\
		&\qquad\leq
		\left|
			\exp\left(\frac{2\pi i m_j}{k_j}\right)-1
		\right|
		\norm{\Dirac_k^{-1}(e_m\otimes\zeta)}{}\\
		&\qquad\leq
		\frac{2\pi}{k_j}r_k(m)
		\frac{2}{r_k(m)}
		\norm{\zeta}{}\\
		&\qquad=
		\frac{4\pi}{k_j}\norm{\zeta}{}.
	\end{align*}
	If $m=0$, then $(u_k^j-1)e_m=0$, so the same estimate holds.

	Moreover,
	\begin{equation*}
		(u_k^j)^*e_m = \exp\left(-\frac{2\pi i m_j}{k_j}\right)e_m
	\end{equation*}
	and
	\begin{equation*}
		\left|\exp\left(-\frac{2\pi i m_j}{k_j}\right)-1\right|
		= \left|\exp\left(\frac{2\pi i m_j}{k_j}\right)-1\right|\text.
	\end{equation*}
	Since $\opnorm{\car_j^*}{}{\Spinor_d}=1$, the preceding	computation also gives
	\begin{equation*}
		\norm{
			\left(
				(((u_k^j)^*-1)\otimes\car_j^*)\Dirac_k^{-1}
			\right)(e_m\otimes\zeta)
		}{}
		\leq \frac{4\pi}{k_j}\norm{\zeta}{}
	\end{equation*}
	for every $m\in\Z_k^d$ and $\zeta\in\Spinor_d$.

	Both operators leave each subspace $\C e_m\otimes\Spinor_d$ invariant. Taking the supremum of their norms on these mutually orthogonal subspaces proves the result.
\end{proof}

\begin{corollary}\label{car-twist-estimate-cor}
	For all $a\in\dom{\Lip_{k,\sigma}}$,
	\begin{equation}\label{car-twist-estimate-eq}
		\opnorm{\rho_{k,\sigma}(a)-a}{}{\Hilbert_k}
		\leq16\pi\left(\sum_{k_j<\infty}\frac1{k_j}\right)\Lip_{k,\sigma}(a)
		\leq16\pi d\,\Lip_{k,\sigma}(a).
	\end{equation}
\end{corollary}

\begin{proof}
	If $k=(\infty,\ldots,\infty)$, then $\rho_{k,\sigma}=\mathrm{id}_{\A_{k,\sigma}}$, so the result is immediate. Assume that $k\neq(\infty,\ldots,\infty)$, and first let $a\in\dom{\gradient_{k,\sigma}}$. Expressions \eqref{car-twist-eq} and \eqref{car-error-eq} give
	\begin{align*}
		\opnorm{\rho_{k,\sigma}(a)-a}{}{\Hilbert_k}
		&\leq \sum_{k_j<\infty} \norm{\fpartial_{k,\sigma}^j(a)}{\A_{k,\sigma}}
		\opnorm{((u_k^j-1)\otimes\car_j)\Dirac_k^{-1}}{}{\Hilbert_k}\\
		&\quad+\sum_{k_j<\infty}\norm{\bpartial_{k,\sigma}^j(a)}{\A_{k,\sigma}}
		\opnorm{(((u_k^j)^*-1)\otimes\car_j^*)\Dirac_k^{-1}}{}{\Hilbert_k}\\
		&\leq	4\pi\sum_{k_j<\infty} \frac{1}{k_j}
			\left(\norm{\fpartial_{k,\sigma}^j(a)}{\A_{k,\sigma}} + \norm{\bpartial_{k,\sigma}^j(a)}{\A_{k,\sigma}}\right)\\
		&\leq 16\pi
			\left(\sum_{k_j<\infty}\frac{1}{k_j}\right)\norm{\gradient_{k,\sigma}a}{\A_{k,\sigma}\otimes\operatorname{CAR}_d}\\
		&=	16\pi \left(	\sum_{k_j<\infty}\frac{1}{k_j} \right) \Lip_{k,\sigma}(a)\text,
	\end{align*}
	where the second inequality follows from Lemma \ref{car-riesz-lemma}, and the third follows from Expression \eqref{car-component-bound-eq}.

	Now let $a\in\dom{\Lip_{k,\sigma}}$. By the definition of $\Lip_{k,\sigma}$, there exists a sequence $(a_n)_{n\in\N}$ in $\dom{\gradient_{k,\sigma}}$ such that
	\begin{equation*}
		\lim_{n\to\infty}\norm{a_n-a}{\A_{k,\sigma}}=0
		\qquad\text{and}\qquad
		\limsup_{n\to\infty}\norm{\gradient_{k,\sigma}a_n}{\A_{k,\sigma}\otimes\operatorname{CAR}_d } \leq \Lip_{k,\sigma}(a) \text.
	\end{equation*}
	Since $\rho_{k,\sigma}$ is continuous,
	\begin{align*}
		\opnorm{\rho_{k,\sigma}(a)-a}{}{\Hilbert_k}
		&=	\lim_{n\to\infty} \opnorm{\rho_{k,\sigma}(a_n)-a_n}{}{\Hilbert_k}\\
		&\leq 16\pi \left(\sum_{k_j<\infty}\frac{1}{k_j}\right) \Lip_{k,\sigma}(a) \text.
	\end{align*}
	Finally, since $1/k_j\leq1$ whenever $k_j<\infty$,
	\begin{equation*}
		\sum_{k_j<\infty}\frac{1}{k_j}\leq d \text,
	\end{equation*}
	which proves the second inequality.
\end{proof}

We fix from now on
\begin{equation}\label{kappa-d-eq}
	\kappa_d\coloneqq\frac1{16\pi d}.
\end{equation}
Thus Expression \eqref{car-twist-estimate-eq} gives $\kappa_d\opnorm{\rho_{k,\sigma}(a)-a}{}{\Hilbert_k}\leq\Lip_{k,\sigma}(a)$.

Therefore, our twisted spectral triple for a fuzzy torus in this paper is:
\begin{equation*}
	\left( \A_{k,\sigma}, \Hilbert_k , \Dirac_k, \rho_{k,\sigma} \right) \text,
\end{equation*}
where, as before, we denote the Hilbert space $\ell^2(\Z^d_k)\otimes\Spinor_d$ by $\Hilbert_k$. 

While $\opnorm{[\Dirac_k,a]_{\rho_{k,\sigma}}}{}{\Hilbert_k} = \norm{\gradient_{k,\sigma} a}{\A_{k,\sigma}\otimes\operatorname{CAR}_d}$ by construction for all $a\in \dom{\gradient_{k,\sigma}}$, we actually now prove that our twisted spectral triple recovers the L-seminorm $\Lip_{k,\sigma}$. In the process, we establish some helpful results about this spectral triple. 

We will use the following construction repeatedly in this paper.

\begin{notation}
	Let $k \in \Pbar^d$. We endow $\U_k^d$ with its Haar probability measure $\mu_k$. If $f \in L^1(\T^d,\mu_\infty)$, then we define, for any $2$-cocycle $\sigma$ of $\Z^d_k$, the continuous linear endomorphism $\alpha_{k,\sigma}^f$ of $\A_{k,\sigma}$ by:
	\begin{equation*}
		\alpha_{k,\sigma}^f (a) \coloneqq \int_{\U_k^d} f(z) \alpha_{k,\sigma}^z(a) \, d\mu_k(z) \text.
	\end{equation*}
	Note that $\opnorm{\alpha_{k,\sigma}^f}{}{\A_{k,\sigma}} \leq \norm{f}{L^1(\U_k^d,\mu_k)}$.
\end{notation}

We will use the following notation until it becomes redundant at the end of this subsection, for clarity.
\begin{notation}
	For all $(k,\sigma) \in \Xi$, we denote by $\A_{k,\sigma}(\Dirac_k)$ the set:
	\begin{equation*}
		\left\{ a \in \A_{k,\sigma} : a \dom{\Dirac_k} \subseteq \dom{\Dirac_k} \text{ and }[\Dirac_k,a]_{\rho_{k,\sigma}} \text{ is bounded } \right\} \text.
	\end{equation*}
	We note that $\dom{\gradient_{k,\sigma}} \subseteq \A_{k,\sigma}(\Dirac_k)$.
	
	We also define the seminorm:
	\begin{equation*}
		D_{k,\sigma}(a) \coloneqq
		\opnorm{[\Dirac_k,a]_{\rho_{k,\sigma}}}{}{\Hilbert_k}
	\end{equation*}
	and, as per our usual convention, we set $D_{k,\sigma}(a) = \infty$ whenever $a\in \A_{k,\sigma} \setminus \A_{k,\sigma}(\Dirac_k)$.
\end{notation}

	\begin{lemma}\label{dual-iso-lemma}
		For all $(k,\sigma) \in \Xi$, for all $a \in \A_{k,\sigma}(\Dirac_k)$, and for all $z\in \U^d_k$, we observe:
		\begin{equation*}
			D_{k,\sigma}(\alpha_{k,\sigma}^z(a))=D_{k,\sigma}(a) \text.
		\end{equation*}
	\end{lemma}
	
	\begin{proof}
	We simplify our notation for this proof by setting $\rho \coloneqq \rho_{k,\sigma}$.
	
	For all $z\in\U^d_k$, we note that of course, since $\U^d_k$ is Abelian, the unitary $u_k^z$ commutes with $\nabla_k^j$, as well as $u_k^j$, and thus with $\Dirac_k$ and its inverse; thus
	\begin{equation*}
		\rho_{k,\sigma}(u_k^z a u_k^{-z}) = u_k^z \rho(a) u_k^{-z}
	\end{equation*}
	and therefore, commuting again with $\Dirac_k$,
	\begin{align*}
		[\Dirac_k,\alpha_{k,\sigma}^z(a)]_\rho
		&=[\Dirac_k, u_k^z a u_k^{-z}]_\rho \\
		&=\Dirac_k u_k^z a u_k^{-z} - \rho(u_k^z a u_k^{-z})\Dirac_k \\
		&=\Dirac_k u_k^z a u_k^{-z} - u_k^z \rho(a )u_k^{-z} \Dirac_k \\
		&=u_k^z\left( \Dirac_k a - \rho(a) \Dirac_k \right) u_k^{-z} \\
		&= u_k^z [\Dirac_k,a]_\rho u_k^{-z} \text.
	\end{align*}
Therefore,
	\begin{equation*}
		D_{k,\sigma}(\alpha_{k,\sigma}^z(a))=D_{k,\sigma}(a) \text,
	\end{equation*}
	as desired.
\end{proof}

We now establish the following lemma. We will actually improve on this result in our section on the spectral propinquity once we state our general definition of a metric twisted spectral triple, but this version will suffice for the present section.
\begin{lemma}\label{lower-semi-lemma}
	For all $(k,\sigma) \in \Xi$, the seminorm $D_{k,\sigma}$ is lower semicontinuous on $\A_{k,\sigma}$.
\end{lemma}

\begin{proof}
		Let $(a_n)_{n\in\N}$ be a sequence in $\A_{k,\sigma}$ converging to $a \in \A_{k,\sigma}$, such that $a_n \in \A_{k,\sigma}(\Dirac_k)$ with 
		\begin{equation*}
			D_{k,\sigma}(a_n) \leq 1
		\end{equation*}
		for all $n\in\N$. The sequence $([\Dirac_k,a_n]_{\rho_{k,\sigma}})_{n\in\N}$ lies in the unit ball of $\B(\Hilbert_k)$, so it admits a subnet converging in the weak operator topology to some $T\in\B(\Hilbert_k)$ with $\opnorm{T}{}{\Hilbert_k}\leq1$. Since $\rho_{k,\sigma}$ is norm-continuous, for all $\xi,\eta\in\dom{\Dirac_k}$ we have
		\begin{equation*}
			\inner{T\xi}{\eta}{\Hilbert_k}
			=
			\inner{a\xi}{\Dirac_k\eta}{\Hilbert_k}
			-
			\inner{\rho_{k,\sigma}(a)\Dirac_k\xi}{\eta}{\Hilbert_k}.
		\end{equation*}
		Thus
		\begin{equation*}
			\inner{a\xi}{\Dirac_k\eta}{\Hilbert_k}
			=
			\inner{T\xi+\rho_{k,\sigma}(a)\Dirac_k\xi}{\eta}{\Hilbert_k}
		\end{equation*}
		for all $\eta\in\dom{\Dirac_k}$. Since $\Dirac_k$ is self-adjoint, this proves that $a\xi\in\dom{\Dirac_k}$ and
		\begin{equation*}
			\Dirac_k a\xi=T\xi+\rho_{k,\sigma}(a)\Dirac_k\xi.
		\end{equation*}
		Hence $[\Dirac_k,a]_{\rho_{k,\sigma}}=T$ and its norm is at most $1$. Therefore $D_{k,\sigma}(a)\leq1$, and the closed unit ball of $D_{k,\sigma}$ is closed in $\A_{k,\sigma}$.
\end{proof}

\begin{corollary}\label{convolution-L-cor}
	For all $(k,\sigma) \in \Xi$, for all $a\in\A_{k,\sigma}(\Dirac_k)$, and for all $f \in L^1(\U^d_k,\mu_k)$, the following inequality holds:
	\begin{equation*}
		D_{k,\sigma}(\alpha_{k,\sigma}^f(a))
		\leq \norm{f}{L^1(\U^d_k,\mu_k)}D_{k,\sigma}(a)\text.
	\end{equation*}
\end{corollary}

\begin{proof}
		Let $a \in \A_{k,\sigma}(\Dirac_k)$. Let $f \in L^1(\U_k^d,\mu_k)$. Since $D_{k,\sigma}$ is lower semicontinuous on $\A_{k,\sigma}$:
		\begin{align*}
		D_{k,\sigma}(\alpha_{k,\sigma}^f(a))
		&\leq \int_{\U^d_k}|f(z)|D_{k,\sigma}(\alpha_{k,\sigma}^z(a))\,d\mu_k(z)\\
		&=D_{k,\sigma}(a)\int_{\U^d_k}|f(z)|\,d\mu_k(z)\\
		&=D_{k,\sigma}(a)\norm{f}{L^1(\U^d_k,\mu_k)}\text,
		\end{align*}
		as desired.
\end{proof}

We conclude with our explicit description of our L-seminorms.
\begin{theorem}\label{L-explicit-thm}
For all $(k,\sigma) \in \Xi$, we have\; $\A_{k,\sigma}(\Dirac_k) = \dom{\Lip_{k,\sigma}}$, and for all $a \in \dom{\Lip_{k,\sigma}}$,
\begin{equation*}
	\Lip_{k,\sigma}(a)
	=\opnorm{[\Dirac_k,a]_{\rho_{k,\sigma}}}{}{\Hilbert_k}\text.
\end{equation*} 
\end{theorem}

\begin{proof}
	First, note that if $a \in \dom{\gradient_{k,\sigma}}$, then by construction,
	\begin{equation*}
		D_{k,\sigma}(a)=\opnorm{[\Dirac_k,a]_{\rho_{k,\sigma}}}{}{\Hilbert_k} = \norm{\gradient_{k,\sigma} a}{\A_{k,\sigma}\otimes\operatorname{CAR}_d}.
	\end{equation*}
	Thus
	\begin{multline*}
		\left\{ a \in \dom{\gradient_{k,\sigma}} : \norm{\gradient_{k,\sigma} a}{\A_{k,\sigma}\otimes\operatorname{CAR}_d} \leq 1 \right\} \\
		= \left\{ a \in \dom{\gradient_{k,\sigma}} : D_{k,\sigma}(a) \leq 1 \right\} \text.
	\end{multline*}
	Therefore, by Lemma \ref{lower-semi-lemma} and since $\dom{\gradient_{k,\sigma}} \subseteq \A_{k,\sigma}(\Dirac_k)$, we conclude:
	\begin{align*}
	&\closure{\left\{ a \in \dom{\gradient_{k,\sigma}} : \norm{\gradient_{k,\sigma} a}{\A_{k,\sigma}\otimes\operatorname{CAR}_d} \leq 1 \right\}} \\
	&\quad = \closure{\left\{ a \in \dom{\gradient_{k,\sigma}} : D_{k,\sigma}(a) \leq 1 \right\}}\\
	&\quad \subseteq \closure{\left\{ a \in \A_{k,\sigma}(\Dirac_k) : D_{k,\sigma}(a) \leq 1 \right\}}\\
	&\quad =\left\{ a \in \A_{k,\sigma}(\Dirac_k) : D_{k,\sigma}(a) \leq 1 \right\} \text.
	\end{align*}
	
	Conversely, let $a \in \left\{ a \in \A_{k,\sigma}(\Dirac_k) : D_{k,\sigma}(a) \leq 1 \right\}$. Let $(f_n)_{n\in\N}$ be an approximate unit in $L^1(\U^d_k,\mu_k)$ consisting of functions valued in $[0,\infty)$ and with norm $1$ whose Fourier transforms  have finite support $S_n \subseteq \Z^d_k$ (of course, the support is not the same for each $n \in \N$). Such a sequence can be obtained for instance by applying \cite[Corollary 3.1]{Latremoliere05} to any length function on $\U^d_k$.  By \cite[Lemma 3.2]{Latremoliere05}, for each $n \in \N$, we assert that $\alpha_{k,\sigma}^{f_n}(a) \in \A_{k,\sigma}(S_n)$, and therefore, $\alpha_{k,\sigma}^{f_n}(a) \in \dom{\gradient_{k,\sigma}}$ for all $n\in\N$.
	
	By Corollary \ref{convolution-L-cor}, we then conclude that for all $n \in \N$:
	\begin{align*}
		\norm{\gradient_{k,\sigma} \alpha_{k,\sigma}^{f_n}(a)}{\A_{k,\sigma}\otimes\operatorname{CAR}_d} 
		&= D_{k,\sigma}(\alpha_{k,\sigma}^{f_n}(a)) \\
		&\leq \norm{f_n}{L^1(\U^d_k,\mu_k)} D_{k,\sigma}(a) \leq 1 \text.
	\end{align*}
	Moreover, $a = \lim_{n\to\infty} \alpha_{k,\sigma}^{f_n}(a)$. So 
	\begin{equation*}
		a \in \closure{\left\{ a \in \dom{\gradient_{k,\sigma}} : \norm{\gradient_{k,\sigma} a}{\A_{k,\sigma}\otimes\operatorname{CAR}_d} \leq 1 \right\}}\text.
	\end{equation*}
	
	Thus,
	\begin{multline}\label{L-expression-eq}
		\closure{\left\{ a \in \dom{\gradient_{k,\sigma}} : \norm{\gradient_{k,\sigma} a}{\A_{k,\sigma}\otimes\operatorname{CAR}_d} \leq 1 \right\}} \\ 
		= \left\{ a \in \A_{k,\sigma}(\Dirac_k) : D_{k,\sigma}(a) \leq 1 \right\} \text,
	\end{multline}
	which, in turn, implies that the Minkowski gauge functional $\Lip_{k,\sigma}$ of the closed balanced convex set in Expression \eqref{L-expression-eq} is indeed given by $D_{k,\sigma}$, as desired.
\end{proof}

\begin{remark}
	We note the advantage of working with a Hilbert space here. The operator $\gradient_{k,\sigma}$ is closed for the C*-algebra norms, but its closed unit ball is not closed in norm. However, the operator $\Dirac_k$ is closed in the Hilbert norm instead, and as a result, we capture a larger set of regular elements than with $\gradient_{k,\sigma}$ --- namely, Lipschitz elements rather than merely $C^1$ --- and the closed unit ball for the induced seminorm is indeed closed in the C*-norm.
\end{remark}

We thus have reached the following point in our tale: starting with the obvious discretized calculus on fuzzy tori induced by the dual action, we constructed a {\qcms} using as L-seminorm the Minkowski functional for the closure of the set of elements of discrete gradients with norm at most $1$; this closure is required to obtain a lower semicontinuous L-seminorm. We then implemented this calculus using a twisted spectral triple involving the natural discrete analogue of the flat Dirac operator, and we then proved that indeed, this twisted spectral triple recovers our L-seminorm as the norm of the twisted commutator, with no closure needed any more.

Our purpose in this paper is to indeed prove that, in the sense of the spectral propinquity for twisted spectral triples, these fuzzy tori do indeed converge nicely to the flat spectral triple on $\T^d$, completing our present project.

\bigskip

Our construction has a useful interpretation in terms of discrete
Riesz transforms. Classically, the Riesz transforms
$R_j=\partial_j(-\Delta)^{-1/2}$ normalize the first-order
derivatives by the square root of the positive Laplacian.
Their discrete counterparts replace derivatives by translation
differences: if $\nabla_j$ are these differences and
$\Delta=\sum_j\nabla_j^\ast\nabla_j$, then
$R_j=\nabla_j\Delta^{-1/2}$, where the inverse is taken to
vanish on the kernel; see \cite{LustPiquard04,JungeMeiParcet18}.
The corresponding operators in our construction use the inverse
Dirac operator instead of the inverse square root of the
Laplacian. Since $(D_k^0)^2=\Delta_k\otimes 1$, they are
Dirac-normalized variants of discrete Riesz transforms, with
Clifford factors and the scalar regularization introduced above.

The property relevant here is that this normalization produces
bounded operators, with bounds uniform in the finite
approximations. More precisely, Lemma~2.18 shows that
\[
 \left\|((u_k^j-1)\otimes c_j)D_k^{-1}\right\|
 \leq \frac{4\pi}{k_j},
\]
with the analogous estimate for the adjoint terms. These are
the operators appearing in the correction
$\rho_{k,\sigma}(a)-a$. Consequently, Corollary~2.19 gives
\[
 \|\rho_{k,\sigma}(a)-a\|
 \leq
 16\pi\left(\sum_{k_j<\infty}\frac{1}{k_j}\right)
 L_{k,\sigma}(a).
\]
Thus, the normalized difference operators both control the
correction needed to recover the discrete gradient as a twisted
commutator and show that this correction vanishes uniformly on
sets with bounded Lipschitz seminorm as the mesh sizes tend to
zero.


Our twist, as we shall see formally at the end of this article, is indeed a perturbation of the identity by a term which, when applied to functions $f$ over a fixed finite set of $\Z^d$ and with bounded L-seminorms, indeed converges to $0$ uniformly as $k$ goes to $(\infty,\ldots,\infty)$. 

We note that our efforts in this section were centered on keeping the L-seminorm, i.e. the quantum metric data, as the norm of the discretized gradient on fuzzy tori, and adjusting the spectral triple instead. Indeed, using the untwisted spectral triple actually leads to exactly one issue for us: the lack of continuity of the quantum metric. We thus begin our discussion of convergence with the metric component.

\section{Convergence for the Propinquity}

We begin with the proof that, if we ignore for now the more complicated differential structure, and focus on the metric aspects only, then for any normalized $2$-cocycle $\sigma$ on $\Z^d_k$,
\begin{equation}\label{goal-prop-limit-eq}
	\lim_{\substack{(k,\varsigma)\to(\infty,\sigma) \\ (k,\varsigma) \in \Xi}} \dpropinquity{4d,0}((\A_{k,\varsigma},\Lip_{k,\varsigma}),(\A_{\infty,\sigma},\Lip_{\infty,\sigma})) = 0 \text,
\end{equation}
where $\dpropinquity{4d,0}$ is the propinquity constructed in \cite{Latremoliere13b,Latremoliere14,Latremoliere15}. We only briefly recall \cite{Latremoliere13b} that a $(4d,0)$-tunnel between two $(4d,0)$-{\qcms s} $(\A,\Lip_\A)$ and $(\B,\Lip_\B)$ is a quadruple $(\D,\Lip_\D,\pi_\A,\pi_\B)$ where $(\D,\Lip_\D)$ is a $(4d,0)$-{\qcms}, and $\pi_\A : (\D,\Lip_\D)\to(\A,\Lip_\A)$ and $\pi_\B:(\D,\Lip_\D)\to(\B,\Lip_\B)$ are two quantum isometries. We also recall \cite{Latremoliere13} that a qunatum isometry $\pi : (\D,\Lip_\D)\to(\A,\Lip_\A)$ between two {\qcms s} $(\D,\Lip_\D)$ and $(\A,\Lip_\A)$ is a surjection such that $\pi(\dom{\Lip_\D}) = \dom{\Lip_\A}$, and $\Lip_\A(a) = \inf\{\Lip_\D(d) : d\in\dom{\Lip_\D},\pi(d) = a\}$ for all $a\in\dom{\Lip_\A}$.

The \emph{extent} of such a tunnel $\tau\coloneqq(\D,\Lip_\D,\pi_\A,\pi_\B)$ is defined as the number
\begin{multline*}
	\tunnelextent{\tau}\coloneqq \max\Big\{ \Haus{\Kantorovich{\Lip_\D}}\left( \left\{ \varphi\circ\pi_\A:\varphi\in\StateSpace(\A)\right\},\StateSpace(\D)\right), \\
	\Haus{\Kantorovich{\Lip_\D}}\left( \left\{ \varphi\circ\pi_\B:\varphi\in\StateSpace(\B)\right\},\StateSpace(\D)\right) \Big\}
\end{multline*}
where we use the following notation.
\begin{notation}
If $(E,d)$ is a metric space, the Hausdorff distance induced on the hyperspace of all closed subsets of $E$ is denoted by $\Haus{d}$ \cite{Hausdorff}. If no confusion may arise, we also may denote this Hausdorff metric by $\Haus{E}$.
\end{notation}

The \emph{$(4d,0)$-propinquity} $\dpropinquity{4d,0}((\A,\Lip_\A),(\B,\Lip_\B))$ between any two $(4d,0)$-{\qcms s} is then  
\begin{equation*}
	\dpropinquity{4d,0}((\A,\Lip_\A),(\B,\Lip_\B)) \coloneqq \inf\left\{ \tunnelextent{\tau} : \tau \text{ is a $(4d,0)$-tunnel from }(\A,\Lip_\A)\text{ to }(\B,\Lip_\B) \right\} \text.
\end{equation*}
We thus define a complete metric up to full quantum isometry on the class of $(4d,0)$-{\qcms s}, as seen in \cite{Latremoliere13b,Latremoliere14,Latremoliere15}, and many examples of convergence for this metric are known. This includes all examples of convergence for the Gromov-Hausdorff distance \cite{Gromov} since  the propinquity, restricted to classical compact metric spaces $(X,d)$ identified with the {\qcms s} of the form $(C(X),\Lip_d)$ (where $\Lip_d$ is the Lipschitz seminorm induced by $d$), is topologically equivalent to the Gromov-Hausdorff distance.

In addition of the intrinsic value of proving a new convergence result for the propinquity, this section also lays the foundation for our work in the next sections. Our proof heavily borrows from our work on quantum and fuzzy tori in \cite{Latremoliere13c} and \cite{Latremoliere20a}, and we will refer to these work when appropriate.

\begin{remark}
	We could easily extend the results in this paper to limits at any point in $\Xi$, but we prefer to focus on the core example which includes all the important techniques.
\end{remark}

Let $(k,\sigma) \in \Xi$. For any $(z_1,\ldots,z_d) \in \U^d_k$, the *-automorphism $\alpha_{k,\sigma}^z$ extends to the *-automorphism $\alpha_{k,\sigma}^z\otimes \mathrm{id}_{\operatorname{CAR}_d}$ on $\A_{k,\sigma}\otimes\operatorname{CAR}_d$. By abuse of notation, we denote this extension still by $\alpha_{k,\sigma}^z$.

\subsection{Gauge Actions act by Isometries, and a Mean Value Theorem}

The proof of Equation (\ref{goal-prop-limit-eq}) follows the main argument of \cite{Latremoliere13c,Latremoliere21a}. In fact, it suffices to simply establish a few key properties, which we can then use and plug in our argument in \cite{Latremoliere13c} to get the stated continuity. We already showed that the dual action acts on $\A_{h,\varsigma}$ by quantum isometries for all $(h,\varsigma) \in \Xi$ in Lemma \ref{dual-iso-lemma}.

The main tool in \cite{Latremoliere05,Latremoliere13c} relate convolution with Fejer kernels and the metric properties.

\begin{notation}
	For any $z \in \T$, we define:
	\begin{equation*}
		\mathsf{len}(z) \coloneqq \min\{ |t| : t \in \R, z = \exp(it) \} \text. 
	\end{equation*}
	Note that $\mathsf{len}$ is a length function on the circle $\T$, and induces a length function by restriction on $\U^d_k$ for all $k\in\Pbar$. The associated distance function induces the usual topology on $\T$.
	
	We then define, for all $(z_1,\ldots,z_d) \in \T^d$:
	\begin{equation*}
		\mathsf{slen}(z_1,\ldots,z_d) \coloneqq \sum_{j=1}^d \mathsf{len}(z_j) \text.
	\end{equation*}
	The function $\mathsf{slen}$ is a length function on $\T^d$, inducing its usual topology. Note that it is \emph{not} the length function induced by the Riemannian metric on $\T^d$ which gives rise to $\Dirac_{\infty,\ldots,\infty}$.
\end{notation}

\begin{lemma}\label{compression-lemma}
	For all $(k,\sigma) \in \Xi$, and for all $\varepsilon > 0$, there exists an open neighborhood $\Omega$ of $(k,\sigma)$ in $\Xi$, a nonempty finite subset $S\subseteq \Z^d_k$, and a function $f \colon \T \to [0,\infty)$ such that, for all $(k,\sigma) \in \Xi$:
	\begin{enumerate}
		\item the range of $\alpha_{k,\sigma}^f$ is the subspace $\A_{k,\sigma}(S)$ of $\A_{k,\sigma}$,
		\item for all $a\in \dom{\Lip_{k,\sigma}}$, we have $\norm{a-\alpha_{k,\sigma}^f(a)}{\A_{k,\sigma}} \leq \varepsilon \Lip_{k,\sigma}(a)$.
	\end{enumerate}
\end{lemma}

\begin{proof}
Let $(k,\sigma) \in \Xi$. Let $z  \in \U_k$. In a useful departure from our usual notation, we write $z = (y_1,\ldots,y_d)$, and set, for each $j \in \{1,\ldots,d\}$:
\begin{equation*}
	z_j \coloneqq (j\gets y_j) =\big(1 ,\ldots, 1, \underbracket{y_j}_{\text{index }j}, 1, \ldots, 1 \big)\text.
\end{equation*}

For all $a \in \A_{k,\sigma}$, noting that $z = \prod_{j=1}^d z_j$, we then have:
\begin{align*}
	a - \alpha_{k,\sigma}^z(a)
	&= a - \alpha_{k,\sigma}^{z_1}\alpha_{k,\sigma}^{z_2} \cdots \alpha_{k,\sigma}^{z_d}(a) \\
	&= a - \alpha_{k,\sigma}^{z_1}(a) + \alpha_{k,\sigma}^{z_1}(a - \alpha_{k,\sigma}^{z_2}(a) ) + \alpha_{k,\sigma}^{z_1}\alpha_{k,\sigma}^{z_2}(a - \alpha_{k,\sigma}^{z_3}(a)) + \cdots \\
	&\quad  - \alpha_{k,\sigma}^{z_1}\alpha_{k,\sigma}^{z_2} \cdots \alpha_{k,\sigma}^{z_d}(a) \\
	&= a - \alpha_{k,\sigma}^{z_1}(a) + \sum_{j=2}^d \alpha_{k,\sigma}^{z_1} \cdots \alpha_{k,\sigma}^{z_{j-1}}(a - \alpha_{k,\sigma}^{z_j}(a)) \text.
\end{align*}
Therefore,
\begin{align*}
	\norm{a - \alpha_{k,\sigma}^z(a)}{\A_{k,\sigma}}
	&\leq \norm{a-\alpha_{k,\sigma}^{z_1}(a)}{\A_{k,\sigma}} + \sum_{j=2}^d \norm{\alpha_{k,\sigma}^{z_1}\cdots\alpha_{k,\sigma}^{z_{j-1}}(a-\alpha_{k,\sigma}^{z_j}(a))}{\A_{k,\sigma}} \\
	&=\sum_{j=1}^d \norm{\alpha_{k,\sigma}^{z_j}(a) - a}{\A_{k,\sigma}} \text{ since $\alpha_{k,\sigma}$ acts by *-automorphisms.}
\end{align*}

We now consider each term above separately. First, it will be helpful to note that if $a \in \dom{\gradient_{k,\sigma}}$, then for each $j \in \{1,\ldots,d\}$,
\begin{align*}
	\max\left\{\norm{\fpartial_{k,\sigma}^j a}{\A_{k,\sigma}},\norm{\bpartial_{k,\sigma}^j a}{\A_{k,\sigma}}\right\}
	&\leq2\norm{\gradient_{k,\sigma} a}{\A_{k,\sigma}\otimes\operatorname{CAR}_d} \\
	&=2\Lip_{k,\sigma}(a) \text{ by Theorem \ref{L-explicit-thm}.}
\end{align*}

Fix $j \in \{1,\ldots,d\}$. If $k_j = \infty$, then writing $y_j = \exp(i s_j \mathsf{len}(y_j) )$ with $s_j \in \{-1,1\}$, we see that if $a\in\dom{\gradient_{k,\sigma}}$ then:
\begin{align*}
	\norm{a-\alpha_{k,\sigma}^{z_j}(a)}{\A_{k,\sigma}}
	&= \norm{\int_0^{s_j \mathsf{len}(y_j)} \alpha_{k,\sigma}^{(j\gets \exp(i t))} \fpartial_{k,\sigma}^j(a) \, dt}{\A_{k,\sigma}} \\
	&\leq \int_0^{\mathsf{len}(y_j)} \norm{\fpartial_{k,\sigma}^j(a)}{\A_{k,\sigma}} \, dt \\
	&\leq 2\mathsf{len}(y_j)\Lip_{k,\sigma}(a) \text.
\end{align*}
If $a \in \dom{\Lip_{k,\sigma}}$ with $\Lip_{k,\sigma}(a)\leq 1$, then there exists $(a_n)_{n\in\N}$ in $\dom{\gradient_{k,\sigma}}$ with $\lim_{n\to\infty} a_n = a$ and 
\begin{equation*}
	\Lip_{k,\sigma}(a_n) = \norm{\gradient_{k,\sigma} a_n}{\A_{k,\sigma}\otimes\operatorname{CAR}_d} \leq 1
\end{equation*}
for all $n\in\N$. Therefore:
\begin{align*}
	\norm{a-\alpha_{k,\sigma}^{z_j}(a)}{\A_{k,\sigma}}
	&= \lim_{n\to\infty} \norm{a_n -\alpha_{k,\sigma}^{z_j}(a_n)}{\A_{k,\sigma}} \\
	&\leq \liminf_{n\to\infty} \; 2\mathsf{len}(y_j) \Lip_{k,\sigma}(a_n) \leq  2\mathsf{len}(y_j) \text.
\end{align*}
By homogeneity, for all $a\in \dom{\Lip_{k,\sigma}}$,
\begin{equation*}
\norm{a-\alpha_{k,\sigma}^{z_j}(a)}{\A_{k,\sigma}}\leq2\mathsf{len}(y_j) \Lip_{k,\sigma}(a) \text.
\end{equation*}

On the other hand, if $k_j < \infty$, choose the integer $q_j$ of smallest absolute value such that $y_j=\lambda_{k,j}^{q_j}$, and set $m\coloneqq|q_j|$. Thus $\mathsf{len}(y_j)=2\pi m/k_j$. Assume first $q_j\geq0$. Let $a \in \dom{\gradient_{k,\sigma}}$. We compute:
\begin{align*}
	\norm{a-\alpha_{k,\sigma}^{z_j} (a)}{\A_{k,\sigma}}
	&\leq \sum_{n = 0}^{m-1} \norm{\alpha_{k,\sigma}^{j \gets \lambda_{k,j}^n} (a)-\alpha_{k,\sigma}^{j \gets \lambda_{k,j}^{n+1}} (a)}{\A_{k,\sigma}} \\
	&\leq \sum_{n=0}^{m-1} \norm{a-\alpha_{k,\sigma}^{j\gets\lambda_{k,j}}(a)}{\A_{k,\sigma}} \\
	&= \frac{2\pi m}{k_j}\norm{\fpartial_{k,\sigma}^j(a)}{\A_{k,\sigma}} \\
	&\leq 2\mathsf{len}(y_j)  \Lip_{k,\sigma}(a) \text.
\end{align*}
If instead $q_j<0$, then by a similar computation using the backward derivative:
\begin{align*}
\norm{a-\alpha_{k,\sigma}^{z_j} (a)}{\A_{k,\sigma}} 
&\leq \sum_{n=0}^{m-1} \norm{a-\alpha_{k,\sigma}^{j\gets\overline{\lambda_{k,j}}}(a)}{\A_{k,\sigma}}\\
&=\frac{2\pi m}{k_j}\norm{\bpartial_{k,\sigma}^j(a)}{\A_{k,\sigma}}\\
&\leq2\mathsf{len}(y_j)  \Lip_{k,\sigma}(a) \text.
\end{align*}
Either way, we can extend these inequalities to any $a\in \dom{\Lip_{k,\sigma}}$ as in the case $k_j = \infty$.

Therefore, if $\operatorname{slen}(z) \coloneqq \sum_{j=1}^d \mathsf{len}(y_j)$, we have shown that if $a\in\dom{\Lip_{k,\sigma}}$, then
\begin{equation*}
	\norm{a-\alpha_{k,\sigma}^z(a)}{\A_{k,\sigma}} \leq  2\operatorname{slen}(z) \Lip_{k,\sigma}(a) \text.
\end{equation*}

It thus follows that if $f\colon \T\to[0,\infty)$ with $\int_{\U_k^d} f = 1$, then, for all $a \in \dom{\Lip_{k,\sigma}}$:
\begin{equation*}
	\norm{a - \alpha_{k,\sigma}^f(a)}{\A_{k,\sigma}} \leq \int_{\U_k^d} f(z) \norm{a-\alpha_{k,\sigma}^z(a)}{\A_{k,\sigma}} \,d\mu_k(z) \leq 2\Lip_{k,\sigma}(a) \int_{\U_k^d} f(z) \operatorname{slen}(z)\, d\mu_k(z) \text.
\end{equation*}
By \cite[Corollary 3.1]{Latremoliere05}, we may indeed choose a function $f \colon \T \to [0,\infty)$ which is a linear combination of characters of $\Z^d$ and such that 
\begin{equation*}
	\int_{\T^d} f(z) \operatorname{slen}(z)\,d\mu(z) < \frac{\varepsilon}{4}\text{ and }\int_{\T^d} f(z) \,d\mu(z) = 1 \text.
\end{equation*}
Now, we note that in particular, $\int_{\U_k^d} f(z)\, d\mu_k (z) = 1$ as well. By \cite[Lemma 3.6]{Latremoliere05}, we also note that there exists $N$ such that if $k\geq N$ then 
\begin{equation*}
	\left|\int_{\U^d_k} f(z) \operatorname{slen}(z) \,d\mu_k(z) - \int_{\T^d} f(z) \operatorname{slen}(z) \,d\mu(z)\right| < \frac{\varepsilon}{4}\text.
\end{equation*}
Altogether, we conclude that, for $k\geq N$:
\begin{equation*}
	\norm{a - \alpha_{k,\sigma}^f(a)}{\A_{k,\sigma}} \leq \varepsilon \Lip_{k,\sigma}(a) \text.
\end{equation*}
Moreover, as discussed in \cite[Lemma 3.2]{Latremoliere05}, the element $\alpha_{k,\sigma}^f(a)$ lies in $\A_{k,\sigma}(S)$.
\end{proof}

\subsection{Continuity of the L-seminorms in the natural parameter}

The last key ingredient is the continuity of the L-seminorms as the parameter varies, in the following sense. We recall from Notation \ref{integrated-rep-notation} that if $f : S\to \C$ is a function over a finite subset $S$ of $\Z^d$, and if $(k,\sigma) \in \Xi$, then $f_{k,\sigma} \coloneqq \sum_{m \in S} f(m) W_{k,\sigma}^{q_k(m)} \in \A_{k,\sigma}(S)$, where $q_k : \Z^d\to \Z^d_k$ is the canonical surjection.

\begin{lemma}\label{continuous-field-L-lemma}
	For all finite subset $S\subseteq \Z^d$, for all $(k,\sigma) \in \Xi$ such that $q_k$ restricts to an injection on $S$, and for all $f :S\to\C$,
	\begin{equation*}
		\lim_{\substack{(y,\varsigma)\to (k,\sigma) \\ (y,\varsigma) \in \Xi }} \Lip_{y,\varsigma}(f_{y,\varsigma}) = \Lip_{k,\sigma}(f_{k,\sigma}) \text.
	\end{equation*}
\end{lemma}

\begin{proof}
	Let
	\begin{equation*}
		\mathscr{G} \coloneqq \coprod_{(k,\sigma) \in \Xi} \left(\Z^d_k \times \Z_2^{2d}\right) \text,
	\end{equation*}
	endowed with the algebraic structure of a groupoid as a bundle of discrete groups. We topologize $\mathscr{G}$ using the final topology for the quotient map:
	\begin{equation*}
		q : (k,\sigma,m,p) \in \Xi\times\Z^d\times\Z_2^{2d} \mapsto (k,\sigma,q_k(m),p) \text,
	\end{equation*}
	where $q_k : \Z^d\twoheadrightarrow \Z^d_k$ is the canonical surjection for each $k \in \Pbar^d$. We note that $q$ is an open map as well. Since $q$ is a surjection, it suffices to show that for all open subsets $U$ of $\Xi$, and for all $n \in \Z^d$ and $p \in \Z_2^{2d}$, the set $q^{-1}(q(U\times\{n\}\times\{p\}))$ is open in the product $\Xi\times\Z^d\times\Z_2^{2d}$ (as such sets form a topological basis for the product topology). 
	
	Let thus $(k,\sigma,m,p) \in q^{-1}(q(U\times\{n\}\times\{p\}))$. Write $n = (n_1,\ldots,n_d)$. By definition, $(k,\sigma) \in U$, which is open in $\Xi$. In particular, there exists an open neighborhood $U_1$ of $k$ in $\Pbar^d$ and an open neighborhood $U_2$ of $\sigma$ in $\mathcal{C}_2(\Z^d)$ such that $U_1\times U_2 \subseteq U$. In turn, there exist open neighborhood $U_{1,1}$, \ldots, $U_{1,d}$ of, respectively, $k_1$, \ldots, $k_d$ in $\Pbar$ such that $U_{1,1}\times \cdots \times U_{1,d} \subseteq U_1$. Now, for each $j \in \{1,\ldots,d\}$ such that $k_j < \infty$, we simply restrict $U_{1,j}$ to $\{ k_j \}$, which indeed, is an open neighborhood of $k_j$ in $\Pbar$.
	 
	 We now set:
	\begin{multline*}
		V = \Bigg\{ ( y , \varsigma , m , p ) \in U \times \Z^d \times \{p\} : \eta \in U_2\text, \\
		 y = (y_1,\ldots,y_d), m = (m_1,\ldots,m_d), \text{ such that } \\
		 \begin{cases}
		 	y_j = k_j, m_j \in n_j + k_j \Z \text{ if }k_j < \infty \text, \\
		 	y_j  \in U_{1,j} , m_j = n_j \text{ if } k_j = \infty 
		\end{cases} \Bigg\}\text.
	\end{multline*} 
	 By construction, $(k,\sigma,m,p) \in V$. Moreover, if $(y,\varsigma,m,p) \in V$, then 
	 \begin{equation*}
	 	q(y,\varsigma,m,p) = (y,\varsigma,q_y(m),p) = (y,\varsigma,q_y(n),p) = q(y,\varsigma,n,p) \in q(U\times\{n\}\times\{p\})\text.
	 \end{equation*}
	 Thus $V\subseteq q^{-1}(q(U\times\{n\}\times\{p\}))$ is an open neighborhood of $(k,\sigma,m,p)$. Hence, as claimed, $q^{-1}(q(U\times\{n\}\times\{p\}))$, being an open neighborhood of each of its point, is itself open. So $q$ is indeed an open map.
	 
	 As a result of $q$ being an open surjection, we conclude that a sequence $(g_n)_{n\in\N}$ converges to $g$ in $\mathscr{G}$ if, and only if, there exists $w \in \Xi\times\Z^d\times\Z_2^{2d}$ such that $q(w) =  g$, and there exists a subsequence $(g_{f(n)})_{n\in\N}$ of $(g_n)_{n\in\N}$ and a sequence $(w_n)_{n\in\N}$ in $\Xi\times\Z^d\times\Z_2^{2d}$ such that $q(w_n) = g_{f(n)}$ and $(w_n)_{n\in\N}$ converges to $w$. Since $\mathscr{G}$ is first countable as easily checked, this fully characterizes the topology on $\mathscr{G}$. In particular, endowing each fiber with the counting measure, the groupoid $\mathscr{G}$ is a topological (locally compact) groupoid.

	The set of composable pairs for our groupoid $\mathscr{G}$ is:
	\begin{equation*}
		\mathscr{G}^{(2)} = \left\{ ((k,\sigma,m,p),(y,\eta,m',p')) \in \mathscr{G}^2 : (k,\sigma) = (y,\eta) \right\}\text.
	\end{equation*}
	By \cite[Theorem 2.6]{Latremoliere05}, if we define the continuous $2$-cocycle $\beta$ of $\mathscr{G}$ by setting, for all $((k,\sigma,m,p), (k,\sigma,m',p')) \in \mathscr{G}^{(2)}$:
	\begin{equation*}
		\beta( (k,\sigma,m,p), (k,\sigma,m',p')) = \sigma(m,m') \mathfrak c(p,p')
	\end{equation*}
	where $\mathfrak c$ is defined via Equation \eqref{car-field-cocycle-eq}, then $\C^\ast(\mathscr{G},\beta)$ is the C*-algebra of continuous sections for the continuous field of C*-algebras $\left( \C^\ast(\Z^d_k,\sigma)\otimes\operatorname{CAR}_d\right)_{(k,\sigma) \in \Xi}$.
	
	We now fix $(k,\sigma) \in \Xi$, and $f : S \to \C$, where $S\subseteq\Z^d$ is a finite, nonempty subset of $\Z^d$ such that $q_k$ restricts to an injection on $S$. Let $\Omega\subseteq\Xi$ be an open neighborhood of $(k,\sigma)$ such that, if $(y,\varsigma) \in \Omega$, then $q_y$ also restricts to an injection on $S$. Practically, it suffices that the entries of $k$ and $y$ simply are larger than all the entries of all the elements in $S$ in absolute value.
	
	For each $(y,\varsigma) \in \Omega$ and $m \in q_y(S)$, we let
	\begin{align*}
		p_j(f)(y,m)&=\begin{cases}
			i m_j f\circ ({q_y}_{|S})^{-1}(m) &\text{if }y_j = \infty,\\
			\frac{y_j}{2 \pi}\left(\exp\left(\frac{2\pi i m_j}{y_j}\right)-1\right)f\circ ({q_y}_{|S})^{-1}(m)&\text{if }y_j < \infty,
		\end{cases}\\
		q_j(f)(y,m)&=\begin{cases}
			i m_j f\circ ({q_y}_{|S})^{-1}(m)&\text{if }y_j = \infty,\\
			\frac{y_j}{2 \pi}\left(1-\exp\left(-\frac{2\pi i m_j}{y_j}\right)\right)f\circ ({q_y}_{|S})^{-1}(m)&\text{if }y_j < \infty.
		\end{cases}
	\end{align*}
	Choose a continuous function $\chi$ supported in $\Omega$ and equal to $1$ on a neighborhood of $(k,\sigma)$. Viewing $\car_j$ and $\car_j^*$ as finitely supported coefficient functions on $\Z_2^{2d}$, define the following compactly supported function over our groupoid $\mathscr{G}$:
	\begin{equation*}
		h(y,\varsigma,m,p)
		=\chi(y,\varsigma)\sum_{j=1}^d\left(p_j(f)(y,m)\car_j(p)-q_j(f)(y,m)\car_j^*(p)\right),
	\end{equation*}
	with value $0$ when $m\notin q_y(S)$.
	
	Since $S$ is finite, the components $p_j(f)(y,m)$ and $q_j(f)(y,m)$ depend continuously on $y \in\Nbar^d$. By appealing to our sequential characterization of the final topology on $\mathscr{G}$, any converging sequence in $\mathscr{G}$ lifts locally to a sequence in the domain whose discrete components are eventually constant. 
	
	By construction, $h$ is a continuous function over $\mathscr{G}$, and an element of $\C^\ast(\mathscr{G},\beta)$. Moreover, on the neighborhood where $\chi=1$, the image of $h$ in the quotient over $(y,\varsigma)$ is, by construction:
	\begin{equation*}
		\sum_{j=1}^d\left(\fpartial_{y,\varsigma}^j(f_{y,\varsigma})\otimes\car_j-\bpartial_{y,\varsigma}^j(f_{y,\varsigma})\otimes\car_j^*\right) = \gradient_{y,\varsigma} f_{y,\varsigma} \text.
	\end{equation*}
	Hence, the continuous field property applies to show that (using Theorem \ref{L-explicit-thm}):
	\begin{equation*}
		\Lip_{y,\varsigma}(f_{y,\varsigma}) = \norm{\gradient_{y,\varsigma} f_{y,\varsigma}}{\A_{y,\varsigma}\otimes\operatorname{CAR}_d} \xrightarrow{(y,\varsigma)\to (k,\sigma)} \norm{\gradient_{k,\sigma} f_{k,\sigma}}{\A_{k,\sigma}\otimes\operatorname{CAR}_d} = \Lip_{k,\sigma}(f_{k,\sigma}) \text.
	\end{equation*}
	We thus have shown the desired pointwise convergence of the L-seminorms, as stated.
\end{proof}

\subsection{Continuity for the propinquity}

We are now ready to apply \cite{Latremoliere21a}, using Lemmas \ref{dual-iso-lemma}, \ref{compression-lemma} and \ref{continuous-field-L-lemma}.

\begin{theorem}\label{prop-cv-thm}
For all normalized $2$-cocycle $\sigma$ of $\Z^d$, the following limit holds:
	\begin{equation*}
		\lim_{\substack{(k,\varsigma)\to (\infty,\sigma) \\ (k,\varsigma) \in \Xi}} \dpropinquity{4d,0}((\A_{k,\varsigma},\Lip_{k,\varsigma}),(\A_{\infty,\sigma},\Lip_{\infty,\sigma})) = 0 \text.
	\end{equation*}
\end{theorem}

This theorem directly follows the argument laid in \cite[Theorem 5]{Latremoliere13c}, also discussed in \cite{Latremoliere21a}. For our purpose, we only record the ingredients which we need. We denote, for all $h \in \Pbar^d$ and $S\subseteq\Z^d_h$, the subspace $\{ \xi \in \ell^2(\Z^d_h) : \forall n \notin S \quad \xi(n) = 0 \}$ by $\ell^2(\Z^d_h | S)$.

For all $\sigma\in\mathscr{C}_2(\Z^d)$, and for all $\varepsilon > 0$, there exists, by \cite{Latremoliere13c},
\begin{itemize}
 	\item a finite nonempty subset $S_\varepsilon \subseteq \Z^d$, 
 	\item a function $f_\varepsilon \in C(\T^d)$ whose Fourier transform is supported on $S_\varepsilon$ and such that $\int_{\T^d} f_\varepsilon(z)\operatorname{slen}(z)\,d\mu_k(z) < \frac{\varepsilon}{4}$, 
 	\item an open neighborhood $\Omega_\varepsilon$ of $(\infty,\sigma)$ in $\Xi$,
 	\item for all $(k,\varsigma) \in \Omega_\varepsilon$, a nondegenerate faithful *-representation $\pi_{k,\varsigma}$ of $\A_{k,\varsigma}$ on $\ell^2(\Z^d)$, and a linear isometry $\vartheta_{k} : \ell^2(\Z^d_k) \to \ell^2(\Z^d)$ 
 	\item a trace-class operator $R_{\varepsilon}$ on $\ell^2(\Z^d)$ of norm $1$,
\end{itemize}
such that, for all $(k,\varsigma) \in \Omega_\varepsilon$,
\begin{itemize}
	\item for all $a \in \A_{k,\varsigma}$
 	\begin{equation*}
 		\pi_{k,\varsigma}(a) \vartheta_{k} = \vartheta_{k} a
 	\end{equation*} 
 	\item $\vartheta_{k} \ell^2(\Z^d_k | S_\varepsilon) = \ell^2(\Z^d | S_\varepsilon)$,
 	\item if $a \in \alpha_{k,\varsigma}^{f_\varepsilon} (\A_{k,\varsigma})$, then $\opnorm{[R_{\varepsilon}, \pi_{k,\varsigma}(a)]}{}{\ell^2(\Z^d)} \leq \varepsilon \Lip_{k,\varsigma}(a)$,
 	\item if $a \in \alpha_{k,\varsigma}^{f_\varepsilon}(\A_{k,\varsigma})$, then $|\Lip_{\infty,\sigma}(a) - \Lip_{k,\varsigma}(a)| \leq \varepsilon \Lip_{\infty,\sigma}(a)$,
	\item by \cite{Latremoliere13b}, setting, for each $(k,\varsigma) \in \Omega_\varepsilon$:
\begin{equation*}
	\tau_{k,\varsigma,\varepsilon} \coloneqq \left( \A_{k,\varsigma}\oplus \A_{\infty,\sigma}, \Lip[T]_{k,\varsigma,\varepsilon}, j_{k,\varsigma}, j_{\infty,\sigma} \right)
\end{equation*}
where, for all $a\in\dom{\Lip_{k,\varsigma}}$, $b \in \dom{\Lip_{\infty,\sigma}}$, we define
\begin{equation*}
	\Lip[T]_{k,\varsigma,\varepsilon}(a,b) \coloneqq \max\left\{ \Lip_{k,\varsigma}(a), \Lip_{\infty,\sigma}(b), \frac{1}{\varepsilon}\opnorm{\pi_{k,\varsigma}(a) R_{\varepsilon} - R_{\varepsilon} \pi_{\infty,\sigma}(b)}{}{\ell^2(\Z^d)} \right\} \text,
\end{equation*}
while $j_{k,\varsigma} : (a,b)\in\A_{k,\varsigma}\oplus \A_{\infty,\sigma}\mapsto a$ and $j_{\infty,\sigma}:(a,b)\in\A_{k,\varsigma}\oplus \A_{\infty,\sigma}\mapsto b$, then $\tau_{k,\varsigma,\varepsilon}$ is a tunnel of extent at most $\varepsilon$.
\end{itemize}
\begin{remark}
The 4-tuple $(\B(\ell^2(\Z^d)), \pi_{\infty,\sigma}, \pi_{k,\varsigma}, R_{\varepsilon})$ is a bridge of length at most $\varepsilon$ in the sense of \cite{Latremoliere13} which was constructed in \cite{Latremoliere13c}; we apply to that construction \cite{Latremoliere13b,Latremoliere14} to obtain the tunnel above, thus bypassing the need to discuss bridges in this paper.
\end{remark}

In the last section of this work, we will start from this tunnel to build a tunnel between our twisted spectral triples, which justifies why we give these details. The complete proof this tunnel works is the entire purpose of \cite{Latremoliere13c}.

We now turn to convergence for the spectral propinquity adjusted to include the convergence of the twist.

\section{The Spectral Propinquity for twisted spectral triples}

The \emph{spectral propinquity} \cite{Latremoliere18g} is a distance function on the class of  metric spectral triples, up to unitary equivalence, obtained by applying the more general \emph{covariant metrical propinquity} \cite{Latremoliere18d, Latremoliere18g} to a specific action of the monoid $[0,\infty)$ on specific metrical C*-correspondences constructed from spectral triples. Our extension of the spectral propinquity to twisted spectral triples follow this approach and thus, is reasonably straightforward. We nonetheless have to check that our construction does lead to a distance function up to unitary equivalence, including for the twist. Moreover, a welcome surprise arises: at the very light cost of generalizing metrical C*-correspondences, we can work with twisted spectral triples with unbounded twists, even if the resulting {\qcms} is \emph{relaxed}, i.e. does not satisfy any Leibniz inequality. 

We thus extend the spectral propinquity to the following class of twisted metric spectral triples.

\subsection{Metric Twisted Spectral Triples}

We begin from Definition \ref{twisted-def} of a twisted spectral triple, and add appropriate regularity conditions and a relaxed {\qcms} structure.

\begin{definition}\label{metric-twisted-def}
	A \emph{$\kappa$-metric twisted spectral triple} $(\A,\Hilbert,\Dirac,\rho)$, for $0 < \kappa\leq 1$, is a twisted spectral triple such that
	\begin{enumerate}
	\item if we define
	\begin{equation*}
		\forall a \in \A_{\Dirac} \quad \Lip(a) \coloneqq \max\left\{
		\opnorm{[\Dirac,a]_\rho}{}{\Hilbert},
		\kappa\opnorm{\rho(a)-a}{}{\Hilbert}
		\right\}
	\end{equation*}
	then the {\MongeKant} $\Kantorovich{\Lip}$ metrizes the weak* topology on $\StateSpace(\A)$,
		\item if $(a_j)_{j \in J}$ is a net in $\A_{\Dirac}$, $a \in \A$ and $T$ is a bounded operator on $\Hilbert$, such that:
		\begin{equation*}
			a_j \xrightarrow{j \in J} a \text{ in norm, } \qquad\rho(a_j) \xrightarrow{j \in J}T \text{ in WOT, }\qquad \text{ and }\sup_{j \in J} \opnorm{[\Dirac,a_j]_\rho}{}{} < \infty \text, 
		\end{equation*}
		then
		\begin{equation*}
			a \in \A_{\Dirac}\text{ and }\rho(a) = T \text;
		\end{equation*}
	\item $\rho(1) = 1$.
	\end{enumerate}
\end{definition}

The second term records, in the metric data themselves, how far the twist is from the identity. Notice that it gives
\begin{equation}\label{augmented-twist-bound-eq}
	\opnorm{\rho(a)}{}{\Hilbert}
	\leq\norm{a}{\A}+\kappa^{-1}\Lip(a)
	\leq(1+\kappa^{-1})\max\{\norm{a}{\A},\Lip(a)\}.
\end{equation}

\begin{remark}
	Our second condition on the twist $\rho$ can be restated as saying the graph of the restriction of $\rho$ to the set $\{ a \in \A_{\Dirac} : \opnorm{[\Dirac,a]_\rho}{}{\Hilbert} \leq 1 \}$ is closed from the norm topology to the WOT topology.
\end{remark}

\begin{remark}
	Expression \eqref{augmented-twist-bound-eq} implies that $\rho$ is continuous from $(\dom{\Lip},\max\{\Lip,\norm{\cdot}{\A}\})$ to the norm topology on $\B(\Hilbert)$.
\end{remark}

\begin{remark}
	We do not require any Leibniz property in Definition \ref{metric-twisted-def}. This is a surprise to us: as we shall see, in this particular setting, we can circumvent this, unlike the general situation for {\qcms s}. In particular, the space $\A_{\Dirac}$ of a metric twisted spectral triple need not be a subalgebra.
\end{remark}

\begin{proposition}
	If $(\A,\Hilbert,\Dirac,\rho)$ is a twisted spectral triple satisfying Conditions (1) and (3) of Definition \ref{metric-twisted-def}, and if $\rho$ is continuous from $\A$ to the algebra of bounded linear operators on $\Hilbert$ endowed with the weak operator topology, then $(\A,\Hilbert,\Dirac,\rho)$ is a metric twisted spectral triple.
\end{proposition}

\begin{proof}
	Since the weak operator topology is Hausdorff, the graph of the continuous $\rho$ is closed.
\end{proof}

Of course, if $\rho$ is continuous for the norm topologies, then it is continuous from the norm topology to the weak operator topology, and thus satisfies the closure Condition (2) of Definition \ref{metric-twisted-def}. Thus, in particular:
\begin{proposition}
	For all $(k,\sigma) \in \Xi$, the twisted spectral triple $(\A_{k,\sigma}, \Hilbert_k, \Dirac_k, \rho_{k,\sigma})$ is $\kappa_d$-metric.
\end{proposition}
\begin{proof}
	This follows from Theorem \ref{qcms-thm}, Theorem \ref{L-explicit-thm}, Corollary \ref{car-twist-estimate-cor}, and the norm continuity of $\rho_{k,\sigma}$.
\end{proof}

We now make a small digression to prove that other twists of a form commonly found in the literature (for instance, for conformal transformations of spectral triples) also satisfy our closure condition. We believe this case and the simpler, continuous twist case cover many common examples of twisted spectral triples.

\begin{proposition}
	Let $\Hilbert$ be a Hilbert space, and let $\A$ be a C*-algebra acting on $\Hilbert$. If $S$ is a closed invertible operator on a dense subspace $\dom{S}$ of $\Hilbert$, and if $\rho(a) = S a S^{-1}$ for all $a \in \A$ such that $a\dom S \subseteq \dom S$ and $S a S^{-1}$ is bounded, then $\rho$ is closed from the norm topology on $\A$ to the weak operator topology.
\end{proposition}

\begin{proof}
	Let $(a_n)_{n\in\N}$ be a sequence in $\dom{\rho}$ converging in $\A$ to $a$, and such that $\rho(a_n) = S a_n S^{-1}$ converges to some operator $T$ in the weak operator topology.

	Since $S$ is a closed operator, its graph is closed in the norm topology of $\Hilbert\oplus\Hilbert$ by definition. By Mazur's theorem, since the graph of $S$ is a subspace, it is convex, and thus it is closed in the weak topology. So if $(\zeta_n)_{n\in\N}$ is a sequence in $\Hilbert$ weakly converging to some $\zeta\in\Hilbert$, and if $(S\zeta_n)_{n\in\N}$ weakly converges to some $\omega$, then $S\zeta = \omega$.

	Let $\xi \in \dom{S}$ and let $\eta \coloneqq S\xi$. For all $\omega\in\Hilbert$, the sequence $(\inner{\rho(a_n)\eta}{\omega}{\Hilbert})_{n\in\N}$ converges to $\inner{T\eta}{\omega}{\Hilbert}$. So $(S a_n \xi)_{n\in\N}$ weakly converges to $TS\xi$.
	
	Now, $(a_n)_{n\in\N}$ converges in norm, so $(a_n \xi)_{n\in\N}$ converges in norm to $a\xi$, hence it converges weakly. Therefore, as $S$ is closed, its graph is closed for the weak topology, and thus we conclude the weak limit of $(S a_n\xi)_{n\in\N}$ is, in fact, $S a \xi$. So $S a \xi = T S \xi$. So $S a S^{-1} = T$ as desired.
\end{proof}

We now return to our primary interest. The construction used to associate a monoid action on a metrical C*-correspondence to any spectral triple generalizes immediately to our present situation, with the change that we must allow for the lack of the Leibniz properties for the Lip-norm induced by a twisted spectral triple. However, we \emph{do} require the other Leibniz properties (modular and inner) that we used for metrical C*-correspondences.

\begin{definition}
	A \emph{relaxed $(C,D,G,H)$-metrical C*-correspondence} $(\module{M}, \CDN, \A, \Lip_\A, \B, \Lip_\B)$, for $C\geq 1$, $D\geq 0$, $G\geq 1$ and $H\geq 2$, is an $\A$-$\B$ C*-correspondence $(\module{M}, \A, \B)$, together with:
	\begin{enumerate}
		\item a norm $\CDN$ defined on a dense $\C$-subspace $\dom{\CDN}$ of $\module{M}$ 
		\item a $(C,D)$-{\qcms} $(\B,\Lip_\B)$,
		\item a relaxed \qcms $(\A,\Lip_\A)$,
	\end{enumerate}
	such that
	\begin{enumerate}
		\item the closed unit ball is compact for the Hilbert norm on $\module{M}$
		\item $\CDN$ dominates the restriction to $\dom{\CDN}$ of the Hilbert module norm of $\module{M}$,
		\item we require a form of \emph{modular Leibniz property}: for all $a \in \dom{\Lip_\A}$ and for all $\xi \in \dom{\CDN}$, we impose $a\xi \in \dom{\CDN}$, and
		\begin{equation*}
			\CDN(a\xi) \leq G\max\left\{ \norm{a}{\A} , \Lip_\A(a)\right\}  \CDN(\xi) \text,
		\end{equation*}
		\item we also require a form of \emph{inner Leibniz property}: for all $\xi, \eta \in \dom{\CDN}$, the inner product $\inner{\xi}{\eta}{\module{M}}$ is in $\dom{\Lip_\B}$, and
		\begin{equation*}
			\Lip_\B(\inner{\xi}{\eta}{\module{M}}) \leq H \CDN(\xi) \CDN(\eta) \text.
		\end{equation*}
	\end{enumerate}

	If $(\A,\Lip_\A)$ is a $(F,F')$-{\qcms}, then $(\module{M},\CDN,\A,\Lip_\A,\B,\Lip_\B)$ is a \emph{metrical C*-correspondence} with parameter $(F,F',C,D,G,H)$. 
\end{definition}

The key observation which starts our conversation on convergence of metric twisted spectral triples is the following generalization of \cite[Theorem 2.7]{Latremoliere18g}.
\begin{theorem}
	If $(\A,\Hilbert,\Dirac,\rho)$ is a $\kappa$-metric twisted spectral triple for $0 < \kappa \leq 1$, then the sextuple
	\begin{equation*}
		\mcc{\A}{\Hilbert}{\Dirac,\rho} \coloneqq (\Hilbert,\CDN,\A,\Lip,\C,0)
	\end{equation*}
	where, for all $\xi \in \dom{\Dirac}$, we define
	\begin{equation*}
	\CDN(\xi) = \norm{\xi}{\Hilbert} + \norm{\Dirac\xi}{\Hilbert}
	\end{equation*}
	and for all $a \in \A_{\Dirac}$, we set
	\begin{equation*}
	\Lip(a) = \max\left\{
	\opnorm{[\Dirac,a]_\rho}{}{\Hilbert},
	\kappa\opnorm{\rho(a)-a}{}{\Hilbert}
	\right\}\text,
	\end{equation*}
	is a $(1,0,\kappa^{-1}+1,2)$-relaxed metrical C*-correspondence. Moreover, if we set for all $t \in \R$:
	\begin{equation*}
		V_t \coloneqq \exp( i t \Dirac )	
	\end{equation*}
	then $t \in [0,\infty) \mapsto V_t$ is a strongly continuous action on $\mcc{\A}{\Hilbert}{\Dirac,\rho}$ of the monoid $[0,\infty)$ by unitaries. 
\end{theorem}

\begin{remark}
	If moreover, $\Lip$ is $(C,D)$-Leibniz, then $\mcc{\A}{\Hilbert}{\Dirac,\rho}$ is a $(C,D,1,0,\kappa^{-1}+2,2)$ metrical C*-correspondence.
\end{remark}

\begin{proof}
First, note that $[\Dirac,1]_\rho = 0$ by Condition (3) in Definition \ref{metric-twisted-def}. Since $\Kantorovich{\Lip}$ metrizes the weak* topology of $\StateSpace(\A)$, it follows that $\Lip(a) = 0 \implies a \in \C\unit_\A$, so we conclude $\ker \Lip = \C \unit_\A$. 

By assumption, $\Kantorovich{\Lip}$ metrizes the weak* topology. We now prove that $\Lip$ is also lower semicontinuous on $\A$.

Let $(a_n)_{n\in\N}$ be a sequence in $\A$ converging to $a$. If $\liminf_{n\to\infty}\Lip(a_n)=\infty$, there is nothing to prove. Otherwise, after passing to a subsequence, we may assume that
\begin{equation*}
	\lim_{n\to\infty}\Lip(a_n)=\liminf_{n\to\infty}\Lip(a_n)<\infty.
\end{equation*}
The sequences $(\rho(a_n))_{n\in\N}$ and $([\Dirac,a_n]_\rho)_{n\in\N}$ are norm-bounded by Expression \eqref{augmented-twist-bound-eq} and the definition of $\Lip$. By weak-operator compactness of bounded balls in $\B(\Hilbert)$, they admit a common subnet converging in the weak operator topology, respectively, to operators $R,T\in\B(\Hilbert)$. By assumption on $\rho$, we conclude that $R=\rho(a)$ (the graph of the restriction of $\rho$ to balls for the L-seminorm $\Lip$ is closed from norm to WOT). In particular, for all $\xi,\eta\in\Hilbert$,
\begin{equation*}
	\inner{T\xi}{\eta}{\Hilbert} = \lim_{n\to\infty} \inner{[\Dirac,a_n]_\rho \xi}{\eta}{\Hilbert} \text.
\end{equation*}

On the other hand, for all $\xi,\eta\in\dom{\Dirac}$, we then have, by convergence of $(a_n)_{n\in\N}$ and WOT convergence of $(\rho(a_n))_{n\in\N}$:
\begin{align*}
	\inner{[\Dirac,a_n]_\rho \xi}{\eta}{\Hilbert}
	&= \inner{\Dirac a_n \xi - \rho(a_n) \Dirac \xi}{\eta}{\Hilbert} \\
	&= \inner{a_n \xi}{\Dirac \eta}{\Hilbert} - \inner{\rho(a_n) \Dirac\xi}{\eta}{\Hilbert} \\
	&\xrightarrow{n \to \infty} \inner{a \xi}{\Dirac\eta}{\Hilbert} - \inner{\rho(a)\Dirac\xi}{\eta}{\Hilbert} \text.
\end{align*}
Thus
\begin{equation*}
	\inner{a\xi}{\Dirac\eta}{\Hilbert}
	=
	\inner{T\xi+\rho(a)\Dirac\xi}{\eta}{\Hilbert}
\end{equation*}
for all $\eta\in\dom{\Dirac}$. Since $\Dirac$ is self-adjoint, this proves that $a\xi\in\dom{\Dirac}$ and $[\Dirac,a]_\rho=T$. The weak-operator lower semicontinuity of the operator norm gives
\begin{align*}
	\opnorm{[\Dirac,a]_\rho}{}{\Hilbert}
	&\leq\liminf_{n\to\infty}\opnorm{[\Dirac,a_n]_\rho}{}{\Hilbert},\\
	\opnorm{\rho(a)-a}{}{\Hilbert}
	&\leq\liminf_{n\to\infty}\opnorm{\rho(a_n)-a_n}{}{\Hilbert}.
\end{align*}
Consequently, $\Lip(a)\leq\liminf_{n\to\infty}\Lip(a_n)$.
Hence $\Lip$ is lower semicontinuous. So $(\A,\Lip)$ is indeed a relaxed {\qcms}.

It is immediate that $(\Hilbert,\A,\C)$ is a C*-correspondence, with $\A$ acting on the left and $\C$ seen as acting on the right (unusually for Hilbert spaces, but fitting our definition). Moreover, since $\Dirac$ has compact resolvent, just as in \cite{Latremoliere18g}, we note that $\{ \xi \in \dom{\Dirac} : \CDN(\xi) \leq 1\}$ is compact. Moreover $\CDN\geq\norm{\cdot}{\Hilbert}$ by construction.

We turn to the modular Leibniz properties. Let $a \in \dom{\Lip}$ and $\xi \in \dom{\Dirac}$. We compute:
	\begin{align*}
		\CDN(a\xi)
		&=\norm{a\xi}{\Hilbert}+\norm{\Dirac a \xi}{\Hilbert} \\
		&\leq\norm{a}{\A}\norm{\xi}{\Hilbert}
		+ \norm{(\Dirac a - \rho(a)\Dirac)\xi}{\Hilbert} + \norm{\rho(a)\Dirac\xi}{\Hilbert} \\
		&\leq\norm{a}{\A}\norm{\xi}{\Hilbert}+\Lip(a)\norm{\xi}{\Hilbert}+\opnorm{\rho(a)}{}{\Hilbert}\norm{\Dirac\xi}{\Hilbert} \\
		&\leq \norm{a}{\A}\norm{\xi}{\Hilbert}+\Lip(a)\norm{\xi}{\Hilbert}+(\kappa^{-1}\Lip(a)+\norm{a}{\A})\norm{\Dirac\xi}{\Hilbert} \\
		&\leq(\norm{a}{\A}+\kappa^{-1}\Lip(a))\CDN(\xi) \text{ since }\kappa \leq 1 \text.
	\end{align*}
	
	In particular,
	\begin{equation*}
		\CDN(a\xi)\leq (1+\kappa^{-1})\max\{\norm{a}{\A},\Lip(a)\} \CDN(\xi) \text.
	\end{equation*}
	
	The inner Leibniz property is trivial here as the L-seminorm on $\C$ is $0$. Hence, we have shown that $\mcc{\A}{\Hilbert}{\Dirac,\rho}$ is a $(1,0,\kappa^{-1}+1,2)$-relaxed metrical C*-correspondence.
	
	It is of course standard that $V:t\in [0,\infty) \to V_t$ is a strongly continuous action by unitaries of $\Hilbert$ of the additive monoid $[0,\infty)$.
\end{proof}

\begin{remark}
	Lemma \ref{lower-semi-lemma} is now just a simple corollary of the above result.
\end{remark}

We note in passing the following consequence of our construction.
\begin{corollary}
	If $(\A,\Hilbert,\Dirac,\rho)$ is a metric twisted spectral triple, then $(\A,\Lip)$ is a relaxed \qcms.
\end{corollary}

We pause to emphasize once more than, in our current setting, we \emph{do not} require that our twisted spectral triples give rise to \emph{Leibniz}-type L-seminorms via the twisted commutator. This indeed allows us to deal with unbounded twists in general. The \emph{surprise} here will be that the spectral propinquity, extended to metric twisted spectral triples, is still a metric up to unitary equivalence, including distance zero implying that the underlying C*-algebras are isomorphic. This normally follows from the now unassumed Leibniz property. This is specific to the spectral propinquity situation, however: in general, when working with the propinquity or its covariant and modular counterparts, we \emph{do} need the Leibniz property. For instance, removing the Leibniz property from the dual propinquity construction leads to Rieffel's distance instead \cite{Latremoliere20b}. So it is very surprising that in this very special case of the spectral propinquity for metric twisted spectral triple, this condition is not needed --- it is particularly welcomed since indeed, we do not expect the Leibniz property to hold in any form when working with unbounded twists. To be clear, it means that when working with metric twisted spectral triples, the dual propinquity alone does not give us a distance up to full quantum isometry, since the Leibniz property may not hold, \emph{but} the spectral propinquity will apply just well by capturing what is missing from the lack of the Leibniz property from the modular Leibniz property instead.

\emph{However}, it should be made equally clear that in the case of the twisted spectral triples on fuzzy tori, the twists are bounded and the L-seminorms are, in fact, $(4d,0)$-Leibniz. We also note that when working with untwisted metric spectral triples, then the Leibniz property is immediately satisfied, so these observations were not made in that case.

Since we have a strongly continuous action of a monoid on a relaxed metrical C*-correspondence, we naturally would like to use the covariant modular propinquity to define a distance on metric twisted spectral triples, just as we did in the untwisted case. However, we do want to also account for the twist in our distance, and we also must keep in mind that we work with \emph{relaxed} metrical C*-correspondence.

\subsection{Incorporating a distance between twists in the spectral propinquity}

Let us first agree on what we mean by unitarily equivalent twisted spectral triples.

\begin{definition}\label{unitary-eq-def}
	Two twisted spectral triples $(\A,\Hilbert,\Dirac,\rho)$ and $(\B,\Hilbert[J],\Dirac[\nabla],\varrho)$ are \emph{unitarily equivalent} when there exists a unitary $U : \Hilbert\rightarrow\Hilbert[J]$ such that
	\begin{enumerate}
		\item $U\dom{\Dirac} = \dom{\Dirac[\nabla]}$,
		\item $U\Dirac U^\ast = \Dirac[\nabla]$,
		\item $\pi : a \in \A \mapsto U a U^\ast$ is a *-isomorphism from $\A$ onto $\B$,
		\item $U\rho(\cdot) U^\ast = \varrho\circ\pi(\cdot)$ on $\A_{\Dirac}$.
	\end{enumerate}
	We say that such a unitary $U$ \emph{implements} the unitary equivalence between the two twisted spectral triples.
\end{definition}
\begin{remark}
	If $\varrho$ and $\rho$ are the identity, then (3) and (4) become the same statement, meaning that our notion is compatible with the usual unitary equivalence of spectral triples. It will be the pattern moving forward: our construction immediately reduces to the usual propinquity between spectral triples when they are ``untwisted'', i.e. the twist is the identity.
\end{remark}

We dive a little bit in the details on how the spectral propinquity was constructed between (untwisted) spectral triples, at the cost of unpacking the definition of the covariant metrical propinquity. This will be necessary to prove that a natural extension of the spectral propinquity can indeed account for the twist, and to emphasize how working with relaxed metrical C*-correspondences affect our construction.

\begin{definition}
  A \emph{quantum isometry} $(\Psi,\pi,\theta)$ from a relaxed metrical C*-correspondence $\mathds{P} = (\module{P},\TDN,\D,\Lip_\D,\alg{E},\Lip_{\alg{E}})$ to another relaxed metrical C*-correspondence $\mathds{M} = (\module{M},\CDN,\A,\Lip_\A,\B,\Lip_\B)$ is a triple of a surjective $\C$-linear map $\Psi :  \module{P} \to \module{M}$ and two unital surjective *-morphisms $\pi : \D \to \A$ and $\theta : \alg{E} \to \B$, such that:
\begin{enumerate}
	\item $\Psi(a \xi b) = \pi(a) \Psi(\xi) \theta(b)$ for all $a \in \D$, $b\in\alg{E}$ and $\xi \in \module{P}$,
	\item $\inner{\Psi(\xi)}{\Psi(\eta)}{\module{M}} = \theta(\inner{\xi}{\eta}{\module{P}})$ for all $\xi,\eta \in \module{P}$,
	\item $\pi$ and $\theta$ are quantum isometries,
	\item $\Psi(\dom{\TDN}) = \dom{\CDN}$
	\item $\CDN(\xi) = \inf\left\{ \TDN(\omega) : \omega\in\dom{\TDN}, \Psi(\omega) = \xi \right\}$ for all $\xi \in \dom{\CDN}$.
\end{enumerate}
\end{definition}
Our definition of a quantum isometry is exactly the same definition as for metrical C*-correspondences, since none of these requirements depend on $(\D,\Lip_\D)$ and $(\A,\Lip_\A)$ being Leibniz.

\begin{definition}\label{tunnel-def}
We assume henceforth that we are given two relaxed $(C,D,G,H)$-metrical C*-correspondences $\mathds{M}_1$ and $\mathds{M}_2$ (for $C\geq 1$, $D\geq 0$, $G\geq 1$ and $H\geq 2$), which we write for each $j \in \{1,2\}$ as:
\begin{equation*}
	\mathds{M}_j \coloneqq \left( \mathscr{M}_j, \CDN_j, \A_j, \Lip_j, \B_j, \Lip[Q]_j \right)\text.
\end{equation*}
A \emph{$(C,D,G,H)$-tunnel} $\tau\coloneqq (\mathds{P},\Pi_1,\Pi_2)$ between these two relaxed metrical C*-correspondences is the data of another $(C,D,G,H)$-relaxed metrical C*-correspondace
\begin{equation*}
	\mathds{P} \coloneqq \left( \mathscr{P}, \TDN, \D, \Lip[T], \alg{E}, \Lip[S] \right)
\end{equation*}
together with two quantum isometries of relaxed metrical C*-correspondences $\Pi_1 = (\Psi_1,\pi_1,\theta_1)$ from $\mathds{P}$ to $\mathds{M}_1$ and $\Pi_2 = (\Psi_2, \pi_2, \theta_2)$ from $\mathds{P}$ to $\mathds{M}_2$.

The \emph{domain} $\dom{\tau}$ of $\tau$ is $\mathds{M}_1$, and the codomain $\codom{\tau}$ of $\tau$  is $\mathds{M}_2$.
\end{definition}

\begin{notation}
	If $\tau = (\mathds{P},\Pi_1,\Pi_2)$ is a tunnel, then $\tau^{-1}$ is the tunnel $(\mathds{P},\Pi_2,\Pi_1)$ from $\codom{\tau}$ to $\dom{\tau}$.
\end{notation}

\begin{definition}
		If $(\A,\Hilbert,\Dirac,\rho)$ and $(\B,\Hilbert[S],\Dirac[\nabla],\varrho)$ are two $\kappa$-metric twisted spectral triple, then a tunnel from $(\A,\Hilbert,\Dirac,\rho)$ to $(\B,\Hilbert[S],\Dirac[\nabla],\varrho)$ is a tunnel from $\mcc{\A}{\Hilbert}{\Dirac,\rho}$ to $\mcc{\B}{\Hilbert[S]}{\Dirac[\nabla],\varrho}$.
\end{definition}

\begin{definition}
	If $(\A,\Hilbert,\Dirac,\rho)$ and $(\B,\Hilbert[S],\Dirac[\nabla],\varrho)$ are two unitarily equivalent $\kappa$-metric twisted spectral triples, with $U : \Hilbert_1 \to \Hilbert_2$ a unitary implementing this equivalence, then
	\begin{equation*}
		\left( \mcc{\A}{\Hilbert}{\Dirac,\rho} , (\mathrm{id}_{\Hilbert_1}, \mathrm{id}_{\A}, \mathrm{id}_\C) , (U, U\cdot U^\ast, \mathrm{id}_\C) \right)
	\end{equation*}
	is called the \emph{trivial tunnel} between $(\A,\Hilbert,\Dirac,\rho)$ and $(\B,\Hilbert[S],\Dirac[\nabla],\varrho)$.
\end{definition}

A tunnel allows us to define certain quantities to quantify how close the relaxed metrical C*-correspondences $\mathds{M}_1$ and $\mathds{M}_2$, and actions of monoids on them, are. 
\begin{definition}
	With the notations of Definition \ref{tunnel-def}, the \emph{extent} $\tunnelextent{\tau}$ of the $(C,D,G,H)$-tunnel $\tau \coloneqq (\mathds{P},\Pi_1,\Pi_2)$ is the maximum of the extents of the tunnels $(\D,\Lip[T],\pi_1,\pi_2)$ and $(\alg{E},\Lip[S],\rho_1,\rho_2)$.
\end{definition}

\begin{proposition}
	The extent of a trivial tunnel is $0$.
\end{proposition}

\begin{proof}
	This is an immediate computation.
\end{proof}

\begin{remark}
	For any metrical tunnel, $\tunnelextent{\tau^{-1}} = \tunnelextent{\tau}$ trivially.
\end{remark}

	 The extent suffices to define the metrical propinquity, albeit when working with metrical C*-correspondences; we work here with relaxed metrical C*-correspondences. Our focus remains the construction of a metric between metric twisted spectral triples, rather than general relaxed metrical C*-correspondences, in the present paper.
	 
	 When working with spectral triples, twisted or not, we wish to also account for the Dirac operators more explicitly, so as to obtain the desired coincidence property for the spectral propinquity. Henceforth, we assume that $\mathds{M}_1 = \mcc{\A_1}{\Hilbert_1}{\Dirac_1,\rho_1}$ and $\mathds{M}_2 = \mcc{\A_2}{\Hilbert_2}{\Dirac_2,\rho_2}$ for two $\kappa$-metric twisted spectral triples $(\A_1,\Hilbert_1,\Dirac_1,\rho_1)$ and $(\A_2,\Hilbert_2,\Dirac_2,\rho_2)$ for some $0<\kappa\leq1$. 

As in \cite{Latremoliere18g,Latremoliere22}, given a $(C,D,G,H)$-tunnel $(\mathds{P}, (\Psi_1,\pi_1,\theta_1), (\Psi_2,\pi_2,\theta_2))$ from $\module{M}_1$ to $\module{M}_2$, where
\begin{equation*}
	\mathds{P} \coloneqq (\module{P}, \TDN, \D, \TLip, \alg{E}, \Lip[Q])\text,
\end{equation*}
we propose a tool to quantify how far from each other the actions $t \in [0,\infty) \mapsto \exp(2i\pi t \Dirac_1)$ and $t \in [0,\infty) \mapsto \exp(2i\pi t \Dirac_2)$ are. Given a nonempty set $J$ and two families of linear maps $(A_j)_{j\in J}$ and $(B_j)_{j \in J}$ on $\Hilbert_1$ and $\Hilbert_2$, respectively, we set
\begin{multline*}
	\tunnelsep{(A_j)_{j \in J} }{(B_j)_{j\in J}} \coloneqq \Haus{\Kantorovich{\TDN,J}}\Bigg( \left\{ \left( \inner{A_j\Psi_1(\omega)}{\Psi_1(\cdot)}{\Hilbert_1}\right)_{j\in J} : \TDN(\omega)\leq 1 \right\}, \\
 	 \left\{ \left( \inner{B_j\Psi_2(\omega)}{\Psi_2(\cdot)}{\Hilbert_2}\right)_{j \in J} : \TDN(\omega)\leq 1\right\} \Bigg) \text,
\end{multline*}
where, for any two families $(\varphi_j)_{j\in J}$ and $(\lambda_j)_{j\in J}$ of linear functionals over $\module{P}$, we set: 
\begin{equation*}
	\Kantorovich{\TDN,J}((\varphi_j)_{j \in J},(\lambda_j)_{j \in J}) \coloneqq \sup\left\{ |\varphi_j(\eta) - \lambda_j(\eta)| : \TDN(\eta)\leq 1, j \in J \right\} \text.
\end{equation*}
We refer to \cite{Latremoliere22} for discussion of these notions. We will also use the following notation for the closed unit ball in the domain on $\Lip[T]$:
\begin{equation*}
	B(\tau) \coloneqq \{ d \in \dom{\Lip[T]} : \max\{\norm{d}{\D}, \Lip[T](d) \} \leq 1 \} \text.
\end{equation*}

With this notation, we define the \emph{$\varepsilon$-magnitude} between our metric twisted spectral triples, as seen from the perspective of our tunnel $(\mathds{P},\Pi_1,\Pi_2)$, as the following number. We note that it is an immediate generalization of our work in \cite{Latremoliere18g}, designed to account for a possible twist.

\begin{multline}\label{magnitude-eq}
	\tunnelmagnitude{\tau ; \varepsilon} \coloneqq \\ 
	\max\Bigg\{ \overbracket[1pt]{\tunnelextent{\tau}, \tunnelsep{ (e^{i t \Dirac_1})_{t\in [0,\varepsilon^{-1}]}}{(e^{i t \Dirac_2})_{t\in [0,\varepsilon^{-1}]}}}^{\text{ from \cite{Latremoliere18g}, handles everything but the twist}}, \\ \underbracket[1pt]{\frac{1}{GH} \tunnelsep{(\rho_1\circ\pi_1(d)^\ast)_{d \in B(\tau)}}{ (\rho_2\circ\pi_2(d)^\ast)_{d\in B(\tau)}}}_{\text{ new part to handle the twist }}  \Bigg\} \text.
\end{multline}

Note that we use the separation between the families of the \emph{adjoints} of $\rho_1\circ\pi_1(d)$ and $\rho_2\circ\pi_2(d)$ as $d$ ranges over the ball $B(\tau)$. This is a quirk of our initial choice of definition for the separation in \cite{Latremoliere22}, and we use here the same notation.

In any case, as is common in our work, we allow for some flexibility in choosing a class of tunnels to define our propinquity. We will use the following notion.
	 \begin{definition}
	 	Let $0<\kappa\leq1$, and let $\mathcal{C}$ be a nonempty class of $\kappa$-metric twisted spectral triples. A \emph{$\mathcal{C}$-appropriate class of tunnels} $\mathcal{T}$ is a class of $(C,D,G,H)$-tunnels between relaxed $(C,D,G,H)$-metrical C*-correspondences with $C\geq 1$, $D\geq 0$, $G\geq \kappa^{-1}+2$ and $H\geq 2$ such that:
	 	\begin{enumerate}
	 	 	\item there exists a tunnel in $\mathcal{T}$ between any two $\kappa$-metric twisted spectral triples in $\mathcal{C}$,
	 		\item if $(\A_1,\Hilbert_1,\Dirac_1,\rho_1)$ and $(\A_2,\Hilbert_2,\Dirac_2,\rho_2)$ are in $\mathcal{C}$ and are unitarily equivalent, then $\mathcal{T}$ contains a trivial tunnel between them,
	 		\item if $\tau \in \mathcal{T}$ then $\tau^{-1} \in \mathcal{T}$,
	 	 	\item if $\tau, \tau' \in \mathcal{T}$, with $\codom{\tau} = \dom{\tau'}$, if $\tunnelmagnitude{\tau;\varepsilon_1} \leq \varepsilon_1$ for some $\varepsilon_1 > 0$, if $\tunnelmagnitude{\tau';\varepsilon_2} \leq \varepsilon_2$ for some $\varepsilon_2 > 0$, and if $\varepsilon > 0$, then there exists a tunnel $\tau\circ_\varepsilon\tau' \in \mathcal{T}$ from $\dom{\tau}$ to $\codom{\tau'}$ such that
	 		\begin{equation*}
	 			\tunnelmagnitude{\tau\circ_\varepsilon\tau';\varepsilon_1 + \varepsilon_2 + \varepsilon} \leq \varepsilon_1 + \varepsilon_2 + \varepsilon \text.
	 		\end{equation*}
	 	\end{enumerate} 
	 \end{definition}

\begin{remark}
	It is important to note that an appropriate class of tunnels consists entirely of $(C,D,G,H)$ tunnels for some fixed $C\geq 1$, $D\geq 0$, $G\geq\kappa^{-1}+2$ and $H\geq 2$ common to \emph{all} the tunnels in $\mathcal{T}$. This uniformity is essential to obtain the desired coincidence property for the spectral propinquity in general.
\end{remark}

We then proceed as usual to define the spectral propinquity:
\begin{definition}\label{spectral-prop-def}
	Let $\mathcal{C}$ be a nonempty class of $\kappa$-metric twisted spectral triples. Let $\mathcal{T}$ be a $\mathcal{C}$-appropriate class of tunnels. The \emph{$\mathcal{T}$-spectral propinquity} between any two $\kappa$-metric twisted spectral triples $(\A_1,\Hilbert_1,\Dirac_1,\rho_1)$ and $(\A_2,\Hilbert_2,\Dirac_2,\rho_2)$ in $\mathcal{C}$ is the real number:
	\begin{multline*}
		\spectralpropinquity{\mathcal{T}}((\A_1,\Hilbert_1,\Dirac_1,\rho_1), (\A_2,\Hilbert_2,\Dirac_2,\rho_2)) = 
		\inf\Bigg\{ \varepsilon>0 : \tunnelmagnitude{\tau ; \varepsilon}\leq\varepsilon, \\
		 \tau \text{ a tunnel in $\mathcal{T}$ from }(\A_1,\Hilbert_1,\Dirac_1,\rho_1) \text{ to }(\A_2,\Hilbert_2,\Dirac_2,\rho_2) \Bigg\} \text,
	\end{multline*}
	where $\tunnelmagnitude{\tau ; \varepsilon}$ is defined by Expression \eqref{magnitude-eq}.
\end{definition}

A consequence of our definition of an appropriate class leads to the following.
\begin{proposition}
	Let $\mathcal{C}$ be a nonempty class of $\kappa$-metric twisted spectral triples. Let $\mathcal{T}$ be a $\mathcal{C}$-appropriate class of tunnels. The following assertions hold:
	\begin{enumerate}
		\item if $(\A_1,\Hilbert_1,\Dirac_1,\rho_1)$ and $(\A_2,\Hilbert_2,\Dirac_2,\rho_2)$ in $\mathcal{C}$ are unitarily equivalent,
		\begin{equation*}
			\spectralpropinquity{\mathcal{T}}((\A_1,\Hilbert_1,\Dirac_1,\rho_1),(\A_2,\Hilbert_2,\Dirac_2,\rho_2)) = 0 \text;
		\end{equation*}
		\item if $(\A_1,\Hilbert_1,\Dirac_1,\rho_1)$ and $(\A_2,\Hilbert_2,\Dirac_2,\rho_2)$ are in $\mathcal{C}$ then
		\begin{equation*}
			\spectralpropinquity{\mathcal{T}}((\A_2,\Hilbert_2,\Dirac_2,\rho_2),(\A_1,\Hilbert_1,\Dirac_1,\rho_1)) = \spectralpropinquity{\mathcal{T}}((\A_1,\Hilbert_1,\Dirac_1,\rho_1),(\A_2,\Hilbert_2,\Dirac_2,\rho_2)) \text.
		\end{equation*}
	\end{enumerate}
\end{proposition}

\begin{proof}
	These are all immediate consequences of the definition of an appropriate class of tunnels, and the observations that:
	\begin{enumerate}
		\item $\tunnelmagnitude{\tau;\varepsilon} = 0$ for all $\varepsilon > 0$ and for all trivial tunnel $\tau$ in $\mathcal{T}$,
		\item $\tunnelmagnitude{\tau^{-1};\varepsilon} = \tunnelmagnitude{\tau;\varepsilon}$ for all $\tau \in \mathcal{T}$, as is immediate by Expression \eqref{magnitude-eq}.
	\end{enumerate}
	Our proof follows.
\end{proof}

We check that, in particular, for given $0<\kappa\leq1$, $C\geq 1$, $D\geq 0$, $G\geq\kappa^{-1}+2$ and $H\geq 2$, the class of all $(C,D,G,H)$-relaxed metrical C*-correspondences is appropriate over the class of all $\kappa$-metric twisted spectral triples.

\begin{proposition}
	If $0<\kappa\leq1$ and $G\geq\kappa^{-1}+2$, then for all $C\geq 1$, $D\geq 0$ and $H\geq 2$, the class of all $(C,D,G,H)$ metrical tunnels is appropriate for the class of all $\kappa$-metric twisted spectral triples.
\end{proposition}

\begin{proof}
	All conditions for our class of tunnels to be appropriate are trivially met, except the fourth condition on the composition, up to an arbitrary $\varepsilon >0$, of two tunnels. However, most of this condition is also known to hold, as our starting point for this proof is \cite[Theorem 3.21]{Latremoliere18g}. Please note that we will use the notation of \cite{Latremoliere18g}, so that the magnitude defined there for tunnels is denoted by $\tunnelmodmagnitude{\cdot}{\cdot}$; this quantity is given by removing from Expression \eqref{magnitude-eq} the separation term involving the twists.
	
		Indeed, by that result, given three $\kappa$-metric twisted spectral triples $(\A_j,\Hilbert_j,\Dirac_j,\rho_j)$ for $j \in \{1,2,3\}$, a $(C,D,G,H)$-tunnel $\tau_1$ from $(\A_1,\Hilbert_1,\Dirac_1,\rho_1)$ to $(\A_2,\Hilbert_2,\Dirac_2,\rho_2)$ and a tunnel $\tau_2$ from $(\A_2,\Hilbert_2,\Dirac_2,\rho_2)$ to $(\A_3,\Hilbert_3,\Dirac_3,\rho_3)$, and
	\begin{enumerate}
	 	\item any $\varepsilon > 0$,
	 	\item any $\varepsilon_1>  0$ and $\varepsilon_2>  0$ such that $\tunnelmodmagnitude{\tau_1}{\varepsilon_1}\leq\varepsilon_1$ and $\tunnelmodmagnitude{\tau_2}{\varepsilon_2}\leq\varepsilon_2$,
	 \end{enumerate}
	 then there exists a tunnel $\tau_3$ from $\dom{\tau_1}$ to $\codom{\tau_2}$ such that 
	 \begin{equation*}
\tunnelmodmagnitude{\tau_3}{\varepsilon_1 + \varepsilon_2 + \varepsilon}\leq\varepsilon_1 + \varepsilon_2 + \varepsilon \text.
	\end{equation*}
		The only caveat, once more, is that $\tau_3 = (\mathds{P}_3, \Pi_{3,1},\Pi_{3,2})$ where $\mathds{P}_3$ is now a relaxed metrical C*-correspondence (note that, if in fact, $\tau_1$ and $\tau_2$ are built using actual metrical C*-correspondences, then $\mathds{P}_3$ would belong to the same class, but this is not important here).

	We now make one simple modification to our tunnel $\tau_3$. To this end, we need a few notation from \cite{Latremoliere18g}, though we will try to only involve what we need. First, write $\tau_j = (\mathds{P}_j, \Pi_{j,1}, \Pi_{j,2})$ for $j \in \{1,2,3\}$, where
	\begin{equation*}
		\mathds{P}_j = \left( \module{P}_j, \TDN_j, \D_j, \TLip_j, \alg{E}_j, \Lip[Q]_j \right) 
	\end{equation*}
	and $\Pi_{a,b} = (\Psi_{a,b}, \pi_{a,b}, \theta_{a,b})$ for any indices $a$ in $\{1,2,3\}$ and $b\in\{1,2\}$.
	
	The key aspect of \cite[Theorem 3.21]{Latremoliere18g} we need here is simply that $\D_3 = \D_1 \oplus \D_2$. With this in mind, we now construct a new tunnel $\tau_4$ by setting:
	\begin{equation*}
		\tau_4 \coloneqq \left( \mathds{P}_4, \Pi_{3,1}, \Pi_{3,2} \right)
	\end{equation*}
	with
	\begin{equation*}
		\mathds{P}_4 = \left( \module{P}_3, \TDN_3, \D_3, \TLip_4, \alg{E}_3, \Lip[Q]_3\right)
	\end{equation*} 
	where, for all $(d_1,d_2) \in \dom{\TLip_3}$, we set:
	\begin{equation*}
		\TLip_4(d_1,d_2) \coloneqq \max\left\{ \TLip_3(d_1,d_2), \frac{1}{\varepsilon} \opnorm{\rho_2\circ\pi_{1,2}(d_1) - \rho_2\circ\pi_{2,1}(d_2)}{}{\Hilbert_2} \right\}\text.
	\end{equation*}
		In plain language, the only change from $\tau_3$ to $\tau_4$ is the replacement of $\TLip_3$ as constructed in \cite[Theorem 3.21]{Latremoliere18g} with $\TLip_4$.
	
		Since $\TLip_4\geq\TLip_3$ on their common domain $\dom{\TLip_3}$, we conclude immediately from the proof of \cite[Theorem 3.21]{Latremoliere18g} that indeed, $\tau_4$ is again a $(C,D,G,H)$-tunnel from $\dom{\tau_1}$ to $\codom{\tau_2}$, and all the computations there remain valid, again with the exception that $\mathds{P}_4$ is now a relaxed C*-metrical correspondence.
	
	It is now an easy, though notationally heavy computation to check that:
	\begin{equation*}
		\tunnelsep{(\rho_1(\pi_{3,1}(d))^\ast)_{d \in B(\tau_4)}}{(\rho_3(\pi_{3,2}(d))^\ast)_{d \in B(\tau_4)}} \leq GH(\varepsilon_1 + \varepsilon_2 + \varepsilon) \text.
	\end{equation*}
	using the fact that the Hausdorff distance is indeed a metric. We will leave this check to the reader.
	
		The tunnel $\tau_4$ thus has the desired properties.
\end{proof}

\begin{notation}
	For any $0<\kappa\leq1$, we define $\spectralpropinquity{\kappa}$ as $\spectralpropinquity{\mathcal{T}_\kappa}$, as in Definition \ref{spectral-prop-def}, where $\mathcal{T}_\kappa$ is the class of \emph{all} $(1,0,\kappa^{-1}+2,2)$ tunnels. This is the natural choice when working with $\kappa$-metric twisted spectral triples. However, when working with untwisted spectral triples, it is natural to restrict ourselves to the class of all $(1,0,1,0,2,2)$ tunnels, i.e. to metrical C*-correspondences built with Leibniz {\qcms s}. 
\end{notation}

We now reconcile our new definition of the spectral propinquity with our previous one, when working with (untwisted) metric spectral  triples. To this end, we establish two lemmas.

\begin{lemma}\label{zero-twist-lemma}
Let $\mathds{M}_j \coloneqq (\module{M}_j, \CDN_j, \A_j, \Lip_j, \B_j, \Lip[S]_j)$ be a relaxed $(C,D,G,H)$-metrical C*-correspondence for each $j \in \{1,2\}$. Let $(\mathds{P},(\Psi_1,\pi_1,\theta_1), (\Psi_2,\pi_2,\theta_2))$ be a tunnel from $\mathds{M}_1$ to $\mathds{M}_2$, with $\mathds{P} = (\module{P},\TDN,\D,\TLip,\alg{E},\Lip[Q])$ a relaxed $(C,D,G,H)$-metrical C*-correspondence. If $\omega,\zeta\in\dom{\TDN}$, then:
	\begin{equation*}
		\left|\theta_1(\inner{\zeta}{\omega}{\module{P}}) - \theta_2(\inner{\zeta}{\omega}{\module{P}})\right| \leq H \TDN(\zeta) \TDN(\omega) \tunnelextent{\tau} \text.
	\end{equation*}
\end{lemma}

\begin{proof}
Since $\theta_1 : \alg{E} \to \C$ and $\theta_2 : \alg{E}\to \C$ are, respectively, the unique lifts of the unique state of $\C$ as states of $\alg{E}$ in our tunnel $\tau$, we have by definition of the extent of the tunnel $(\alg{E},\Lip[Q],\theta_1,\theta_2)$ from $\C$ to $\C$ that
\begin{equation}\label{zero-twist-eq-5}
	\Kantorovich{\Lip[Q]}(\theta_{1}, \theta_{2}) = \tunnelextent{(\alg{E},\Lip[Q],\theta_{1},\theta_{2})} \leq \tunnelextent{\tau} \text.
\end{equation}
In particular, by the inner Leibniz inequality, for all $\omega , \zeta \in \dom{\TDN}$, the following holds:
\begin{equation*}
	\Lip[Q](\inner{\zeta}{\omega}{\module{P}}) \leq H \TDN(\zeta)\TDN(\omega) \leq H  \text.
\end{equation*}	
	Therefore, applying Expression \eqref{zero-twist-eq-5}, we get by homogeneity that
	\begin{equation*}
		\left| \theta_{2}(\inner{\zeta}{\omega}{\module{P}}) - \theta_{1}(\inner{\zeta}{\omega}{\module{P}})\right| \leq H  \TDN(\zeta)\TDN(\omega) \tunnelextent{\tau} \text,
	\end{equation*}
	as desired.
\end{proof}

\begin{lemma}\label{sanity-lemma}
Let $\mathds{M}_j \coloneqq (\module{M}_j, \CDN_j, \A_j, \Lip_j, \B_j, \Lip[S]_j)$ be a relaxed $(C,D,G,H)$-metrical C*-correspondence for each $j \in \{1,2\}$. Let $(\mathds{P},(\Psi_1,\pi_1,\theta_1), (\Psi_2,\pi_2,\theta_2))$ be a tunnel from $\mathds{M}_1$ to $\mathds{M}_2$, with $\mathds{P} = (\module{P},\TDN,\D,\TLip,\alg{E},\Lip[Q])$ a relaxed $(C,D,G,H)$-metrical C*-correspondence. We then have:
	\begin{equation*}
		\tunnelsep{(\pi_1(d))_{\substack{d\in\dom{\TLip} \\ \TLip(d)\leq 1}}}{(\pi_2(d))_{\substack{d \in \dom{\TLip} \\ \TLip(d)\leq 1}}} \leq G H \tunnelextent{\tau} \text.
	\end{equation*}
\end{lemma}

\begin{proof}
	Let $\xi \in \dom{\Dirac_1}$ with $\CDN_1(\xi) = 1$. Since $\Psi_1$ is a quantum isometry, there exists $\chi \in \dom{\TDN}$ with $\TDN(\chi) = 1$ and $\Psi_1(\chi) = \xi$. Let $\eta = \Psi_2(\chi)$ (in the terminology of \cite{Latremoliere18d}, $\eta \in \targetsettunnel{\tau}{\xi}{1}$).

	Let $d\in \dom{\TDN}$ with $\max\{\norm{d}{\D}, \TLip(d)\} = 1$. Let $a \coloneqq \pi_1(d)$ and $b = \pi_2(d)$, so that $\max\{ \norm{a}{\A_1}, \Lip_1(a) \} \leq 1$ and $\max\{ \norm{b}{\A_1}, \Lip_1(b) \} \leq 1$. By \cite[Proposition 4.7]{Latremoliere18d}, there exists $\zeta \in \dom{\TDN}$ with $\TDN(\zeta) = G$, such that $\Psi_1(\zeta) = a\xi$ and $\Psi_2(\zeta) = b\eta$. In particular, for all $\omega \in \dom{\TDN}$ with $\TDN(\omega)\leq 1$,
	\begin{align*}
		\left| \inner{a\xi}{\Psi_1(\omega)}{\Hilbert_1} - \inner{b\eta}{\Psi_2(\omega)}{\Hilbert_2}\right| 
		&=\left|\inner{\Psi_1(\zeta)}{\Psi_1(\omega)}{\Hilbert_1} - \inner{\Psi_2(\zeta)}{\Psi_2(\omega)}{\Hilbert_2}\right|\\
		&=\left|\theta_1(\inner{\zeta}{\omega}{\module{P}}) - \theta_2(\inner{\zeta}{\omega}{\module{P}})\right|\\
		&\leq G H \tunnelextent{\tau} \text{ by Lemma \ref{zero-twist-lemma}.}
	\end{align*}
	In summary, we have shown that:
	\begin{equation*}
		\sup_{\substack{\xi \in \dom{\CDN_1} \\ \CDN_1(\xi)\leq 1}} \inf_{\substack{\eta \in \dom{\CDN_2} \\ \CDN_2(\eta) \leq 1}} \underbracket{\sup_{\substack{d \in B(\tau) \\ \omega \in \dom{\TDN} \\ \TDN(\omega)\leq 1}} \left|\inner{\pi_1(d)\xi}{\Psi_1(\omega)}{\Hilbert_1}-\inner{\pi_2(d)\eta}{\Psi_2(\omega)}{\Hilbert_2}\right|}_{=\Kantorovich{\TDN,B(\tau)}(\pi_1(d),\pi_2(d))} \leq G H \tunnelextent{\tau} \text.
	\end{equation*}
	
	By switching the roles of $\xi$ and $\eta$ above, and by definition of the Hausdorff distance, we therefore conclude:
	\begin{equation*}
		\tunnelsep{(\pi_1(d))_{d \in B(\tau)}}{(\pi_2(d))_{d \in B(\tau)}} \leq G H \tunnelextent{\tau}
	\end{equation*}
	as desired.
\end{proof}

\begin{corollary}
	For two metric spectral triples $(\A_1,\Hilbert_1,\Dirac_1)$ and $(\A_2,\Hilbert_2,\Dirac_2)$, we have:
	\begin{equation*}
		\underbracket{\spectralpropinquity{}((\A_1,\Hilbert_1,\Dirac_1),(\A_2,\Hilbert_2,\Dirac_2))}_{\text{usual spectral propinquity \cite{Latremoliere18g,Latremoliere22}}} = \underbracket{\spectralpropinquity{\mathcal{L}}((\A_1,\Hilbert_1,\Dirac_1,\mathrm{id}_{\A_1}), (\A_2,\Hilbert_2,\Dirac_2,\mathrm{id}_{\A_2}))}_{\text{twisted spectral propinquity with the identity as twist from Def. \ref{spectral-prop-def}}} \text,
	\end{equation*}
	where $\mathcal{L}$ is the class of all $(1,0,1,0,2,2)$ tunnels.
\end{corollary}

\begin{proof}
	This follows immediately from Lemma \ref{sanity-lemma} and the definition of the magnitude in Expression \eqref{magnitude-eq} compared to \cite[Definition ]{Latremoliere22}.
\end{proof}

Now, we apply the method introduced in \cite{Latremoliere18d} and \cite{Latremoliere18g} to obtain the following theorem validating our choice of a definition for the propinquity between twisted metric spectral triples.
\begin{theorem}
Let $0<\kappa\leq1$, let $\mathcal{C}$ be a nonempty class of $\kappa$-metric twisted spectral triples, and let $\mathcal{T}$ be a $\mathcal{C}$-appropriate class of tunnels. For all $(\A_1,\Hilbert_1,\Dirac_1,\rho_1)$ and $(\A_2,\Hilbert_2,\Dirac_2,\rho_2)$ in $\mathcal{C}$,
	\begin{equation*}
		\spectralpropinquity{\mathcal{T}}((\A_1, \Hilbert_1, \Dirac_1, \rho_1),(\A_2, \Hilbert_2, \Dirac_2 ,\rho_2)) = 0
	\end{equation*}
	if, and only if $(\A_1,\Hilbert_1,\Dirac_1,\rho_1)$ and $(\A_2,\Hilbert_2,\Dirac_2,\rho_2)$ are unitarily equivalent twisted spectral triples.
\end{theorem}

\begin{proof}
	If $(\A_1,\Hilbert_1,\Dirac_1,\rho_1)$ and $(\A_2,\Hilbert_2,\Dirac_2,\rho_2)$ are unitarily equivalent, then by definition of an appropriate class of tunnels, there exists a trivial tunnel between them in our class, whose extent is $0$. We now turn to the converse, and assume
	\begin{equation*}
	\spectralpropinquity{\mathcal{T}}((\A_1, \Hilbert_1, \Dirac_1, \rho_1),(\A_2, \Hilbert_2, \Dirac_2 ,\rho_2)) = 0\text.
	\end{equation*}

	In this proof, for each $j \in \{1,2\}$, we write:
	\begin{equation*}
		\mcc{\A_j}{\Hilbert_j}{\Dirac_j,\rho_j} = (\Hilbert_j, \CDN_j, \A_j, \Lip_j, \C, 0)
	\end{equation*}
	so $\CDN_j$ is the graph norm of $\Dirac_j$ and $\Lip_j(a) = \max\{\opnorm{[\Dirac_j,a]_{\rho_j}}{}{\Hilbert_j},\kappa\opnorm{\rho_j(a)-a}{}{\Hilbert_j}\}$ for all $a$ where this is well-defined.

	By \cite{Latremoliere18g}, we immediately gather the following (since the spectral propinquity for twisted spectral triples dominates the metrical propinquity without a twist): there exists a unitary $U : \Hilbert_1 \to \Hilbert_2$, a continuous self-adjoint linear isomorphism $\pi : (\A_1,\Lip_1) \to (\A_2,\Lip_2)$ of norm $1$, and a sequence $(\tau_n)_{n\in\N}$ of metrical tunnels from $(\A_1,\Hilbert_1,\Dirac_1,\rho_1)$ to $(\A_2,\Hilbert_2,\Dirac_2,\rho_2)$, 
	\begin{equation*}
		\tau_n = \left( \mathds{P}_n, (\Psi_{1,n}, \pi_{1,n}, \theta_{1,n}),  (\Psi_{2,n}, \pi_{2,n}, \theta_{2,n}) \right)
	\end{equation*}
	where
	\begin{equation*}
	\mathds{P}_n \coloneqq \left( \module{P}_n, \TDN_n, \D_n, \TLip_n, \alg{E}_n, \Lip[Q]_n\right) \text,
	\end{equation*}
	with the following properties. 
	\begin{enumerate}
		\item there is a sequence $(\varepsilon_n)_{n\in\N}$ converging to $0$ such that $\tunnelmagnitude{\tau_n;\varepsilon_n}\leq\varepsilon_n$ for all $n\in\N$,
		\item for all $a \in \dom{\Lip_1}$ and for all $l \geq \Lip_1(a)$, we have 
		\begin{equation*}
			\lim_{n\to\infty}\Haus{\A_2}(\targetsettunnel{\tau_n}{a}{l}, \{ \pi(a) \} ) = 0\text,
		\end{equation*}
		where
		\begin{equation*}
			\targetsettunnel{\tau_n}{a}{l} = \left\{ \pi_{2,n}(d) : d \in \dom{\TLip_n}, \TLip_n(d) \leq l, \norm{d}{\D_n}\leq\norm{a}{\A_1}+l\tunnelextent{\tau_n}, \pi_{1,n}(d) = a \right\}\text;
		\end{equation*}
		\item  for all $\xi \in \dom{\CDN_2}$ and for all $l \geq \CDN_2(\xi)$, we have 
		\begin{equation*}
			\lim_{n\to\infty}\Haus{\Hilbert_1}(\targetsettunnel{\tau_n^{-1}}{\xi}{l}, \{ U^\ast \xi \} ) = 0\text,
		\end{equation*}
		where
		\begin{equation*}
		\targetsettunnel{\tau_n^{-1}}{\xi}{l} = \left\{ \Psi_{1,n}(\omega) : \omega \in \dom{\TDN_n}, \TDN_n(\omega) \leq l, \Psi_{2,n}(\omega) = \xi \right\}\text;
		\end{equation*}
		\item  for all $\xi \in \dom{\CDN_1}$ and for all $l \geq \CDN_1(\xi)$, we have 
		\begin{equation*}
			\lim_{n\to\infty}\Haus{\Hilbert_2}(\targetsettunnel{\tau_n}{\xi}{l}, \{ U \xi \} ) = 0\text,
		\end{equation*}
		where
		\begin{equation*}
		\targetsettunnel{\tau_n}{\xi}{l} = \left\{ \Psi_{2,n}(\omega) : \omega \in \dom{\TDN_n}, \TDN_n(\omega) \leq l, \Psi_{1,n}(\omega) = \xi \right\}\text;
		\end{equation*}
		\item $U\Dirac_1 U^\ast = \Dirac_2$ (including equality of the domains, as in particular $U\dom{\Dirac_1} = \dom{\Dirac_2}$),
		\item $\pi(a) = U a U^\ast$ for all $a \in \A_1$,
		\item $\pi(\dom{\Lip_1}) = \dom{\Lip_2}$ and $\Lip_2\circ\pi = \Lip_1$ on $\dom{\Lip_1}$.
	\end{enumerate}
	
	\emph{There is one apparent change in the above: the map $\pi$ is not stated to be a quantum isometry in (1), because it is no longer multiplicative, from the proof in \cite{Latremoliere13,Latremoliere13b}.} This is because, as we work with \emph{twisted} spectral triples, which in turn means working with\emph{relaxed} metrical C*-correspondences, we do \emph{not assume} that $\Lip_1$ and $\Lip_2$ have any Leibniz property, so \cite[Proposition 5.12]{Latremoliere13} and \cite[Claim 5.17, 5.18]{Latremoliere13} do not apply. However, \cite[Claim 5.19]{Latremoliere13} does, by replacing *-homomorphism with self-adjoint linear map of norm $1$. As it is based on this construction, \cite{Latremoliere13b,Latremoliere14,Latremoliere15} apply unchanged except that distance zero gives us a self-adjoint linear isomorphism of norm $1$. 
	
	From there, everything else in the proof of \cite{Latremoliere18d} remains unchanged, leading to the conclusion stated above. To be clear, first note that we do require that relaxed metrical C*-correspondences still include a {\qcms} as a C*-algebra of scalars --- not a relaxed one. This is to ensure that distance zero does imply *-isomorphism between these C*-algebras of scalars for our Hilbert modules in our C*-correspondences. Here, the metrical C*-correspondences for our twisted spectral triples use $\C$ as scalar C*-algebra, and thus, isomorphism here simply means the identity. In particular, $(\pi,U)$ is indeed a modular morphism, not as a corollary of the (missing) Leibniz property of $\Lip_1$ and $\Lip_2$, but from the \emph{modular Leibniz property} used in \cite[Proposition 4.7]{Latremoliere18d}, then used in \cite[Theorem 4.9]{Latremoliere18d}. Since, for all $a \in \A_1$ and $\xi \in \Hilbert_1$, we must have $\pi(a) U \xi = U(a\xi)$, we conclude that $\pi(a)=UaU^\ast$. As $U$ is a unitary, we indeed recover that $\pi$ is multiplicative. Since $\pi$ is self-adjoint and onto $\A_2$, the map $\pi$ is a *-isomorphism from $\A_1$ onto $\A_2$. Since it is an isometry for the L-seminorms, it is also a full quantum isometry!
	
	Thus Assertions (1), (2) and (3) of Definition \ref{unitary-eq-def} hold. It remains to prove that the twists are intertwined.

	Let $a\in\dom{\Lip_1}$ with $\max\{\norm{a}{\A_1},\Lip_1(a)\}\leq1$, and choose
	\begin{equation*}
		b_n\in\targetsettunnel{\tau_n}{a}{1}.
	\end{equation*}
	Then $b_n\to\pi(a)$. Set
	\begin{equation*}
		s_n\coloneqq
		\tunnelsep{(\rho_1(\pi_{1,n}(d))^\ast)_{d\in B(\tau_n)}}
		{(\rho_2(\pi_{2,n}(d))^\ast)_{d\in B(\tau_n)}}.
	\end{equation*}
	By the definition of the magnitude, $s_n\leq GH\varepsilon_n$, and hence $s_n\to0$.

	Fix $\xi,\psi\in\dom{\Dirac_2}$ with $\CDN_2(\xi),\CDN_2(\psi)\leq1$. Since $\Psi_{1,n}$ is a quantum isometry, choose $\omega_n\in\dom{\TDN_n}$ such that
	\begin{equation*}
		\TDN_n(\omega_n)\leq1
		\qquad\text{and}\qquad
		\Psi_{1,n}(\omega_n)=U^\ast\psi.
	\end{equation*}
	The definition of $s_n$ provides $\chi_n\in\dom{\TDN_n}$, with $\TDN_n(\chi_n)\leq1$, such that, if $y_n\coloneqq\Psi_{2,n}(\chi_n)$, then
	\begin{multline}\label{zero-twist-sep-eq}
		\sup_{\substack{d\in B(\tau_n)\\ \vartheta\in\dom{\TDN_n}\\\TDN_n(\vartheta)\leq1}}
		\Big|
		\inner{\rho_1(\pi_{1,n}(d))^\ast U^\ast\psi}{\Psi_{1,n}(\vartheta)}{\Hilbert_1}\\
		-\inner{\rho_2(\pi_{2,n}(d))^\ast y_n}{\Psi_{2,n}(\vartheta)}{\Hilbert_2}
		\Big|
		\leq s_n.
	\end{multline}
	Write $t_n\coloneqq\Psi_{2,n}(\omega_n)$. Then $t_n\to\psi$ by the target-set convergence above. Applying \eqref{zero-twist-sep-eq} with $d=1$, first to $\vartheta=\chi_n$ and then to $\vartheta=\omega_n$, and applying Lemma \ref{zero-twist-lemma} to $\omega_n$ and $\chi_n$, gives
	\begin{equation*}
		\norm{y_n-t_n}{\Hilbert_2}^2
		\leq2s_n+2H\tunnelextent{\tau_n}.
	\end{equation*}
	Consequently, $y_n\to\psi$.

	Similarly, choose $\vartheta_n\in\dom{\TDN_n}$ with
	\begin{equation*}
		\TDN_n(\vartheta_n)\leq1,
		\qquad
		\Psi_{1,n}(\vartheta_n)=U^\ast\xi,
	\end{equation*}
	and set $z_n\coloneqq\Psi_{2,n}(\vartheta_n)$. The target-set convergence gives $z_n\to\xi$.

	For all sufficiently large $n$, we may choose $d_n\in\dom{\TLip_n}$ such that
	\begin{equation*}
		\pi_{1,n}(d_n)=a,\qquad
		\pi_{2,n}(d_n)=b_n,\qquad
		\TLip_n(d_n)\leq1,\qquad
		\norm{d_n}{\D_n}\leq\norm{a}{\A_1}+\tunnelextent{\tau_n}\leq2.
	\end{equation*}
	Thus $d_n/2\in B(\tau_n)$. Using $d_n/2$ and $\vartheta_n$ in \eqref{zero-twist-sep-eq}, then taking complex conjugates, yields
	\begin{equation*}
		\left|
		\inner{\rho_1(a)U^\ast\xi}{U^\ast\psi}{\Hilbert_1}
		-\inner{\rho_2(b_n)z_n}{y_n}{\Hilbert_2}
		\right|
		\leq2s_n.
	\end{equation*}
	Moreover, $\Lip_2(b_n)\leq1$, $\norm{b_n}{\A_2}\leq2$ for all sufficiently large $n$, and hence
	\begin{equation*}
		\opnorm{\rho_2(b_n)}{}{\Hilbert_2}
		\leq\norm{b_n}{\A_2}+\kappa^{-1}\Lip_2(b_n)
		\leq2+\kappa^{-1}\leq G.
	\end{equation*}
	Since $y_n\to\psi$, $z_n\to\xi$, and $s_n\to0$, we conclude that
	\begin{equation*}
		\inner{\rho_2(b_n)\xi}{\psi}{\Hilbert_2}
		\longrightarrow
		\inner{U\rho_1(a)U^\ast\xi}{\psi}{\Hilbert_2}.
	\end{equation*}
	By density and the uniform bound above, $\rho_2(b_n)$ converges to $U\rho_1(a)U^\ast$ in the weak operator topology. Since $b_n\to\pi(a)$ in norm and since $\opnorm{[\Dirac_2,b_n]_{\rho_2}}{}{\Hilbert_2} \leq \Lip_2(b_n)\leq 1$, by assumption on $\rho_2$:
	\begin{equation*}
		\rho_2(\pi(a))=U\rho_1(a)U^\ast.
	\end{equation*}
	By homogeneity, this identity holds for every $a\in\dom{\Lip_1}$, as desired.
\end{proof}

In summary of this section, we record that:
\begin{corollary}
	If $0<\kappa\leq1$, if $\mathcal{C}$ is a nonempty class of $\kappa$-metric twisted spectral triples, and if $\mathcal{T}$ is a $\mathcal{C}$-appropriate class of tunnels, then $\spectralpropinquity{\mathcal{T}}$ is a metric on $\mathcal{C}$ up to unitary equivalence. This applies, in particular, to $\spectralpropinquity{\kappa}$. 
\end{corollary}

Therefore, the spectral propinquity can be easily extended to account for twisted spectral triples, and remains a distance up to unitary equivalence. Note that indeed, this revised definition agrees with the usual spectral propinquity for metric spectral triples by Lemma (\ref{sanity-lemma}).

We also note that the results in \cite{Latremoliere22} all apply immediately since none of these results depend on the specific choice of the formula relating the Dirac operator and the L-seminorm (a remark we made in the introduction of \cite{Latremoliere22}).

\section{Convergence for the spectral propinquity}

We now can conclude our proof that our twisted, fuzzy tori provide finite dimensional approximations of the flat spectral triple of quantum tori, including the classical torus, for the spectral propinquity, thus concluding our paper. We use the notations of Sections 2 and 3.

We begin with a few technical results in the same vein as previously. 
\begin{notation}
	For any $k \in \Pbar^d$, and for any $f \in L^1(\T^d,\mu_k)$, the following defines a continuous linear operator on $\ell^2(\Z^d_k)$:
	\begin{equation*}
		u_k^f  \colon \xi \in \ell^2(\Z^d_k) \mapsto \int_{\U^d_k} f(z) u_k^z \xi \, d\mu_k(z) \text,
	\end{equation*}
	with $\opnorm{u_k^f}{}{\ell^2(\Z^d_k)} \leq \norm{f}{L^1(\U^d_k,\mu)}$.
\end{notation}

In the next two lemmas, for any $k \in \Pbar^d$, we set for all $\xi \in \dom{\Dirac_k}$:
\begin{equation*}
	\CDN_k(\xi) = \norm{\xi}{\Hilbert_k} + \norm{\Dirac_k \xi}{\Hilbert_k} \text.
\end{equation*}

\begin{lemma}\label{Hilbert-compression-lemma}
	If $f \colon \T^d \to [0,\infty)$ with $\int_{\T^d} f \, d\mu= 1$, and if the Fourier transform of $f$ is supported on a finite, nonempty subset $S$ of $\Z^d$, then for all $k \in \Pbar^d$, and for all $\xi \in \dom{\Dirac_k}$, we have $(u_k^f\otimes\mathrm{id}_{\Spinor_d}) \xi \in \ell^2(\Z^d_k | S)\otimes\Spinor_d$, while:
	\begin{equation*}
		\CDN_{k}((u_k^f\otimes \mathrm{id}_{\Spinor_d}) \xi) \leq \CDN_k(\xi)\text,
	\end{equation*}
	and
	\begin{equation*}
		\norm{\xi - (u_k^f\otimes \mathrm{id}_{\Spinor_d}) \xi}{\Hilbert_k} \leq \CDN_k(\xi) \int_{\U_k^d} f(z)\operatorname{slen}(z) \, d\mu(z) \text.
	\end{equation*}
\end{lemma}

\begin{proof}
	For all $z \in \U^d_k$, the unitary $(u_k^z\otimes \mathrm{id}_{\Spinor_d})$ commutes with $\Dirac_k$, so it acts as an isometry of $\CDN_k$. It follows that $(u_k^f\otimes \mathrm{id}_{\Spinor_d})$ is a weak contraction of $\CDN_k$ (note that $\CDN_k$ is lower semicontinuous). Moreover, for all $\xi \in \dom{\CDN_{k}}$, a coordinate telescoping argument gives
	\begin{align*}
		\norm{\xi - (u_k^f\otimes\mathrm{id}_{\Spinor_d})\xi}{\Hilbert_k}
		&\leq \int_{\U_k^d} f(z) \norm{\xi - (u_k^z\otimes\mathrm{id}_{\Spinor_d})\xi}{\Hilbert_k} \, d\mu_k(z) \\
		&\leq \CDN_k(\xi) \int_{\U_k^d} f(z) \operatorname{slen}(z) \, d\mu_k(z) \text.
	\end{align*}
	Indeed, in a continuous coordinate $\|m_j\xi\|\leq\|\Dirac_k^0\xi\|$, while in a finite coordinate $\|(u_k^j-1)\xi\|=(2\pi/k_j)\|\nabla_k^j\xi\|\leq(2\pi/k_j)\|\Dirac_k^0\xi\|$ by Expression \eqref{dirac-square-eq}. Also $\|\Dirac_k^0\xi\|\leq\CDN_k(\xi)$ by our choice of $\varepsilon_k$.
\end{proof}

We now conclude with our main result.
Set
\begin{equation*}
	\kappa_d\coloneqq \frac{1}{16\pi d} \text,
	\qquad
	G_d\coloneqq 2+16\pi d.
\end{equation*}
\begin{theorem}
For all normalized $2$-cocycles $\sigma$ of $\Z^d$, and if $\mathcal{T}$ is the class of all $(4d,0,1,0,G_d,2)$-tunnels between $\kappa_d$-metric twisted spectral triples, then
	\begin{equation*}
		\lim_{\substack{(k,\varsigma) \to (\infty,\sigma) \\ (k,\varsigma) \in \Xi}} \spectralpropinquity{\mathcal{T}}((\A_{k,\varsigma},\Hilbert_{k},\Dirac_{k},\rho_{k,\varsigma}), (\A_{\infty,\sigma},\Hilbert_\infty,\Dirac_\infty,\mathrm{id})) = 0 \text.
	\end{equation*}
\end{theorem}

\begin{proof}
We continue the argument which we started for the proof of Theorem \ref{prop-cv-thm}, and in particular, we import all its notation henceforth. We are thus given a normalized $2$-cocycle $\sigma$ of $\Z^d$. Let $\varepsilon > 0$. Without loss of generality, we assume henceforth that $\varepsilon\in(0,1)$.

By our previous argument, shrinking the neighborhood and choosing the compression kernel for the tolerance $\varepsilon^2$, there exists a neighborhood $\Omega_\varepsilon$ of $(\infty,\sigma)$ in $\Xi$, a finite subset $S_\varepsilon\subseteq\Z^d$, and a function $f : \T^d\to[0,\infty)$ with the properties listed at the end of Section 3. In particular, using Lemma \ref{Hilbert-compression-lemma}, we observe that for all $(k,\varsigma)\in \Omega_\varepsilon$, and for all $\eta \in \dom{\Dirac_k}$, the following inequality holds:
\begin{equation*}
	\norm{v_k^f \eta - \eta}{\Hilbert_k} \leq \varepsilon^2 \CDN_k(\eta)\text,
\end{equation*}
where $v_k^z = u_k^z\otimes\mathrm{id}_{\Spinor_d}$ for all $z\in \U^d_k$, and $v_k^f = u_k^f\otimes\mathrm{id}_{\Spinor_d}$.

Note that by construction, $v_k^f \eta \in \Hilbert_k(S_\varepsilon)$ where $\Hilbert_k(S_\varepsilon) \coloneqq \ell^2(\Z^d_k|S_\varepsilon)\otimes\Spinor_d$.  Therefore, $T_\varepsilon v_k^f \eta = v_k^f \eta$, where $T_\varepsilon = R_{k,\sigma,\varepsilon} \otimes \mathrm{id}_{\Spinor_d}$.

We note that the forward symbols $\nabla_k^j$ and their adjoints converge on every fixed Fourier set to multiplication by $im_j$ and $-im_j$, respectively. Therefore, writing $\theta_k \coloneqq \vartheta_k\otimes\mathrm{id}_{\Spinor_d}$, the operators $\theta_k \Dirac_k\theta_k^\ast$ converge in the strong operator topology to $\Dirac_\infty$ on the Hilbert space $\Hilbert_\infty(S_\varepsilon)$ by construction; since $S_\varepsilon$ is finite, the space $\Hilbert_k(S_\varepsilon)$ is finite dimensional, so $\theta_k \Dirac_k \theta_k^\ast$ converges in norm to $\Dirac_\infty$ as $k$ goes to $\infty$. We let $\Omega$ be a neighborhood of $(\infty,\sigma)$ contained in $\Omega_\varepsilon$ such that, for all $(k,\varsigma) \in \Omega$:
\begin{equation*}
	\opnorm{\theta_k\Dirac_k\theta_k^\ast - \Dirac_\infty}{}{\Hilbert_\infty(S_\varepsilon)} < \varepsilon^2 \text.
\end{equation*}
Henceforth, up to reducing $\Omega$, we also assume that if $(k,\varsigma) \in \Omega$ with $k = (k_1,\ldots,k_d)$, then
\begin{equation*}
	16\pi\sum_{k_j<\infty}\frac1{k_j}\leq\varepsilon \text.
\end{equation*}

Let $(k,\varsigma) \in \Omega$.  For all $\xi \in \Hilbert_k$ and $\eta \in \Hilbert_\infty$, we define:
\begin{equation*}
	\TDN_k(\xi,\eta) \coloneqq\max\left\{\CDN_k(\xi), \CDN_\infty(\eta), \frac{1}{3\varepsilon}\norm{\theta_k(\xi) - \eta}{\Hilbert_\infty} \right\} \text,
\end{equation*}
which is finite exactly over the dense subspace $\dom{\TDN_k} \coloneqq \dom{\Dirac_k}\oplus\dom{\Dirac_\infty}$ of $\Hilbert_k\oplus\Hilbert_\infty$. It is easy to check that $\{ (\xi,\eta) \in \dom{\TDN_k} : \TDN_k(\xi,\eta)\leq 1\}$ is compact (as a closed subset of the product of the closed unit ball of the graph norms of $\Dirac_k$ and $\Dirac_\infty$, both of which are compact).

Let $j_k : (\xi,\eta) \in \Hilbert_k\oplus\Hilbert_\infty \mapsto \xi \in \Hilbert_k$ and $j_\infty :(\xi,\eta) \in \Hilbert_k\oplus\Hilbert_\infty \mapsto \eta \in \Hilbert_\infty$. It is an immediate exercise to check that $(j_k,j_{k,\varsigma})$ and $(j_\infty,j_{\infty,\sigma})$ are surjective module maps. We wish to prove they are quantum isometries. 

Let $\xi \in \dom{\Dirac_k}$ with $\CDN_k(\xi)\leq 1$. We then have
\begin{align*}
	\left|\CDN_{k,\varsigma}(v_k^f \xi) - \CDN_{\infty,\sigma}(\theta_k(v_k^f\xi))\right|
	&\leq \left|\norm{v_k^f\xi}{\Hilbert_k(S_\varepsilon)} - \norm{\theta_k(v_k^f\xi)}{\Hilbert_\infty(S_\varepsilon)}\right| \\
	&\quad + \left|\norm{\Dirac_k v_k^f\xi}{\Hilbert_k(S_\varepsilon)} -  \norm{\Dirac_\infty\theta_k(v_k^f\xi)}{\Hilbert_\infty(S_\varepsilon)}\right| \\
	&= 0 + \left|\norm{\theta_k \Dirac_k \theta_k^\ast \theta_k v_k^f\xi}{\Hilbert_\infty(S_\varepsilon)} -  \norm{\Dirac_\infty\theta_k(v_k^f\xi)}{\Hilbert_\infty(S_\varepsilon)}\right| \text{ as $\theta_k$ isometric,}\\
	&\leq 0 + \opnorm{\theta_k \Dirac_k \theta_k^\ast-\Dirac_\infty}{}{\Hilbert_\infty(S_\varepsilon)} \norm{v_k^f \xi}{\Hilbert_k} \\
	&\leq \varepsilon^2 \leq \varepsilon \text.
\end{align*}
So $\CDN_\infty(\theta_k(v_k^f\xi)) \leq 1+\varepsilon^2$. Therefore, letting $\eta \coloneqq \frac{1}{1+\varepsilon^2}\theta_k v_k^f \xi$, we have $\CDN_\infty(\eta)\leq1$ and
\begin{align*}
	\norm{\theta_k\xi-\eta}{\Hilbert_\infty}
	&\leq\norm{\xi-v_k^f\xi}{\Hilbert_k}
	+\frac{\varepsilon^2}{1+\varepsilon^2}\norm{v_k^f\xi}{\Hilbert_k}\\
	&\leq2\varepsilon^2\leq3\varepsilon.
\end{align*}
Therefore, we conclude that:
\begin{equation*}
	\TDN_{k}(\xi,\eta) \leq 1 \text.
\end{equation*}
This proves that $j_k$ is indeed a quantum isometry. The same reasoning applies to show that $j_\infty$ is also a quantum isometry.

Now, we check the necessary Leibniz inequalities for our prospective $D$-norm $\TDN_k$. Let $a \in \A_{k,\varsigma}$, $b \in \A_{\infty,\sigma}$, $\xi \in \dom{\Dirac_k}$ and $\eta \in \dom{\Dirac_\infty}$. We compute:
\begin{align*}
	\norm{\theta_k(a\xi) - b\eta}{\Hilbert_\infty}
	&=\norm{\pi_{k,\varsigma}(a) \theta_k(\xi) - b\eta}{\Hilbert_\infty} \\
	&\leq \norm{\pi_{k,\varsigma}(a) (\theta_k(\xi) - \eta)}{\Hilbert_\infty} + \norm{(\pi_{k,\varsigma}(a) - b)\eta}{\Hilbert_\infty} \\
	&\leq \norm{a}{\A_{k,\varsigma}}\norm{\theta_k(\xi)-\eta}{\Hilbert_\infty} + \opnorm{(\pi_{k,\varsigma}(a) - b)}{}{\Hilbert_\infty} \norm{\eta-v_\infty^f\eta}{\Hilbert_\infty} \\
	&\quad + \norm{(\pi_{k,\varsigma}(a) - b)v_\infty^f \eta}{\Hilbert_\infty} \\
	&\leq 3\varepsilon\norm{a}{\A_{k,\varsigma}} \TDN_k (\xi,\eta) + \varepsilon^2(\norm{a}{\A_{k,\varsigma}} + \norm{b}{\A_{\infty,\sigma}})\CDN_\infty(\eta) \\
	&\quad + \norm{(\pi_{k,\varsigma}(a) - b) R_k v_\infty^f \eta}{\Hilbert_\infty} \\
	&\leq \varepsilon(4 \norm{a}{\A_{k,\varsigma}} + \norm{b}{\A_{\infty,\sigma}}) \TDN_k (\xi,\eta) + \norm{[b,R_k] v_\infty^f \eta}{\Hilbert_\infty} \\
	&\quad + \norm{(\pi_{k,\varsigma}(a) R_k  - R_k b ) v_\infty^f\eta}{\Hilbert_\infty} \\
	&\leq \varepsilon(4 \norm{a}{\A_{k,\varsigma}} + \norm{b}{\A_{\infty,\sigma}}) \TDN_k(\xi,\eta) + \varepsilon \Lip_{\infty,\sigma}(b) \norm{\eta}{\Hilbert_\infty} \\
	&\quad + \opnorm{\pi_{k,\varsigma}(a) R_k  - R_k b}{}{\Hilbert_\infty} \norm{\eta}{\Hilbert_\infty} \\
	&\leq \varepsilon(4 \norm{a}{\A_{k,\varsigma}} + \norm{b}{\A_{\infty,\sigma}}) \TDN_k(\xi,\eta) + \varepsilon \Lip_{\infty,\sigma}(b) \norm{\eta}{\Hilbert_\infty} + \varepsilon\TLip_{k,\varsigma} (a,b) \norm{\eta}{\Hilbert_\infty} \\
	&\leq \varepsilon (4 \norm{a}{\A_{k,\varsigma}} + \norm{b}{\A_{\infty,\sigma}} + \Lip_{\infty,\sigma}(b) + \TLip_{k,\varsigma}(a,b)) \TDN_k(\xi,\eta) \\
	&\leq \varepsilon (5 \norm{(a,b)}{\D_{k,\varsigma}} + 2 \TLip_{k,\varsigma}(a,b)) \TDN_k(\xi,\eta) \text. 
\end{align*}
The two graph-norm components satisfy the same bound with constant $G_d$ by the general modular Leibniz estimate above, while the link component is bounded by the preceding computation. Therefore,
\begin{align*}
	\TDN_k(a\xi,b\eta) 
	&\leq G_d\max\{ \norm{(a,b)}{\D_{k,\varsigma}} , \TLip_{k,\varsigma}(a,b)\} \TDN_k(\xi,\eta) \text,
\end{align*}

We regard $\Hilbert_k\oplus\Hilbert_\infty$ as a $\C^2$ Hilbert module, with inner product defined for all $(\xi,\eta), (\xi',\eta') \in \Hilbert_k\oplus\Hilbert_\infty$ by
\begin{equation*}
	\inner{(\xi,\eta)}{(\xi',\eta')}{\C^2} = \left(\inner{\xi}{\xi'}{\Hilbert_k}, \inner{\eta}{\eta'}{\Hilbert_\infty}\right) \text.
\end{equation*}
The C*-algebra $\D_{k,\varsigma} \coloneqq \A_{k,\varsigma} \oplus \A_{\infty,\sigma}$ acts by adjoinable operators on $\Hilbert_k\oplus\Hilbert_\infty$ by setting, of course, for all $(a,b) \in \D_{k,\varsigma}$ and $(\xi,\eta)\in\Hilbert_k\oplus\Hilbert_\infty$:
\begin{equation*}
	(a,b)(\xi,\eta) = (a\xi,b\eta) \text.
\end{equation*}
Therefore, $(\Hilbert_k \oplus \Hilbert_\infty, \D_{k,\varsigma}, \C^2)$ is a C*-correspondence. Since $\TDN_k$ makes the $\D_{k,\varsigma}$ right module $\Hilbert_k \oplus \Hilbert_\infty$ into a metrized module, we turn it into a metrical C*-correspondence by introducing an L-seminorm on $\C^2$. Thus, for all $(z,w) \in \C^2$, let
\begin{equation*}
	\Lip[Q](z,w) \coloneqq \frac{1}{3\varepsilon} |  z - w | \text.
\end{equation*}
It is obvious that $\Lip[Q]$ is an L-seminorm on $\C^2$ (it is in fact the Lipschitz seminorm for the metric on the two point set placing each point at distance $3\varepsilon$ from the other). In particular, $(\C^2,\Lip[Q])$ is a tunnel from $(\C,0)$ to $(\C,0)$ of extent $3\varepsilon$.

We note that since $\theta_k$ is an isometry, we also obtain the Leibniz relation:
\begin{align*}
	\Lip[Q](\inner{\xi}{\xi'}{\Hilbert_k}, \inner{\eta}{\eta'}{\Hilbert_\infty})
	&= \frac{1}{3\varepsilon}\left| \inner{\xi}{\xi'}{\Hilbert_k} - \inner{\eta}{\eta'}{\Hilbert_\infty}\right| \\
	&= \frac{1}{3\varepsilon}\left| \inner{\theta_k(\xi)}{\theta_k(\xi')}{\Hilbert_\infty} - \inner{\eta}{\eta'}{\Hilbert_\infty}\right| \\
	&\leq \frac{1}{3\varepsilon}\left| \inner{\theta_k(\xi)-\eta}{\theta_k(\xi')}{\Hilbert_\infty} + \inner{\eta}{\theta_k(\xi')-\eta'}{\Hilbert_\infty}\right| \\
	&\leq \frac{1}{3\varepsilon}\left(\norm{\theta_k(\xi)-\eta}{\Hilbert_\infty} \norm{\xi'}{\Hilbert_k} + \norm{\eta}{\Hilbert_\infty}\norm{\theta_k(\xi') - \eta'}{\Hilbert_\infty}\right) \\
	&\leq \norm{\xi'}{\Hilbert_k} \TDN_k(\xi,\eta) + \norm{\eta}{\Hilbert_\infty} \TDN_k(\xi',\eta') \\
	&\leq \TDN_k(\xi',\eta') \TDN_k(\xi,\eta) + \TDN_k(\xi,\eta) \TDN_k(\xi',\eta') \\
	&= 2 \TDN_k(\xi,\eta) \TDN_k(\xi',\eta') \text.
\end{align*}

Altogether, we conclude that $(\Hilbert_k\oplus\Hilbert_\infty, \TDN_k, \D_{k,\varsigma}, \TLip_{k,\varsigma})$ is a $(4d,0,1,0,G_d,2)$ metrical C*-correspondence. It is then obvious that, setting $p_1:(z,w) \in \C^2\mapsto z$ and $p_2:(z,w)\in\C^2\mapsto w$,
\begin{equation*}
	\Upsilon_{k,\varsigma,\varepsilon}\coloneqq\left[ \left(\Hilbert_k\oplus\Hilbert_\infty, \TDN_k, \D_{k,\varsigma}, \TLip_{k,\varsigma}, \C^2, \Lip[Q]\right), (j_k,j_{k,\varsigma},p_1), (j_\infty,j_{\infty,\sigma},p_2) \right]
\end{equation*}
is a $(4d,0,1,0,G_d,2)$ metrical tunnel (not relaxed!), whose extent is the maximum of the extent of $(\D_{k,\varsigma}, \TLip_{k,\varsigma}, j_{k,\varsigma}, j_{\infty,\sigma})$ (which is at most $\varepsilon$) and the extent of $(\C^2,\Lip[Q],p_1,p_2)$, which is $3\varepsilon$. Hence, the extent of our metrical tunnel is at most $3\varepsilon$:
\begin{equation}\label{upsilon-eq}
	\tunnelextent{\Upsilon_{k,\varsigma,\varepsilon}} \leq 3\varepsilon \text.
\end{equation}

We now compute the separation for the unitary actions generated by $\Dirac_k$ and $\Dirac_\infty$. Since $\exp(it\Dirac_k)$ is unitary, the operator $\theta_k \exp(it\Dirac_k)\theta_k^\ast$ has norm at most $1$. Thus, for all $t \in [0,\infty)$, $\xi \in \Hilbert_k$ and $\eta \in \Hilbert_\infty$: 
\begin{align*}
	\norm{\theta_k(\exp(it \Dirac_k)\xi) - \exp(it\Dirac_\infty)\eta}{\Hilbert_\infty}
	&\leq \norm{\theta_k\exp(it \Dirac_k)\theta_k^\ast\theta_k \xi - \exp(it\Dirac_\infty)\eta}{\Hilbert_\infty} \\
	&\leq \underbracket{\opnorm{\theta_k\exp(it\Dirac_k)\theta_k^\ast}{}{\Hilbert_\infty}}_{=1} \norm{\theta_k(\xi)-\eta}{\Hilbert_\infty} \\
	&\quad + \norm{\left(\theta_k\exp(it\Dirac_k)\theta_k^\ast-\exp(it\Dirac_\infty)\right)\eta}{\Hilbert_\infty} \\
	&\leq 3\varepsilon \TDN_k(\xi,\eta) + \norm{\left(\theta_k\exp(it\Dirac_k)\theta_k^\ast-\exp(it\Dirac_\infty)\right)\eta}{\Hilbert_\infty}  \\ 
	&\leq 3\varepsilon\TDN_k(\xi,\eta) + \norm{\left(\theta_k\exp(it\Dirac_k)\theta_k^\ast-\exp(it\Dirac_\infty)\right)(\eta-v_\infty^f\eta)}{\Hilbert_\infty} \\
	&\quad + \norm{\left(\theta_k\exp(it\Dirac_k)\theta_k^\ast-\exp(it\Dirac_\infty)\right) v_\infty^f \eta}{\Hilbert_\infty} \\
	&\leq 3\varepsilon\TDN_k(\xi,\eta) +  2 \varepsilon^2 \CDN_\infty(\eta)  \\
	&\quad + \norm{\left(\theta_k\exp(it\Dirac_k)\theta_k^\ast-\exp(it\Dirac_\infty)\right) v_\infty^f \eta}{\Hilbert_\infty} \\
	&\leq 5 \varepsilon\TDN_k(\xi,\eta) + \norm{\left(\theta_k\exp(it\Dirac_k)\theta_k^\ast-\exp(it\Dirac_\infty)\right) v_\infty^f \eta}{\Hilbert_\infty} \text.
\end{align*}

 Therefore, using \cite[Ch. 9, Eqn (2.3), p. 497]{Kato}, and noting that $\theta_k^\ast \theta_k = 1$ (i.e. $\theta_k$ is an isometry):
\begin{align*}
	&\norm{\Big(\theta_k\exp(it\Dirac_k)\theta_k^\ast-\exp(it\Dirac_\infty)\Big) v_\infty^f \eta}{\Hilbert_\infty} \\
	&\quad = \norm{(\exp(it \theta_k \Dirac_k\theta_k^\ast) - \exp(it\Dirac_\infty))v_\infty^f\eta}{\Hilbert_\infty} \\
	&\quad \leq \int_0^t \norm{\left(\exp(i (t-s) \theta_k\Dirac_k\theta_k^\ast) (\theta_k\Dirac_k\theta_k^\ast-\Dirac_\infty)\exp(is\Dirac_\infty)\right)  v_\infty^f\eta}{\Hilbert_\infty}\, ds\\
	&\quad \leq \int_0^t \opnorm{\theta_k\Dirac_k\theta_k^\ast-\Dirac_\infty}{}{\Hilbert_\infty(S_\varepsilon)} \norm{v_\infty^f\eta}{\Hilbert_\infty} \, ds \\
	&\quad \leq t \varepsilon^2\CDN_\infty(\eta) \text.
\end{align*} 
Therefore, for all $t \in \left[0,\frac{1}{\varepsilon}\right]$, we conclude:
\begin{equation*}
	\norm{\theta_k(\exp(it \Dirac_k)\xi) - \exp(it\Dirac_\infty)\eta}{\Hilbert_\infty} \leq 6 \varepsilon \TDN_k(\xi,\eta) \text.
\end{equation*}

From this, we easily conclude that:
\begin{equation}\label{sep-eq}
	\tunnelsep{(\exp(it\Dirac_k))_{t \in \left[0,\frac{1}{6\varepsilon}\right]}}{(\exp(it\Dirac_\infty))_{t \in \left[0,\frac{1}{6\varepsilon}\right]}} \leq 6\varepsilon\text.
\end{equation}

We are left with the computation of the separation for our twists. This is rather simple here, since our twists are actually bounded. The two-sided Riesz estimate in Corollary \ref{car-twist-estimate-cor}, and our choice of $\Omega$, instead give
\begin{equation*}
	\opnorm{\rho_{k,\varsigma}(a)-a}{}{\Hilbert_k}
	\leq\varepsilon\Lip_{k,\varsigma}(a).
\end{equation*}
Therefore, if $d_0\coloneqq(a,b) \in \dom{\TLip_{k,\varsigma}}$ with $\max\{\norm{d_0}{\D_{k,\varsigma}},\TLip_{k,\varsigma}(d_0)\} \leq 1$, then $\Lip_{k,\varsigma}(a) \leq 1$ and $\Lip_{\infty,\sigma}(b)\leq 1$. For all vectors and test vectors occurring in the definition of the separation with $D$-norm at most $1$, and noting that $\rho_{\infty,\sigma}$ is the identity, we obtain:
\begin{align*}
	\Big| &\inner{\rho_{k,\varsigma}(a)\xi}{j_k(\omega)}{\Hilbert_k} - \inner{b\eta}{j_\infty(\omega)}{\Hilbert_\infty}\Big|\\
	&\leq \opnorm{\rho_{k,\varsigma}(a) - a}{}{\Hilbert_k} \\
	&\quad + \Big| \inner{a\xi}{j_k(\omega)}{\Hilbert_k} - \inner{b\eta}{j_\infty(\omega)}{\Hilbert_\infty}\Big| \\
	&\leq \varepsilon + 2G_d\tunnelextent{\Upsilon_{k,\varsigma,\varepsilon}}
	\leq (1+6G_d)\varepsilon \text,
\end{align*} 
using Lemma \ref{sanity-lemma}. The same estimate holds in the reverse direction, and taking adjoints and complex conjugates gives the corresponding estimate for the adjoint families used in the definition of the magnitude.

Therefore for all $(k,\varsigma) \in \Omega$, we conclude that:
\begin{equation}\label{sep-eq-2}
 \tunnelsep{(\rho_{k,\varsigma}(a)^\ast)_{(a,b) \in B(\Upsilon_{k,\varsigma,\varepsilon})}}{(b^\ast)_{(a,b) \in B(\Upsilon_{k,\varsigma,\varepsilon})}} \leq (1+6G_d)\varepsilon \text.
\end{equation}

We have therefore proven that for all $(k,\varsigma) \in \Omega$, by Expressions \eqref{upsilon-eq}, \eqref{sep-eq} and \eqref{sep-eq-2},
\begin{equation*}
	\tunnelmagnitude{\Upsilon_{k,\varsigma,\varepsilon}; (1+6G_d)\varepsilon} \leq (1+6G_d)\varepsilon 
\end{equation*}
and thus
\begin{equation*}
	\spectralpropinquity{\mathcal{T}}((\A_{k,\varsigma},\Hilbert_k,\Dirac_k,\rho_{k,\varsigma}),(\A_{\infty, \sigma}, \Hilbert_\infty, \Dirac_\infty,\mathrm{id})) \leq (1+6G_d)\varepsilon \text.
\end{equation*}
This completes our proof.
\end{proof}

\bibliographystyle{amsplain} \bibliography{../thesis.bib}
  
\vfill

\end{document}